%% file: disjoint_data.tex
\documentclass[11pt,a4paper,reqno]{amsart}

\input{preamble}
\selectfont
\title[A sharp spectral threshold for disjoint data]{Uniqueness and nonuniqueness in the\\ disjoint data inverse problem:\\ a sharp spectral threshold}
\author{Lauri Ylinen}
\address{Department of Mathematics and Statistics, University of Jyväskylä,\newline Finland}
\email{Lauri.A.Ylinen@jyu.fi}
\thanks{The author was supported by the Research Council of Finland (Centre of Excellence of Inverse Modelling and Imaging \#353092, Flagship of Advanced Mathematics for Sensing, Imaging and Modelling \#359208).}
\date{}
\hypersetup{pdftitle={Uniqueness and nonuniqueness in the disjoint data inverse problem: a sharp spectral threshold}, pdfkeywords={Inverse problem, wave equation, disjoint data, source-to-solution operator, spectral inequality, nonuniqueness, counterexample}}

\begin{document}

\vspace*{-15pt}
\input{abstract.tex}

\maketitle
\vspace*{-15pt}
\begingroup
\footnotesize
\tableofcontents
\endgroup
\input{intro.tex}
\input{problem.tex}
\input{uniqueness.tex}
\input{nonuniqueness.tex}
\appendix
\crefalias{section}{appendix}
\input{appendix.tex}

\makeatletter
\begingroup
\let\@tocwrite\@gobbletwo
\section*{Use of generative AI}\label{sec:AI-disclosure}

During the development of this work, the author used Claude Code, ChatGPT, and OpenAI Codex, primarily with the Claude Opus 5 and GPT-5.6-sol models, to generate and refine ideas and to test conjectures and candidate arguments. In particular, GPT-5.6-sol produced the initial nonuniqueness construction. The simpler and more general construction presented in \cref{sec:nonuniqueness} was developed in subsequent work by the author with these systems. The systems also assisted with grammatical and stylistic revisions and with writing the TikZ code for the figures. The author takes full responsibility for the mathematical content and final form of the paper.

\endgroup
\makeatother

\begingroup
\setlength{\emergencystretch}{2em}
\printbibliography
\endgroup

\end{document}

%% file: preamble.tex
\usepackage[T1]{fontenc}
\usepackage{mathtools}
\usepackage{amssymb,amsthm}
\usepackage{stmaryrd}       
\usepackage{enumitem}
\usepackage[backend=biber, style=alphabetic, maxbibnames=99]{biblatex}
\AtEveryBibitem{%
  \clearfield{issn}%
  \iffieldundef{doi}{}{\clearfield{url}\clearfield{urldate}}%
}
\makeatletter
\let\aliascnt@PackageWarning\PackageWarning
\def\PackageWarning#1#2{}
\usepackage{aliascnt}       
\let\PackageWarning\aliascnt@PackageWarning
\makeatother
\usepackage[%
pdfusetitle,
colorlinks=true,
citecolor=blue,
linkcolor=blue,
urlcolor=blue]{hyperref}
\usepackage{cleveref}
\usepackage{microtype}
\usepackage{tikz}
\usepackage{interval}
\intervalconfig{soft open fences}

\numberwithin{equation}{section}

\theoremstyle{plain}
\newtheorem{theorem}{Theorem}[section]
\newaliascnt{proposition}{theorem}
\newtheorem{proposition}[proposition]{Proposition}
\aliascntresetthe{proposition}
\newaliascnt{lemma}{theorem}
\newtheorem{lemma}[lemma]{Lemma}
\aliascntresetthe{lemma}
\newaliascnt{corollary}{theorem}
\newtheorem{corollary}[corollary]{Corollary}
\aliascntresetthe{corollary}

\theoremstyle{definition}
\newaliascnt{definition}{theorem}
\newtheorem{definition}[definition]{Definition}
\aliascntresetthe{definition}
\newaliascnt{example}{theorem}

\aliascntresetthe{example}
\newaliascnt{assumption}{theorem}
\newtheorem{assumption}[assumption]{Assumption}
\aliascntresetthe{assumption}

\theoremstyle{remark}
\newaliascnt{remark}{theorem}
\newtheorem{remark}[remark]{Remark}
\aliascntresetthe{remark}

\newenvironment{stepproof}[1][\proofname]{%
  \begin{proof}[#1]%
  }{
\end{proof}}

\crefname{equation}{}{}
\crefname{figure}{figure}{figures}
\crefname{example}{example}{examples}
\Crefname{example}{Example}{Examples}
\crefname{assumption}{assumption}{assumptions}
\Crefname{assumption}{Assumption}{Assumptions}
\AtBeginDocument{}

\newlist{datacond}{enumerate}{1}
\setlist[datacond]{label=(D\arabic*), ref=(D\arabic*), start=0,
itemsep=\medskipamount, leftmargin=*, labelindent=\parindent}
\crefname{datacondi}{}{}
\Crefname{datacondi}{}{}

\newlist{param}{enumerate}{1}
\setlist[param]{label=(P\arabic*), ref=(P\arabic*)}
\crefname{parami}{}{}
\Crefname{parami}{}{}

\newlist{step}{enumerate}{1}
\setlist[step]{label=(\roman*), ref=(\roman*)}
\crefname{stepi}{step}{steps}
\Crefname{stepi}{Step}{Steps}

\newlist{cond}{enumerate}{1}
\setlist[cond]{label=\arabic*., ref=\arabic*}
\crefname{condi}{condition}{conditions}
\Crefname{condi}{Condition}{Conditions}

\newcommand{\R}{\mathbb{R}}
\newcommand{\C}{\mathbb{C}}
\newcommand{\Z}{\mathbb{Z}}
\newcommand{\bbS}{\mathbb{S}}

\newcommand{\calB}{\mathcal{B}}
\newcommand{\calD}{\mathcal{D}}
\newcommand{\calH}{\mathcal{H}}
\newcommand{\calJ}{\mathcal{J}}
\newcommand{\calL}{\mathcal{L}}
\newcommand{\calR}{\mathcal{R}}
\newcommand{\tcalR}{\widetilde{\calR}}

\DeclareMathOperator{\id}{id}
\DeclareMathOperator{\re}{Re}
\DeclareMathOperator{\linspan}{span}
\DeclareMathOperator{\spec}{spec}
\DeclareMathOperator{\supp}{supp}
\DeclareMathOperator{\interior}{int}
\DeclareMathOperator{\sv}{sv}
\DeclareMathOperator{\trace}{tr}
\DeclareMathOperator{\off}{off}
\DeclareMathOperator{\diag}{diag}
\DeclareMathOperator{\diam}{diam}
\DeclareMathOperator{\dist}{dist}
\DeclareMathOperator{\pr}{pr}
\DeclareMathOperator{\res}{res}
\DeclareMathOperator{\ext}{ext}

\DeclarePairedDelimiter{\abs}{\lvert}{\rvert}
\DeclarePairedDelimiter{\norm}{\lVert}{\rVert}
\DeclarePairedDelimiter{\pointnorm}{\lvert}{\rvert}
\DeclarePairedDelimiterX{\inner}[2]{\langle}{\rangle}{#1,#2}

\newcommand{\wt}[1]{\widetilde{#1}}
\newcommand{\M}{M}
\newcommand{\tM}{\widetilde{M}}
\newcommand{\tg}{\widetilde{g}}
\newcommand{\tN}{\wt{N}}
\newcommand{\tx}{\wt{x}}
\newcommand{\ty}{\wt{y}}
\newcommand{\tq}{\wt{q}}

\newcommand{\src}{S}
\newcommand{\rcv}{R}
\newcommand{\tsrc}{\widetilde{\src}}
\newcommand{\trcv}{\widetilde{\rcv}}

\newcommand{\El}{E_\ell}
\newcommand{\tEl}{\widetilde{E}_\ell}

\newcommand{\Grcv}{G^\rcv}
\newcommand{\Glrcv}{G^\rcv_\ell}
\newcommand{\Glsrc}{G^\src_\ell}

\DeclarePairedDelimiterXPP{\Cc}[1]{C_c^\infty}{(}{)}{}{#1}
\DeclarePairedDelimiterXPP{\Lloc}[1]{L^1_{\mathrm{loc}}}{(}{)}{}{#1}

\newcommand{\LM}{L^2(\M)}
\newcommand{\Lsrc}{L^2(\src)}
\newcommand{\Lrcv}{L^2(\rcv)}
\DeclarePairedDelimiterXPP{\normM}[1]{}{\lVert}{\rVert}{_{L^2(\M)}}{#1}
\DeclarePairedDelimiterXPP{\normsrc}[1]{}{\lVert}{\rVert}{_{L^2(\src)}}{#1}
\DeclarePairedDelimiterXPP{\normrcv}[1]{}{\lVert}{\rVert}{_{L^2(\rcv)}}{#1}
\newcommand{\innerM}[2]{\inner{#1}{#2}_{\LM}}
\newcommand{\innerrcv}[2]{\inner{#1}{#2}_{\Lrcv}}

\newcommand{\tLM}{L^2(\tM)}
\newcommand{\tLsrc}{L^2(\tsrc)}
\newcommand{\tLrcv}{L^2(\trcv)}
\DeclarePairedDelimiterXPP{\tnormM}[1]{}{\lVert}{\rVert}{_{L^2(\tM)}}{#1}
\DeclarePairedDelimiterXPP{\tnormsrc}[1]{}{\lVert}{\rVert}{_{L^2(\tsrc)}}{#1}
\DeclarePairedDelimiterXPP{\tnormrcv}[1]{}{\lVert}{\rVert}{_{L^2(\trcv)}}{#1}
\newcommand{\tinnerM}[2]{\inner{#1}{#2}_{\tLM}}
\newcommand{\tinnerrcv}[2]{\inner{#1}{#2}_{\tLrcv}}

\newcommand{\HWM}{\calH_W(\M)}
\newcommand{\tHWM}{\calH_{\widetilde{W}}(\tM)}
\DeclarePairedDelimiterXPP{\normW}[1]{}{\lVert}{\rVert}{_{W}}{#1}
\DeclarePairedDelimiterXPP{\tnormW}[1]{}{\lVert}{\rVert}{_{\widetilde{W}}}{#1}

\newcommand{\cosprop}{C^{\mathsf{W}}}
\newcommand{\sinprop}{S^{\mathsf{W}}}

\newcommand{\LSRw}{\Lambda_{\rcv\shortleftarrow\src}}
\newcommand{\tLSRw}{\widetilde{\Lambda}_{\trcv\shortleftarrow\tsrc}}

\newcommand{\RSR}{R_{\rcv\shortleftarrow\src}}
\newcommand{\tRSR}{\widetilde{R}_{\trcv\shortleftarrow\tsrc}}

\newcommand{\PlRS}[1][\ell]{P_{#1,\rcv\shortleftarrow\src}}
\newcommand{\tPlRS}[1][\ell]{\widetilde{P}_{#1,\trcv\shortleftarrow\tsrc}}
\newcommand{\PlSR}[1][\ell]{P_{#1,\src\shortleftarrow\rcv}}
\newcommand{\tPlSR}[1][\ell]{\widetilde{P}_{#1,\tsrc\shortleftarrow\trcv}}

\newcommand{\word}[1]{\mathtt{#1}}
\newcommand{\bzero}{\word{0}}
\newcommand{\bone}{\word{1}}

\DeclarePairedDelimiterXPP{\Leven}[1]{L^2_{\mathrm{even}}}{(}{)}{}{#1}
\DeclarePairedDelimiterXPP{\Lodd}[1]{L^2_{\mathrm{odd}}}{(}{)}{}{#1}

\newcommand{\portsC}{;{\scriptstyle C_\rcv\shortleftarrow C_\src}}
\newcommand{\portsI}{;{\scriptstyle I_\rcv\shortleftarrow I_\src}}
\newcommand{\Rport}[2][\mu]{R_{#1}^{#2}}
\newcommand{\tRport}[2][\mu]{\widetilde{R}_{#1}^{#2}}

\newcommand{\Tmat}[2]{T_{#2\shortleftarrow #1}}
\newcommand{\tTmat}[2]{\widetilde{T}_{#2\shortleftarrow #1}}

%% file: abstract.tex
\begin{abstract}
  \noindent
  We study the inverse problem of determining a connected, closed Riemannian manifold of dimension at least two from wave observations. The sources are supported in an open set $\src$, and the resulting waves are observed on an open receiver set $\rcv$, which may be at a positive distance from $\src$. We prove uniqueness up to isometry within the class of manifolds that satisfy a subcritical spectral lower bound: $\normrcv{\phi|_\rcv} \ge ce^{-C\lambda^\theta}\normM{\phi}$ for every eigenfunction $\phi$, where $\lambda$ is the corresponding eigenvalue, $\theta\in\rinterval{0}{1/2}$, and the constants $c,C>0$ do not depend on $\phi$. If a bound of this type holds also on $\src$, then uniqueness holds within the class of all connected, closed manifolds. By the spectral inequality, the bound always holds with $\theta=1/2$. Uniqueness nevertheless fails at this endpoint: in every dimension $n\ge 2$, we construct a pair of nonisometric connected, closed Riemannian $n$-manifolds that produce identical wave observations.
\end{abstract}

%% file: intro.tex
\section{Introduction}\label{sec:introduction}

We study the recovery of a Riemannian manifold from waves generated in one open set and observed in another. Let $(\M,g)$ be a connected, closed Riemannian manifold of dimension at least two. Throughout, all manifolds are smooth, and \emph{closed} means compact without boundary. For a source $s\in\Cc[\big]{\ointerval{0}{\infty}\times\M}$, consider the inhomogeneous wave equation
\begin{equation}\label{eq:wave-equation}
  \left\{
    \begin{aligned}
      (\partial_t^2-\Delta_g)\,v &= s && \text{in } \ointerval{0}{\infty}\times\M,\\
      v(0,\cdot) &= 0 && \text{in } \M,\\
      \partial_t v(0,\cdot) &= 0 && \text{in } \M.
    \end{aligned}
    \right.
  \end{equation}
  Here $\Delta_g$ denotes the Laplace--Beltrami operator of $(\M,g)$, and $v=v^s\in C^\infty\bigl(\ointerval{0}{\infty}\times\M\bigr)$ is the unique solution of \cref{eq:wave-equation}.

  Fix nonempty open sets $\src,\rcv\subset\M$. We generate waves using sources supported in $\src$ and observe the resulting solutions on $\rcv$. The possibility that $\src$ and $\rcv$ are disjoint motivates the term \emph{disjoint data}. This measurement configuration is encoded by the \emph{local source-to-solution operator}
  \begin{equation}\label{eq:LSRw}
    \LSRw\colon \Cc[\big]{\ointerval{0}{\infty}\times\src} \to C^\infty\bigl(\ointerval{0}{\infty}\times\rcv\bigr);
    \qquad
    \LSRw s := v^s|_{\ointerval{0}{\infty}\times\rcv}.
  \end{equation}

  The \emph{data} of the inverse problem consist of the local source-to-solution operator $\LSRw$ and the sets $\src$ and $\rcv$ together with their induced smooth structures. We study conditions under which they determine $(\M,g)$ uniquely up to a Riemannian isometry. Note that $\M\setminus(\src\cup\rcv)$ is not given even as a set; the missing portion of the smooth manifold, together with the Riemannian metric there, must be recovered from the data.

  In many measurement settings, waves cannot be generated and observed in the same place: in exploration seismology, for instance, the sources are often explosive, so measurements near the shot points are impractical. Access to the medium may also be restricted to a few separated locations, with the region one wishes to image lying between them.

  Allowing the sets $\src$ and $\rcv$ to be at a positive distance from each other changes the nature of the problem. If $\src=\rcv$, then uniqueness up to isometry is known to hold~\cite{MR3826551}. (If $\src\cap\rcv\neq\emptyset$, then replacing $\src$ and $\rcv$ with their intersection reduces the problem to $\src=\rcv$.) The difficulty in the disjoint data problem is that the global norm of a wave is not directly available from the data. By contrast, when $\src=\rcv$, the Blagoveščenskiĭ identity~\cite{MR282060,MR298226} (see also~\cite[Theorem~12]{MR3826551}) computes $\normM{v^s(t)}$ from the source $s$ and the observation $v^s|_{\ointerval{0}{2t}\times\rcv}$. Knowing the norm makes it possible to create sequences of waves that in the limit localize to arbitrary locations on the manifold, which is a key step in the proof of uniqueness. For disjoint data, this localization argument cannot be applied in the same way.

  \subsection{Description of main results}

  The statements in this subsection are streamlined descriptions of our results; precise formulations appear in the results cited.

  We prove both uniqueness and nonuniqueness results, and we identify the sharp threshold between them within the scale of power-type eigenfunction bounds. More precisely, consider the lower bound
  \begin{equation}\label{eq:intro-subcritical-bound}
    \normrcv{\phi|_\rcv} \ge ce^{-C\lambda^\theta}\normM{\phi},
  \end{equation}
  where $\phi$ is an eigenfunction of $-\Delta_g$ with eigenvalue $\lambda$, and $c,C>0$ and $\theta\in\interval{0}{1/2}$ are independent of $\phi$. Throughout, we use the convention $0^0=1$. By the spectral inequality~\cite[Theorem~14.6]{MR1743865}, every connected, closed manifold and every nonempty open set $\rcv$ admit constants $c,C>0$ for which \cref{eq:intro-subcritical-bound} holds with $\theta=1/2$. Thus the endpoint estimate imposes no additional restriction on the manifold, whereas our uniqueness result assumes the strictly stronger estimate with $\theta<1/2$.

  Our main results are as follows:
  \begin{enumerate}
    \item Within the class of connected, closed Riemannian manifolds satisfying \cref{eq:intro-subcritical-bound} for some $c,C>0$ and $\theta<1/2$, the data determine the manifold up to isometry (\cref{item:subcritical-two-manifold} of \cref{prop:subcritical-result}).

    \item If \cref{eq:intro-subcritical-bound} and the analogous inequality with $\rcv$ replaced by $\src$ both hold for some $c,C>0$ and $\theta<1/2$, then the data determine $(\M,g)$ up to isometry within the class of all connected, closed manifolds (\cref{item:subcritical-single-manifold} of \cref{prop:subcritical-result}).

    \item For every integer $n\ge 2$, there exist two connected, closed Riemannian $n$-manifolds that have the same data but are not related by a Riemannian isometry (\cref{prop:nonuniqueness}).
  \end{enumerate}

  Together, these results show that the restriction $\theta<1/2$ is sharp within the power scale. Beyond the power scale, however, the uniqueness results extend even closer to the threshold $\theta=1/2$: for each of the two uniqueness statements above, the corresponding conclusion remains valid (\cref{prop:main-beta-bounded} combined with \cref{prop:Ingham}) if, for $\lambda>1$, the expression $\lambda^\theta$ in the eigenfunction bounds~\labelcref{eq:intro-subcritical-bound} is replaced by
  \begin{equation*}
    \sqrt{\lambda}\,a(\sqrt{\lambda}),
  \end{equation*}
  where $a\colon\rinterval{1}{\infty}\to\ointerval{0}{\infty}$ is nonincreasing and satisfies
  \begin{equation}\label{eq:intro-Ingham}
    \int_1^\infty \frac{a(\omega)}{\omega}\,d\omega < \infty.
  \end{equation}

  The obstruction to extending the uniqueness proof to the universal endpoint $\theta=1/2$ comes from an uncertainty principle in Fourier analysis: a function and its Fourier transform cannot both be very small. More precisely, \cref{eq:intro-Ingham} is Ingham's condition~\cite{MR1574706}: it holds precisely when there exist nonzero compactly supported functions $\rho$ on $\R$ whose Fourier transforms satisfy bounds of the form $\abs{\widehat{\rho}(\omega)}\le C e^{-\abs{\omega}a(\abs{\omega})}$ for $\abs{\omega}\ge 1$. Such functions provide the localized time mollifiers used in our uniqueness proof.

  We also obtain a complementary uniqueness result requiring no eigenfunction bound. If, in addition to the data, the Riemannian geometry is known on a connected open set meeting both $\src$ and $\rcv$, then the manifold is determined uniquely up to isometry (\cref{item:connected-U} of \cref{prop:main-corollaries}).

  The nonuniqueness result is an explicit construction of a pair of manifolds with the same data, together with a proof that the manifolds are not related by a Riemannian isometry. The same construction works in every dimension $n\ge2$. We believe that the mechanism underlying it may also be useful in establishing nonuniqueness for other inverse problems. A large language model played a significant part in the discovery of this result; see the \hyperref[sec:AI-disclosure]{generative-AI disclosure} preceding the references.

  \subsection{Relation to previous work}\label{sec:relation-to-others}

  Nonuniqueness results in geometric inverse problems have been relatively rare. Sunada's classic construction of isospectral, nonisometric closed manifolds~\cite{MR782558} may be viewed as one: eigenvalues alone do not determine a closed manifold up to isometry. Transformation-optics constructions form another family, in which singular changes of variables produce cloaking and exact invisibility~\cite{MR2457072}. More recently, Daudé, Kamran, and Nicoleau proved nonuniqueness for the anisotropic Calderón problem with data on disjoint sets~\cite{MR3973447,MR3916964}. Their construction and ours both use two warped Riemannian metrics on a smooth product manifold, but our combinatorial word mechanism appears to be new in geometric inverse problems. Further nonuniqueness results concern singular or nonsmooth metrics~\cite{MR4106946,MR4061977}, the anisotropic Calderón problem with a zeroth-order term~\cite{DEHKN-2026}, and smooth, time-dependent Lorentzian metrics~\cite{MR4382875,OS-2026}.

  We now narrow the discussion to uniqueness results for inverse problems related to ours. Much of this literature concerns the boundary analogue: for a manifold with boundary and nonempty open subsets $\src$ and $\rcv$ of the boundary, one prescribes Dirichlet data supported in $\src$ and observes the corresponding Neumann data on $\rcv$, seeking to determine $(\M,g)$. We call this the \emph{boundary} problem and ours the \emph{interior} problem. In either setting, we speak of \emph{disjoint data} when $\src$ and $\rcv$ are allowed to be disjoint. We begin with the case $\src=\rcv$.

  In dimensions $n\ge 2$, the boundary control method was put forward by Belishev~\cite{MR924687}, who combined finite-speed arguments with control theory to solve the boundary problem for domains of Euclidean space. Belishev and Kurylev~\cite{MR1177292} then extended the method to the recovery of a Riemannian manifold. A geometric formulation based on the boundary-distance representation was subsequently developed by Katchalov, Kurylev, and Lassas; see their monograph~\cite{MR1889089}.

  Inverse problems for eigenfunction data restricted to an interior hypersurface of a closed manifold were studied by Krupchyk, Kurylev, and Lassas~\cite{MR2435213}. Under bounded-geometry assumptions, Bosi, Kurylev, and Lassas obtained stable reconstruction from spectral data on an open ball~\cite{MR4409880}. Helin, Lassas, Oksanen, and Saksala proved uniqueness for the interior problem with $\src=\rcv$ on complete manifolds~\cite{MR3826551}; the geometric part of their argument forms the core of our uniqueness proof (see \cref{sec:HLOS-balls}). A single suitably chosen measurement on one open set determines the manifold for the heat and space--time fractional diffusion equations~\cite{MR4109393}, and the metric of a given closed manifold for the strongly damped wave equation~\cite{MR4713820}.

  Relatively few results are available for inverse problems with disjoint source and receiver sets. For a one-dimensional wave equation, Rakesh showed in~\cite{MR1766761} that boundary measurements with sources and receivers on different ends determine the coefficient. In $\R^3$, Rakesh and Sacks~\cite{MR2794398} recovered a potential in an annulus from a point source and receivers on a surrounding sphere, under an angular-control condition. Imanuvilov, Uhlmann, and Yamamoto~\cite{MR2819949} determined a conductivity, and likewise a potential for the stationary Schrödinger equation, on a known planar domain from Dirichlet data on one boundary set and Neumann data on another whose closures are disjoint, under an additional condition on the arrangement of the boundary sets.

  For the boundary and interior disjoint data problems, unconditional uniqueness had remained open, and our nonuniqueness result explains why. For the boundary problem, Lassas and Oksanen showed in~\cite{MR2661691} that uniqueness holds if the closures of $\src$ and $\rcv$ intersect. In a substantial advancement~\cite{MR3192525}, they proved that uniqueness holds under exact controllability from $\src$, or under the more general Hassell--Tao condition for eigenvalues and Dirichlet eigenfunctions:
  \begin{equation}\label{eq:intro-Hassell-Tao}
    \normsrc{\partial_\nu\phi} \ge c\sqrt{\lambda}\normM{\phi}.
  \end{equation}

  For the interior disjoint data problem, uniqueness was proved in~\cite{MR4833220} under a condition similar to~\labelcref{eq:intro-Hassell-Tao}: for every eigenfunction $\phi$ of $-\Delta_g$,
  \begin{equation}\label{eq:intro-condition-C}
    \normsrc{\phi|_\src} \ge c \normM{\phi}.
  \end{equation}
  This condition is implied by exact controllability from $\src$, and on a closed surface whose geodesic flow is Anosov it holds for every nonempty open set $\src$~\cite{MR4374954}. Li~\cite{Li_2026} treated a space--time fractional Schrödinger equation with disjoint data under the same condition.

  The condition~\labelcref{eq:intro-condition-C} corresponds to $\theta=0$ in our \cref{eq:intro-subcritical-bound}. (The data are symmetric in $\src$ and $\rcv$ by \cref{rem:data-symmetry}, so the bound may be assumed on either set.) Even in the case $\theta=0$, our result extends that of~\cite{MR4833220} in two ways: (i)~we assume only that the smooth structures of $\src$ and $\rcv$ are known, rather than the Riemannian metric; and (ii)~we prove that any two manifolds with the same data admit an isometry restricting to the identification maps on $\src$ and $\rcv$, whereas~\cite{MR4833220} proves only that an isometry exists.

  We are not aware of any uniqueness results, for the interior or the boundary problem, under a hypothesis corresponding to $\theta\in\ointerval{0}{1/2}$. Ingham-type spectral conditions have recently appeared in related settings: Daudé, Enciso, Helffer, Kamran, and Nicoleau~\cite{DEHKN2-2026} prove local-to-global propagation results for Dirichlet-to-Neumann maps, with sources and observations on a common open boundary set, under quasianalytic spectral hypotheses. These involve the divergent logarithmic-integral condition, complementary to the convergent condition~\labelcref{eq:intro-Ingham} used here.

  Recovery, up to the natural gauge, of lower-order coefficients from disjoint data when the metric is known was proved in~\cite{MR3950480,MR5079202}.

  As explained above, in disjoint data problems the global norm of a solution is not directly available from the data. A related loss of normalization occurs in passive measurements~\cite{FK-2025} and in the Calderón problem for fractional Schrödinger operators~\cite{F-2026}, where the eigenfunctions carry no normalization or orthogonality information. In the latter, the corresponding Gel'fand inverse spectral problem is solved under~\labelcref{eq:intro-condition-C}, with $\src$ replaced by the observation set, together with an additional geometric condition involving antipodal sets.

  For certain fractional and nonlinear equations, uniqueness can hold even with disjoint source and receiver sets~\cite{MR4078233,MR4336252}. In these settings, nonlocality or nonlinear interaction can provide information unavailable to the linear wave equation. For the use of nonlinear interaction in inverse problems, we refer the reader to~\cite{MR3802298}.

  \subsection{Outline of the paper}

  In \cref{sec:data-inverse-problem}, we introduce cross-spectral data, which extend local spectral data to possibly disjoint source and receiver sets, and prove that they are equivalent to the local source-to-solution data. In \cref{sec:uniqueness}, we prove the uniqueness results by pairing waves with functions in spectrally weighted spaces and by a geometric ball-inclusion argument. In \cref{sec:nonuniqueness}, we prove the nonuniqueness result by a construction based on warped products and the transfer matrices of binary words. Finally, in \cref{sec:HLOS-balls}, we prove the geometric reconstruction theorem used in the uniqueness argument.

%% file: problem.tex
\section{The disjoint data inverse problem}\label{sec:data-inverse-problem}

This section develops the common framework for the uniqueness and nonuniqueness arguments. We introduce cross-spectral data and establish equivalent formulations in terms of cross-compressed resolvents and spectral projectors, as well as the local source-to-solution operator for the wave equation. For any two manifolds with equivalent data, we also show that the given identifications of the source and receiver sets are Riemannian isometries, and we construct matching maps between corresponding eigenspaces.

\subsection{Assumptions and notation}\label{sec:assumptions}

\begin{assumption}[Standing assumption]\label{ass:standing}
  Let $(\M,g)$ and $(\tM,\tg)$ be smooth, connected, closed Riemannian manifolds of dimension at least two. Furthermore, let $\src,\rcv\subset\M$ and $\tsrc,\trcv\subset\tM$ be nonempty open subsets, with diffeomorphisms
  \begin{equation*}
    \Phi_\src\colon\src\to\tsrc
    \quad\text{and}\quad
    \Phi_\rcv\colon\rcv\to\trcv.
  \end{equation*}
\end{assumption}
\noindent These hypotheses and this notation are used throughout without further comment.

The maps $\Phi_\src$ and $\Phi_\rcv$ identify the subsets $\src$ and $\rcv$ with $\tsrc$ and $\trcv$, respectively. They are assumed to be diffeomorphisms, but not a priori to be Riemannian isometries. Thus, the smooth structures of $\src$ and $\rcv$ are known, whereas their induced metrics---and hence their volume measures---are not.

We use the following terminology for functions that agree under the identifications $\Phi_\src$ and $\Phi_\rcv$:
\begin{definition}[Matched functions]\label{def:matched}
  Measurable functions $\psi$ on $\src$ and $\wt{\psi}$ on $\tsrc$ are said to be \emph{$\Phi_\src$-matched} if $\psi=\wt{\psi}\circ\Phi_\src$ almost everywhere on $\src$.

  Time-dependent sources $s\in\Lloc{\rinterval{0}{\infty};\Lsrc}$ and $\tilde{s}\in\Lloc{\rinterval{0}{\infty};\tLsrc}$ are \emph{$\Phi_\src$-matched} if $s(t)$ and $\tilde{s}(t)$ are $\Phi_\src$-matched for almost every $t\in\rinterval{0}{\infty}$. The $\Phi_\rcv$-matched functions are defined similarly. When it is clear from the domains which of $\Phi_\src,\Phi_\rcv$ is meant, we simply say that the functions are \emph{matched}.
\end{definition}

The notion of matching is well defined on $L^2$-equivalence classes, since $\Phi_\src$ and its inverse preserve null sets; the same holds for $\Phi_\rcv$.

We sometimes regard functions of time $t$ and space $x$ as functions of $t$ alone, valued in a space of functions of $x$. In $L^2$-based settings, we identify functions on $\src$ or $\rcv$ with their zero extensions to $\M$, and analogously on $\tM$. Inclusions such as $\Cc{\src}\subset\LM$ or $\Cc[\big]{\ointerval{0}{\infty}\times\src}\subset\Lloc{\rinterval{0}{\infty};\LM}$ are understood in light of these conventions.

The following notation is used throughout the paper. Unless stated otherwise, all vector spaces are over $\R$.

\begin{itemize}
  \item The Laplace--Beltrami operator on $(\M,g)$ is denoted by $\Delta_g$. It is a self-adjoint, nonpositive operator with domain $H^2(\M)$. The \emph{distinct} eigenvalues of $-\Delta_g$ are $0=\lambda_0<\lambda_1<\cdots\to\infty$, indexed by the nonnegative integers $\ell$. The corresponding eigenspaces are the finite-dimensional spaces $\El:=\ker(-\Delta_g - \lambda_\ell)\subset C^\infty(\M)$, and the spectral projectors $P_\ell\colon\LM\to \El$ are the orthogonal projectors onto them. Each $\El$ is equipped with the inner product inherited from $\LM$. The spectrum of an operator $A$ is denoted by $\spec(A)$, so $\spec(-\Delta_g)=\{\lambda_\ell: \ell\ge 0\}$. For these and other properties of the Laplace--Beltrami operator, we refer the reader to~\cite{MR2569498,MR768584}.

  \item Similarly, the \emph{distinct} eigenvalues of $-\Delta_{\tg}$ on $\tM$ are denoted by $0=\wt{\lambda}_0<\wt{\lambda}_1<\cdots\to\infty$, the corresponding eigenspaces are $\tEl\subset C^\infty(\tM)$, and the spectral projectors are $\wt{P}_\ell\colon\tLM\to\tEl$.

  \item For a normed space $X$, $\calB(X)$ denotes the space of bounded linear operators from $X$ to itself.

  \item The wave propagators on $\LM$ are denoted by
    \begin{align}
      \cosprop\colon\R&\to\calB(\LM); \quad
      \cosprop_t\psi  := \sum_{\ell\ge 0} \cos(\omega_\ell t)\,P_\ell\psi,
      \label{eq:cosprop}\\
      \sinprop\colon\R&\to\calB(\LM); \quad
      \sinprop_t\psi  := \sum_{\ell\ge 0} \sigma_\ell(t)\,P_\ell\psi,
      \label{eq:sinprop}
    \end{align}
    where $\omega_\ell:=\sqrt{\lambda_\ell}$ and
    \begin{equation*}
      \sigma_0(t):=t
      \quad\text{and}\quad
      \sigma_\ell(t):=\frac{\sin(\omega_\ell t)}{\omega_\ell},
      \quad \ell\ge1.
    \end{equation*}
    For every $t\in\R$ and $\psi\in\LM$, both series~\labelcref{eq:cosprop,eq:sinprop} converge in $\LM$. For a source $f\in\Lloc{\rinterval{0}{\infty};\LM}$, the solution $v^f\in C(\rinterval{0}{\infty};\LM)$ to the wave equation with zero initial data is
    \begin{equation}\label{eq:wave-solution}
      v^f(t) := \int_0^t \sinprop_{t-\tau} f(\tau)\,d\tau,
    \end{equation}
    where integration is in the sense of Bochner. The analogous solution on $(\tM,\tg)$ with source $\tilde{f}$ is denoted by $\wt{v}^{\tilde{f}}$. When a source is spatially supported in $\src$ or $\rcv$, we often denote it by $s$ or $r$, respectively, and use $\tilde{s}$ or $\tilde{r}$ on $\tM$.

\end{itemize}

\subsection{Cross-spectral data}\label{sec:cross-spectral-data}

For each $\ell$, let $K_\ell\in C^\infty(\M\times\M)$ be the \emph{spectral kernel} of $P_\ell$, that is, its Schwartz kernel with respect to the Riemannian volume measure $dV_g$:
\begin{equation}\label{eq:P-kernel}
  (P_\ell \psi)(x) = \int_\M K_\ell(x,y)\,\psi(y)\,dV_g(y),
  \qquad \psi\in\LM,\quad x\in\M.
\end{equation}
Concretely, let $(\phi_j)_j$ be an orthonormal eigenbasis of $-\Delta_g$ in $\LM$, and set $\calJ_\ell:=\{j:\phi_j\in\El\}$. Then
\begin{equation}\label{eq:K-kernel}
  K_\ell(x,y)=\sum_{j\in\calJ_\ell}\phi_j(x)\phi_j(y).
\end{equation}
Define $\wt{K}_\ell\in C^\infty(\tM\times\tM)$ analogously for $(\tM,\tg)$.

By \cref{eq:K-kernel}, the kernel $K_\ell$ is symmetric, and each of its slices $K_\ell(x,\cdot)$, $x\in\M$, belongs to $\El$. Applying \cref{eq:P-kernel} to $\psi=\phi\in\El$ shows that
\begin{equation}\label{eq:K-reproducing}
  \innerM{K_\ell(x,\cdot)}{\phi}=\phi(x),
  \qquad \phi\in\El,\quad x\in\M;
\end{equation}
in particular, $\innerM{K_\ell(x,\cdot)}{K_\ell(x',\cdot)}=K_\ell(x,x')$ for all $x,x'\in\M$. The same identities hold on $(\tM,\tg)$.

\begin{definition}[Cross-spectral data]\label{def:cross-spectral-data}
  Let $K_\ell$ and $\wt{K}_\ell$ be as above.
  \begin{enumerate}
    \item The functions $K_\ell|_{\src\times\rcv}\in C^\infty(\src\times\rcv)$ are the \emph{cross-spectral kernels} of $(\M,g)$ relative to the pair $(\src,\rcv)$.
    \item The family $\bigl(\lambda_\ell,K_\ell|_{\src\times\rcv}\bigr)_\ell$ is the \emph{cross-spectral data} of $(\M,g)$ relative to the pair $(\src,\rcv)$.
    \item The manifolds $(\M,g)$ and $(\tM,\tg)$ have \emph{equivalent cross-spectral data} relative to the identifications $(\Phi_\src,\Phi_\rcv)$ if, for every $\ell$,
      \begin{equation*}
        \lambda_\ell = \wt{\lambda}_\ell\quad\text{and}\quad
        K_\ell(x,y)=\wt{K}_\ell(\Phi_\src(x),\Phi_\rcv(y))\quad\text{for all}\quad (x,y)\in\src\times\rcv.
      \end{equation*}
  \end{enumerate}
\end{definition}

When the pair of maps $(\Phi_\src,\Phi_\rcv)$ is clear from the context, we simply speak of equivalent cross-spectral data.

The inverse problem is to recover $(\M,g)$, up to Riemannian isometry, from the cross-spectral data. In \cref{prop:data-summary}, we prove that the cross-spectral data determine and are determined by the observations on $\rcv$ of waves generated by sources in $\src$.

\begin{remark}\label{rem:classical-data}
  Let $(\lambda'_j)_j$ denote the nondecreasing sequence of eigenvalues of $-\Delta_g$, \emph{repeated according to multiplicity}; the prime distinguishes it from the sequence $(\lambda_\ell)_\ell$ of distinct eigenvalues. In the classical case $\src=\rcv=U$, the local spectral data are of the form
  \begin{equation}\label{eq:classical-spectral-data}
    (\lambda'_j,\phi_j|_U)_j,
  \end{equation}
  where $(\phi_j)_j$ is an orthonormal eigenbasis of $-\Delta_g$ in $\LM$, with corresponding eigenvalues $(\lambda'_j)_j$~\cite{MR3826551}. After choosing such an eigenbasis, one could analogously take as data for the disjoint-data problem the products $\phi_j|_\src\otimes\phi_j|_\rcv$. Instead, we take their sums over each eigenspace---the cross-spectral kernels (see \cref{eq:K-kernel}). These sums are independent of the choice of orthonormal eigenbasis. When $\src=\rcv=U$, the local spectral data~\labelcref{eq:classical-spectral-data} determine the cross-spectral data.
\end{remark}

The following \namecref{prop:trace-spaces} shows that the slices $K_\ell(x,\cdot)$ with $x$ ranging over a nonempty open set already span $\El$. Consequently, the cross-spectral data determine the restrictions of each eigenspace $\El$ to $\src$ and to $\rcv$, and identify them with the corresponding restrictions of $\tEl$.

\begin{lemma}\label{prop:trace-spaces}
  Write $\El|_\src:=\{\phi|_\src:\phi\in\El\}$ for the restriction of $\El$ to $\src$, and define $\El|_\rcv$ analogously.
  \begin{enumerate}
    \item\label{item:K-spans} For every $\ell$ and every nonempty open set $U\subset\M$,
      \begin{equation}\label{eq:K-spans}
        \linspan\{K_\ell(x,\cdot) : x\in U\} = \El.
      \end{equation}

    \item\label{item:trace-spaces} Assume that $(\M,g)$ and $(\tM,\tg)$ have equivalent cross-spectral data. Then, for every $\ell$,
      \begin{equation}\label{eq:trace-spaces}
        \El|_\src = \{\wt{\phi}\circ\Phi_\src : \wt{\phi}\in\tEl\}
        \quad\text{and}\quad
        \El|_\rcv = \{\wt{\phi}\circ\Phi_\rcv : \wt{\phi}\in\tEl\}.
      \end{equation}
  \end{enumerate}
\end{lemma}
\begin{proof}
  \begin{enumerate}
    \item If $\phi\in\El$ is orthogonal to $K_\ell(x,\cdot)$ for every $x\in U$, then $\phi$ vanishes on $U$ by \cref{eq:K-reproducing}, and hence $\phi=0$ by unique continuation~\cite{MR92067}. Therefore, \cref{eq:K-spans} holds.

    \item We prove the first identity in \cref{eq:trace-spaces}; the second is proved similarly. We have
      \begin{align*}
        \El|_\src
        &= \linspan\{K_\ell(y,\cdot)|_\src : y\in\rcv\}\\
        &= \linspan\{\wt{K}_\ell(\Phi_\rcv(y),\Phi_\src(\cdot)) : y\in\rcv\}\\
        &= \linspan\{\wt{K}_\ell(\ty,\Phi_\src(\cdot)) : \ty\in\trcv\}\\
        &= \{\wt{\phi}\circ\Phi_\src : \wt{\phi}\in\tEl\}.
      \end{align*}
      The first equality is \cref{eq:K-spans} with $U=\rcv$, restricted to $\src$; the second follows from the equivalence of the cross-spectral data and the symmetry of the kernels; the third holds because $\Phi_\rcv$ maps $\rcv$ onto $\trcv$; and the fourth is \cref{eq:K-spans} on $(\tM,\tg)$ with $U=\trcv$, restricted to $\tsrc$ and composed with $\Phi_\src$.\qedhere
  \end{enumerate}
\end{proof}

\subsection{Equivalent formulations and the metric on \texorpdfstring{$\src$ and $\rcv$}{S and R}}\label{sec:equivalent-data}

We now give four further equivalent formulations of the cross-spectral data: in terms of cross-compressed resolvents, cross-compressed spectral projectors in the two directions, and local source-to-solution operators. Passing between the cross-spectral kernels and their operator realizations is not immediate, because the Riemannian volume measure is not a priori part of the data. In particular, we do not yet know whether pullback by $\Phi_\rcv$ defines a bounded map from $\tLrcv$ to $\Lrcv$.

Testing a pullback against a fixed $\chi\in\Cc{\rcv}$ is nevertheless possible. Indeed, let $J\in C^\infty(\rcv)$ be the positive function (Jacobian of $\Phi_\rcv$) so that $\int_\rcv (\wt{\xi}\circ\Phi_\rcv)\,J\,dV_g = \int_{\trcv} \wt{\xi}\,dV_{\tg}$ for every measurable nonnegative function $\wt{\xi}$ on $\trcv$. Set $\gamma_\chi:=(\chi/J)\circ\Phi_\rcv^{-1}\in\Cc{\trcv}$. Then
\begin{equation}\label{eq:Gamma-chi}
  \Gamma_\chi\wt{\xi}
  := \int_\rcv (\wt{\xi}\circ\Phi_\rcv)\,\chi\,dV_g
  = \tinnerrcv{\wt{\xi}}{\gamma_\chi},
  \qquad \wt{\xi}\in\tLrcv,
\end{equation}
where the integral converges absolutely; in particular, $\Gamma_\chi$ is a bounded linear functional on $\tLrcv$.

The following lemma shows that matching (\cref{def:matched}) passes to limits and Bochner integrals. We state it for $\Phi_\rcv$; the version for $\Phi_\src$ is identical.

\begin{lemma}\label{prop:matched-calculus}
  \begin{enumerate}
    \item\label{item:matched-limit} Let $u_j\in\Lrcv$ and $\wt{u}_j\in\tLrcv$ be $\Phi_\rcv$-matched for every $j$. If $u_j\to u$ in $\Lrcv$ and $\wt{u}_j\to\wt{u}$ in $\tLrcv$, then $u$ and $\wt{u}$ are $\Phi_\rcv$-matched.

    \item\label{item:matched-integral} Let $I\subset\R$ be a nonempty interval. If $F\in L^1(I;\Lrcv)$ and $\wt{F}\in L^1(I;\tLrcv)$ are such that $F(\tau)$ and $\wt{F}(\tau)$ are $\Phi_\rcv$-matched for a.e.\ $\tau\in I$, then $\int_I F\,d\tau$ and $\int_I \wt{F}\,d\tau$ are $\Phi_\rcv$-matched.
  \end{enumerate}
\end{lemma}
\begin{proof}
  \begin{enumerate}
    \item After passing to a common subsequence, $u_j\to u$ almost everywhere on $\rcv$ and $\wt{u}_j\to\wt{u}$ almost everywhere on $\trcv$. Since $\Phi_\rcv^{-1}$ preserves null sets, the latter convergence gives $\wt{u}_j\circ\Phi_\rcv\to\wt{u}\circ\Phi_\rcv$ almost everywhere on $\rcv$. Letting $j\to\infty$ in $u_j=\wt{u}_j\circ\Phi_\rcv$ gives the claim.

    \item Let $\chi\in\Cc{\rcv}$. Bounded functionals commute with Bochner integrals, so, by \cref{eq:Gamma-chi},
      \begin{align*}
        \innerrcv{\int_I F\,d\tau}{\chi}
        &= \int_I \innerrcv{F(\tau)}{\chi}\,d\tau\\
        &= \int_I \Gamma_\chi\bigl(\wt{F}(\tau)\bigr)\,d\tau\\
        &= \Gamma_\chi\Bigl(\int_I \wt{F}\,d\tau\Bigr).
      \end{align*}
      It follows that $\int_I F\,d\tau$ and $\bigl(\int_I \wt{F}\,d\tau\bigr)\circ\Phi_\rcv$ agree almost everywhere.\qedhere
  \end{enumerate}
\end{proof}

\subsubsection*{The five conditions.}
Recall from \cref{eq:LSRw} the local source-to-solution operator $\LSRw\colon \Cc[\big]{\ointerval{0}{\infty}\times\src} \to C^\infty\bigl(\ointerval{0}{\infty}\times\rcv\bigr)$ for the wave equation on $(\M,g)$. The corresponding operator on $(\tM,\tg)$ is denoted by $\tLSRw$.

Let $\res_\rcv\colon \LM\to\Lrcv$ and $\ext_\src\colon \Lsrc\to \LM$ be, respectively, the restriction and extension-by-zero operators on $\M$, and let $\res_{\trcv}$ and $\ext_{\tsrc}$ be the analogous operators on $\tM$. Define the \emph{cross-compressed resolvents}
\begin{align}
  \RSR(z)&:=\res_{\rcv}\circ(-\Delta_g-z)^{-1}\circ\ext_{\src},
  \qquad z\in\R\setminus\spec(-\Delta_g),\label{eq:RSR}\\
  \tRSR(z)&:=\res_{\trcv}\circ(-\Delta_{\tg}-z)^{-1}\circ\ext_{\tsrc},
  \qquad z\in\R\setminus\spec(-\Delta_{\tg}).\label{eq:tRSR}
\end{align}
These are bounded operators $\Lsrc\to\Lrcv$ and $\tLsrc\to\tLrcv$, respectively. Similarly, the \emph{cross-compressed spectral projectors} are
\begin{equation}\label{eq:PlRS}
  \PlRS := \res_{\rcv}\circ P_\ell\circ\ext_{\src}
  \quad\text{and}\quad
  \tPlRS := \res_{\trcv}\circ\wt{P}_\ell\circ\ext_{\tsrc},
\end{equation}
which are bounded operators between the same pairs of spaces. The operators $\PlSR\colon\Lrcv\to\Lsrc$ and $\tPlSR\colon\tLrcv\to\tLsrc$ below are defined by reversing the roles of $\src$ and $\rcv$ in \cref{eq:PlRS}.

Condition \cref{data:cross-spectral} below restates the equivalence of cross-spectral data introduced in \cref{def:cross-spectral-data}, while \crefrange{data:resolvents}{data:projectors-opposite} are alternative formulations. We prove in \cref{prop:data-summary} that all five conditions are equivalent.

\begin{datacond}
\item\label{data:cross-spectral} The manifolds $(\M,g)$ and $(\tM,\tg)$ have equivalent cross-spectral data relative to $(\Phi_\src,\Phi_\rcv)$.

\item\label{data:resolvents} For every pair of $\Phi_\src$-matched functions $\psi\in\Cc{\src}$, $\wt{\psi}\in\Cc{\tsrc}$ and every $z<0$, the functions $\RSR(z)\psi$ and $\tRSR(z)\wt{\psi}$ are $\Phi_\rcv$-matched.

\item\label{data:projectors} For every $\ell$, we have $\lambda_\ell=\wt{\lambda}_\ell$ and, for every pair of $\Phi_\src$-matched functions $\psi\in\Cc{\src}$, $\wt{\psi}\in\Cc{\tsrc}$, the functions $\PlRS\psi$ and $\tPlRS\wt{\psi}$ are $\Phi_\rcv$-matched.

\item\label{data:wave} For every pair of $\Phi_\src$-matched sources $s\in\Cc[\big]{\ointerval{0}{\infty}\times\src}$, $\tilde{s}\in\Cc[\big]{\ointerval{0}{\infty}\times\tsrc}$ and every $t>0$, the observations $\LSRw s\,(t)$ and $\tLSRw \tilde{s}\,(t)$ are $\Phi_\rcv$-matched.
\end{datacond}

\begin{remark}\label{rem:data-symmetry}
  Since $K_\ell(x,y)=K_\ell(y,x)$, the cross-spectral data relative to $(\src,\rcv)$ determine those relative to $(\rcv,\src)$. Once \crefrange{data:cross-spectral}{data:wave} have been shown equivalent, this yields the three corresponding formulations with the source and receiver roles reversed. We state only the reversed projector formulation, which will be used in \cref{prop:Gl-on-wave}.
  \begin{datacond}[resume]
  \item\label{data:projectors-opposite} For every $\ell$, we have $\lambda_\ell=\wt{\lambda}_\ell$ and, for every pair of $\Phi_\rcv$-matched functions $\psi\in\Cc{\rcv}$, $\wt{\psi}\in\Cc{\trcv}$, the functions $\PlSR\psi$ and $\tPlSR\wt{\psi}$ are $\Phi_\src$-matched.
  \end{datacond}
\end{remark}

\subsubsection*{Equivalence of the conditions.}
In~\cite[Appendix~A]{MR4833220}, conditions \cref{data:projectors,data:wave} were shown to be equivalent under the assumption that the identifications $\Phi_\src$ and $\Phi_\rcv$ are Riemannian isometries. The results below remove this assumption: the identifications are initially only diffeomorphisms, and equivalence of the data forces them to be Riemannian isometries.

We first establish the equivalence of \crefrange{data:resolvents}{data:wave}. We then recover the Riemannian metric on $\src$ and $\rcv$ and use the resulting measure preservation to incorporate \cref{data:cross-spectral}. Finally, symmetry gives \cref{data:projectors-opposite}.

\begin{lemma}\label{prop:operator-data-equivalence}
  The conditions \cref{data:resolvents,data:projectors,data:wave} are equivalent.
\end{lemma}

\begin{proof}
  (\labelcref{data:resolvents} $\Rightarrow$ \labelcref{data:projectors}) Let $\psi\in\Cc{\src}$ and $\wt{\psi}\in\Cc{\tsrc}$ be matched, and let $\chi\in\Cc{\rcv}$. We first show that
  \begin{equation}\label{eq:meromorphic}
    \sum_{\ell\ge 0} \frac{\innerrcv{\PlRS\psi}{\chi}}{\lambda_\ell-z}
    \quad\text{and}\quad
    \sum_{\ell\ge 0} \frac{\Gamma_\chi\bigl(\tPlRS\wt{\psi}\bigr)}{\wt{\lambda}_\ell-z}
  \end{equation}
  are meromorphic functions of $z\in\C$ that agree for $z\notin\spec(-\Delta_g)\cup\spec(-\Delta_{\tg})$. Here $\Gamma_\chi$ is the functional defined in \cref{eq:Gamma-chi}.

  The coefficients of the first series are
  \begin{equation*}
    a_\ell
    :=\innerrcv{\PlRS\psi}{\chi}
    =\innerM{P_\ell\psi}{P_\ell\chi},
  \end{equation*}
  so by the Cauchy--Schwarz inequality and Parseval's identity,
  \begin{equation}\label{eq:coefficients-summable}
    \sum_{\ell\ge 0}\abs{a_\ell}
    \le \normM{\psi}\,\normM{\chi}
    <\infty.
  \end{equation}
  By \cref{eq:coefficients-summable} and the Weierstrass $M$-test, the series on the left-hand side of \cref{eq:meromorphic} converges uniformly on compact subsets of $\C\setminus\spec(-\Delta_g)$, so it defines a holomorphic function there. By the same reasoning, for each $\ell_0=0,1,\dots$, the series either extends holomorphically to $\lambda_{\ell_0}$ (if $a_{\ell_0}=0$) or has a simple pole there, so the series defines a meromorphic function on $\C$.

  The right-hand side of \cref{eq:meromorphic} is handled similarly. By \cref{eq:Gamma-chi}, its coefficients are
  \begin{equation*}
    b_\ell
    := \Gamma_\chi\bigl(\tPlRS\wt{\psi}\bigr)
    = \tinnerrcv{\tPlRS\wt{\psi}}{\gamma_\chi}.
  \end{equation*}
  This is of the same form as $a_\ell$, so the preceding argument applies to $(b_\ell)_\ell$ as well, proving that the second series is also meromorphic on $\C$.

  For $z<0$, applying the spectral decompositions of the resolvents and using the boundedness of $\xi\mapsto\innerrcv{\xi}{\chi}$ on $\Lrcv$ and of $\Gamma_\chi$ on $\tLrcv$ yields
  \begin{align*}
    \innerrcv{\RSR(z)\psi}{\chi}
    &=\sum_{\ell\ge 0}\frac{a_\ell}{\lambda_\ell-z};\\
    \Gamma_\chi\bigl(\tRSR(z)\wt{\psi}\bigr)
    &=\sum_{\ell\ge 0}\frac{b_\ell}{\wt{\lambda}_\ell-z}.
  \end{align*}
  By the above and \labelcref{data:resolvents}, the two series define meromorphic functions that agree for $z<0$ and hence on the connected open set $\C\setminus\bigl(\spec(-\Delta_g)\cup\spec(-\Delta_{\tg})\bigr)$. This proves the claim above.

  It follows from unique continuation that $\PlRS\neq0$ for every $\ell$, so that $\innerrcv{\PlRS\psi}{\chi}\neq 0$ for a suitable choice of $\psi$ and $\chi$. For such a choice, $\lambda_\ell$ is a pole of the function on the left-hand side of \cref{eq:meromorphic}, hence also a pole of the function on the right-hand side, so that $\spec(-\Delta_g)\subset\spec(-\Delta_{\tg})$. The opposite inclusion follows similarly, and the ordering of the eigenvalues then gives $\lambda_\ell=\wt{\lambda}_\ell$ for all $\ell$. Comparing the residues at $z=\lambda_\ell$ shows that $\innerrcv{\PlRS\psi}{\chi}=\Gamma_\chi\bigl(\tPlRS\wt{\psi}\bigr)$ for every matched pair $\psi,\wt{\psi}$ and every $\chi\in\Cc{\rcv}$, whence $\PlRS\psi = (\tPlRS\wt{\psi})\circ\Phi_\rcv$ in $\Lrcv$.

  (\labelcref{data:projectors} $\Rightarrow$ \labelcref{data:wave}) Let $s\in\Cc[\big]{\ointerval{0}{\infty}\times\src}$ and $\tilde{s}\in\Cc[\big]{\ointerval{0}{\infty}\times\tsrc}$ be matched, and denote $s_\ell(\tau) :=\PlRS s(\tau)\in\Lrcv$ and $\tilde{s}_\ell(\tau) := \tPlRS \tilde{s}(\tau)\in\tLrcv$. Both sides of the matching identity $s_\ell(\tau) = \tilde{s}_\ell(\tau)\circ\Phi_\rcv$ are continuous in $(\tau,y)$ and agree almost everywhere; hence they agree pointwise on $\rcv$ for every $\tau\ge 0$.

  For every $t>0$, every $y\in\rcv$, and every $\ell$,
  \begin{align*}
    \left[\int_0^t\sigma_\ell(t-\tau)\, s_\ell(\tau)\,d\tau\right](y)
    &= \int_0^t\sigma_\ell(t-\tau)\, s_\ell(\tau,y)\,d\tau\\
    &= \int_0^t\sigma_\ell(t-\tau)\, \tilde{s}_\ell(\tau,\Phi_\rcv(y))\,d\tau\\
    &= \left[\int_0^t\sigma_\ell(t-\tau)\, \tilde{s}_\ell(\tau)\,d\tau\right]\bigl(\Phi_\rcv(y)\bigr),
  \end{align*}
  where the first and last equalities follow from the fact that $s_\ell$ and $\tilde{s}_\ell$ take their values in finite-dimensional spaces, and the second holds by hypothesis (the weight $\sigma_\ell$ is the same on both manifolds because $\lambda_\ell=\wt{\lambda}_\ell$). When summed over $\ell$, the terms in brackets converge in $\Lrcv$ and $\tLrcv$ to $\LSRw s(t)$ and $\tLSRw \tilde{s}(t)$, so \cref{item:matched-limit} of \cref{prop:matched-calculus} proves the implication.

  (\labelcref{data:wave} $\Rightarrow$ \labelcref{data:resolvents}) Let $\psi\in\Cc{\src}$ and $\wt{\psi}\in\Cc{\tsrc}$ be matched, fix a nonnegative $\rho\in \Cc[\big]{\ointerval{0}{\infty}}$ with $\rho\neq0$, and apply \labelcref{data:wave} to the matched sources $s(\tau):=\rho(\tau)\psi$ and $\tilde{s}(\tau):=\rho(\tau)\wt{\psi}$. For $p>0$, let us write $\calL u(p):=\int_0^\infty e^{-pt}u(t)\,dt$ for the Laplace transform of a Banach-space-valued function $u$ whenever the Bochner integral exists. The transforms below converge absolutely, since $\normM{v^s(t)}\le Ct$ and $\tnormM{\wt{v}^{\tilde{s}}(t)}\le Ct$. Applying the bounded operator $P_\ell$ to \cref{eq:wave-solution} and using $\calL\sigma_\ell(p)=1/(p^2+\lambda_\ell)$ shows that
  \begin{equation*}
    \calL v^s (p)
    = \calL\rho (p) \sum_{\ell\ge 0} \frac{1}{p^2+\lambda_\ell} P_\ell\psi
    = \calL\rho (p) R(-p^2)\psi,
  \end{equation*}
  where $R(-p^2)\in\calB(\LM)$ is the resolvent at $-p^2<0$. An analogous identity holds on $(\tM,\tg)$. Since $\calL\rho(p)>0$, applying $\res_\rcv$ to the displayed identity and $\res_{\trcv}$ to its tilded counterpart and using \cref{item:matched-integral} of \cref{prop:matched-calculus} yields, for a.e.\ $y\in\rcv$,
  \begin{align*}
    \bigl(\RSR(-p^2)\psi\bigr)(y)
    &= \frac{1}{\calL\rho(p)} \bigl(\calL(\LSRw s)(p)\bigr)(y)\\
    &= \frac{1}{\calL\rho(p)} \bigl(\calL(\tLSRw \tilde{s})(p)\bigr)\bigl(\Phi_\rcv(y)\bigr)\\
    &= \bigl(\tRSR(-p^2)\wt{\psi}\bigr)\bigl(\Phi_\rcv(y)\bigr).
  \end{align*}
  This proves the final implication.
\end{proof}

The next lemma isolates the mechanism by which each of the conditions \crefrange{data:cross-spectral}{data:wave} forces $\Phi_\src$ and $\Phi_\rcv$ to be Riemannian isometries.

\begin{lemma}\label{prop:trace-isometry}
  Assume that, for every $\ell$, $\lambda_\ell=\wt{\lambda}_\ell$ and the identities~\labelcref{eq:trace-spaces} hold. Then the diffeomorphisms $\Phi_\src$ and $\Phi_\rcv$ of \cref{ass:standing} are Riemannian isometries.
\end{lemma}
\begin{proof}
  We only prove $(\Phi_\rcv)^\ast\tg=g|_\rcv$; the other case is similar. If $\wt{\phi}\in\tEl$, then $\wt{\phi}\circ\Phi_\rcv\in\El|_\rcv$ by hypothesis, so
  \begin{equation*}
    \Delta_{(\Phi_\rcv)^\ast\tg}(\wt{\phi}\circ\Phi_\rcv)
    = \bigl(\Delta_{\tg}\wt{\phi}\bigr) \circ \Phi_\rcv
    = -\lambda_\ell\, \wt{\phi}\circ\Phi_\rcv
    = \Delta_g (\wt{\phi}\circ\Phi_\rcv).
  \end{equation*}
  It follows that $\Delta_{(\Phi_\rcv)^\ast\tg}=\Delta_g$ on $\linspan\bigcup_\ell \El|_\rcv$. For $\psi\in\Cc{\rcv}$, the expansion $\psi=\sum_{\ell\ge 0}P_\ell\psi$ converges in $\LM$, hence in $\calD'(\rcv)$ after restriction to $\rcv$, and both operators are continuous on $\calD'(\rcv)$. Therefore, $\Delta_{(\Phi_\rcv)^\ast\tg}=\Delta_g$ on $\Cc{\rcv}$, which implies $(\Phi_\rcv)^\ast\tg = g|_\rcv$.
\end{proof}

Under \cref{data:projectors}, we recover $\El|_\src$ from its orthogonal complement using the following elementary lemma, whose proof is included for completeness.

\begin{lemma}\label{prop:orthogonality-test}
  Let $H$ be a Hilbert space, $D\subset H$ a dense subspace, and $X,Y\subset H$ finite-dimensional subspaces. Assume that, for every $\psi\in D$, $\psi\perp X$ if and only if $\psi\perp Y$. Then $X=Y$.
\end{lemma}
\begin{proof}
  It suffices to prove that $D\cap X^\perp$ is dense in $X^\perp$ for every finite-dimensional subspace $X\subset H$. Indeed, once this density statement is proved, the hypothesis gives $D\cap X^\perp=D\cap Y^\perp$, and hence
  \begin{equation*}
    X^\perp
    =\overline{D\cap X^\perp}
    =\overline{D\cap Y^\perp}
    =Y^\perp.
  \end{equation*}
  Since $X$ and $Y$ are finite-dimensional and therefore closed, it follows that $X=(X^\perp)^\perp=(Y^\perp)^\perp=Y$.

  Let $e_1,\dots,e_k$ be a basis of $X$. The image of $D$ under the coordinate map $\psi\mapsto(\inner{\psi}{e_1},\dots,\inner{\psi}{e_k})$ is all of $\R^k$. Indeed, if it were a proper subspace, there would be a nonzero $v\in\R^k$ with $\inner{\psi}{\sum_i v_ie_i}=0$ for every $\psi\in D$, so that $\sum_i v_ie_i=0$, which contradicts the linear independence of $e_1,\dots,e_k$. Hence we may choose $\eta_1,\dots,\eta_k\in D$ with $\inner{\eta_i}{e_j}=\delta_{ij}$.

  Now let $u\in X^\perp$ and $(\psi_j)_j\subset D$ satisfy $\psi_j\to u$. Then
  \begin{equation*}
    \psi'_j := \psi_j-\sum_{i=1}^k\inner{\psi_j}{e_i}\,\eta_i \in D\cap X^\perp,
  \end{equation*}
  and $\inner{\psi_j}{e_i}\to\inner{u}{e_i}=0$ for every $i$, so $\psi'_j\to u$.
\end{proof}

\begin{lemma}\label{prop:data-isometry}
  Under any of the conditions \crefrange{data:cross-spectral}{data:wave}, the diffeomorphisms $\Phi_\src$ and $\Phi_\rcv$ are Riemannian isometries.
\end{lemma}

\begin{proof}
  By \cref{prop:operator-data-equivalence}, it suffices to treat \cref{data:cross-spectral,data:projectors}. In both cases, we verify the hypotheses of \cref{prop:trace-isometry}; the eigenvalues agree by definition under either condition. Under \cref{data:cross-spectral}, the identities~\labelcref{eq:trace-spaces} hold by \cref{item:trace-spaces} of \cref{prop:trace-spaces}.

  Let us assume \cref{data:projectors}. Then the two identities in \cref{eq:trace-spaces} can be read from the range and from the kernel of $\PlRS$ as follows.

  For the range, it follows from unique continuation that the operator $P_\ell$ maps $\Cc{\src}$ onto $\El$. Hence $\{\PlRS\psi : \psi\in\Cc{\src}\}=\El|_\rcv$, and likewise on $(\tM,\tg)$. Since composition with $\Phi_\src$ defines a bijection $\Cc{\tsrc}\to\Cc{\src}$, \cref{data:projectors} states exactly that the second identity in \cref{eq:trace-spaces} holds.

  For the kernel, first consider $\ell=0$. Since $\M$ is connected and closed, $E_0$ consists of the constant functions, and so $\PlRS[0]\psi$ is the constant function $\operatorname{vol}(\M)^{-1}\int_\src\psi\,dV_g$ on $\rcv$, and similarly on $(\tM,\tg)$. Matching the two constants therefore gives
  \begin{equation}\label{eq:measure-up-to-c}
    \int_{\tsrc}\wt{\psi}\,dV_{\tg} = c\int_\src \wt{\psi}\circ\Phi_\src\,dV_g
    \quad\text{for all}\quad
    \wt{\psi}\in\Cc{\tsrc},
  \end{equation}
  where $c:=\operatorname{vol}(\tM)/\operatorname{vol}(\M)$.

  Let $\psi\in\Cc{\src}$ and $\wt{\psi}\in\Cc{\tsrc}$ be matched. It follows from unique continuation that $\PlRS\psi=0$ if and only if $P_\ell\psi=0$. This implies that $\PlRS\psi=0$ if and only if $\psi\perp\El|_\src$ in $\Lsrc$. The same computation on $(\tM,\tg)$, followed by the change of variables~\labelcref{eq:measure-up-to-c}, shows that $\tPlRS\wt{\psi}=0$ if and only if $\psi\perp\{\wt{\phi}\circ\Phi_\src : \wt{\phi}\in\tEl\}$ in $\Lsrc$. By \cref{data:projectors}, $\PlRS\psi=0$ if and only if $\tPlRS\wt{\psi}=0$, so a function $\psi\in\Cc{\src}$ is orthogonal to $\El|_\src$ if and only if it is orthogonal to $\{\wt{\phi}\circ\Phi_\src : \wt{\phi}\in\tEl\}$. These two spaces are finite-dimensional subspaces of $\Lsrc$, and $\Cc{\src}$ is dense in $\Lsrc$, so \cref{prop:orthogonality-test} shows that they coincide. This is the first identity in \cref{eq:trace-spaces}.
\end{proof}

\begin{lemma}\label{prop:data-equivalence}
  The conditions \crefrange{data:cross-spectral}{data:wave} are equivalent.
\end{lemma}

\begin{proof}
  By \cref{prop:operator-data-equivalence}, it is enough to prove that \cref{data:cross-spectral} and \cref{data:projectors} are equivalent. Under either of them, the eigenvalues agree, and $\Phi_\src$ is measure-preserving by \cref{prop:data-isometry}, so for every matched pair $\psi\in\Cc{\src}$, $\wt{\psi}\in\Cc{\tsrc}$, and every $\ell$ and $y\in\rcv$,
  \begin{equation}\label{eq:kernels-vs-projectors-1}
    \begin{split}
      (\tPlRS\wt{\psi})(\Phi_\rcv(y))
      &= \int_{\tsrc}\wt{K}_\ell(\tx,\Phi_\rcv(y))\,\wt{\psi}(\tx)\,dV_{\tg}(\tx)\\
      &= \int_{\src}\wt{K}_\ell(\Phi_\src(x),\Phi_\rcv(y))\,\psi(x)\,dV_g(x),
    \end{split}
  \end{equation}
  while
  \begin{equation}\label{eq:kernels-vs-projectors-2}
    (\PlRS\psi)(y)=\int_\src K_\ell(x,y)\,\psi(x)\,dV_g(x).
  \end{equation}

  If \cref{data:cross-spectral} holds, the last integrand in \cref{eq:kernels-vs-projectors-1} and the one in \cref{eq:kernels-vs-projectors-2} coincide, so $\PlRS\psi$ and $\tPlRS\wt{\psi}$ are $\Phi_\rcv$-matched, which gives \cref{data:projectors}. Conversely, if \cref{data:projectors} holds, then
  \begin{equation*}
    \int_\src\bigl[K_\ell(x,y)-\wt{K}_\ell(\Phi_\src(x),\Phi_\rcv(y))\bigr]\,\psi(x)\,dV_g(x)=0
  \end{equation*}
  for every $\psi\in\Cc{\src}$ and a.e.\ $y\in\rcv$, so the two kernels, being continuous, agree on $\src\times\rcv$, and \cref{data:cross-spectral} follows.
\end{proof}

Under any of the conditions \crefrange{data:cross-spectral}{data:wave}, the identifications $\Phi_\src$ and $\Phi_\rcv$ are not independent of one another. In fact, a Riemannian isometry $(\M,g)\to(\tM,\tg)$ extending the receiver identification $\Phi_\rcv$ necessarily extends the source identification $\Phi_\src$ as well.

\begin{lemma}\label{prop:isometry-identifications}
  Assume that $(\M,g)$ and $(\tM,\tg)$ have equivalent cross-spectral data, and let $\Phi\colon(\M,g)\to(\tM,\tg)$ be a Riemannian isometry with $\Phi|_\rcv=\Phi_\rcv$. Then $\Phi|_\src=\Phi_\src$.
\end{lemma}

\begin{proof}
  Fix $\ell$ and $x\in\src$. Then, for every $y\in\rcv$,
  \begin{equation}\label{eq:spectral-kernels}
    \wt{K}_\ell\bigl(\Phi(x),\Phi(y)\bigr)
    =K_\ell(x,y)
    =\wt{K}_\ell\bigl(\Phi_\src(x),\Phi_\rcv(y)\bigr)
    =\wt{K}_\ell\bigl(\Phi_\src(x),\Phi(y)\bigr).
  \end{equation}
  The first equality follows from the invariance of the spectral kernels under $\Phi$, the second from the equivalence of the cross-spectral data, and the third from $\Phi|_\rcv=\Phi_\rcv$.

  The functions $\tM\ni\ty\mapsto\wt{K}_\ell\bigl(\Phi(x),\ty\bigr)$ and $\tM\ni\ty\mapsto\wt{K}_\ell\bigl(\Phi_\src(x),\ty\bigr)$ belong to $\tEl$ and, by \cref{eq:spectral-kernels}, agree on $\Phi_\rcv(\rcv)=\trcv$, so by unique continuation they agree everywhere. It follows that, for every $\wt{\psi}\in\tLM$,
  \begin{equation*}
    (\wt{P}_\ell \wt{\psi})\bigl(\Phi(x)\bigr)=(\wt{P}_\ell \wt{\psi})\bigl(\Phi_\src(x)\bigr).
  \end{equation*}

  Since $\ell$ was arbitrary, $\wt{\phi}(\Phi(x))=\wt{\phi}(\Phi_\src(x))$ for every $\wt{\phi}\in\bigcup_{\ell}\tEl$. The linear span of the eigenfunctions is dense in $C(\tM)$, so the two evaluation functionals agree on $C(\tM)$. Since continuous functions separate points, $\Phi(x)=\Phi_\src(x)$ for every $x\in\src$.
\end{proof}

We collect the results of this subsection.

\begin{proposition}\label{prop:data-summary}
  \begin{enumerate}
    \item\label{item:summary-equivalence} The conditions \crefrange{data:cross-spectral}{data:projectors-opposite} are equivalent.

    \item\label{item:summary-isometry} If they hold, the identifications $\Phi_\src$ and $\Phi_\rcv$ are Riemannian isometries.

    \item\label{item:summary-extension} If they hold and if $\Phi\colon(\M,g)\to(\tM,\tg)$ is a Riemannian isometry with $\Phi|_\rcv=\Phi_\rcv$, then $\Phi|_\src=\Phi_\src$.
  \end{enumerate}
\end{proposition}
\begin{proof}
  \Cref{prop:data-equivalence} proves the equivalence of \crefrange{data:cross-spectral}{data:wave}. As $K_\ell(x,y)=K_\ell(y,x)$, \cref{data:cross-spectral} is equivalent to the manifolds $(\M,g)$ and $(\tM,\tg)$ having equivalent cross-spectral data relative to $(\Phi_\rcv,\Phi_\src)$. By \cref{prop:data-equivalence} applied with the roles of $\src$ and $\rcv$ interchanged, this in turn implies the equivalence between \cref{data:cross-spectral} and \cref{data:projectors-opposite}. Hence, \crefrange{data:cross-spectral}{data:projectors-opposite} are all equivalent.

  \Cref{prop:data-isometry,prop:isometry-identifications} give the second and third claims, respectively.
\end{proof}

\begin{remark}\label{rem:data-equivalence}
  Assume that the conditions \crefrange{data:cross-spectral}{data:projectors-opposite} hold.
  \begin{enumerate}
    \item\label{rem:local-isometry} By \cref{item:summary-isometry} of \cref{prop:data-summary}, $\Phi_\src$ is measure-preserving, so $\Phi_\src$-matched functions have equal $L^2$-norms over $\src$ and $\tsrc$ (similarly for $\Phi_\rcv$-matched functions).

    \item\label{item:data-extension} By the density of $\Cc{\src}$ in $\Lsrc$, \cref{item:summary-isometry} of \cref{prop:data-summary}, and \cref{item:matched-limit} of \cref{prop:matched-calculus}, the matchings in \cref{data:resolvents,data:projectors} extend to every $\Phi_\src$-matched pair $\psi\in\Lsrc$, $\wt{\psi}\in\tLsrc$: for every such pair, every $z<0$, and every $\ell$, the pairs $\bigl(\RSR(z)\psi,\tRSR(z)\wt{\psi}\bigr)$ and $\bigl(\PlRS\psi,\tPlRS\wt{\psi}\bigr)$ are $\Phi_\rcv$-matched. Similarly, the matching in \cref{data:projectors-opposite} extends to every $\Phi_\rcv$-matched pair $\psi\in\Lrcv$, $\wt{\psi}\in\tLrcv$: for every such pair and every $\ell$, the functions $\PlSR\psi$ and $\tPlSR\wt{\psi}$ are $\Phi_\src$-matched.
  \end{enumerate}
\end{remark}

The five formulations serve different purposes. Condition \cref{data:wave} is the one accessible to measurement, and \cref{data:cross-spectral} is the form used in \cref{sec:Gl-maps}. The later sections use the compressed formulations in the $L^2$-extended form recorded in \cref{item:data-extension} of \cref{rem:data-equivalence}: \cref{data:projectors,data:projectors-opposite} in \cref{sec:uniqueness}, and \cref{data:resolvents} in \cref{sec:nonuniqueness}.

\subsection{The eigenfunction matching maps \texorpdfstring{$\Glsrc$ and $\Glrcv$}{G\_l\textasciicircum{}S and G\_l\textasciicircum{}R}}\label{sec:Gl-maps}

Equivalent cross-spectral data induce matching maps between the corresponding eigenspaces. Indeed, \cref{item:trace-spaces} of \cref{prop:trace-spaces} and unique continuation imply that every $\phi\in\El$ agrees with a unique element of $\tEl$ under each of the identifications $\Phi_\src$ and $\Phi_\rcv$. These two elements are not a priori known to coincide, nor is either a priori known to have the same global $L^2$-norm as $\phi$. The following \namecref{prop:Gl-maps} defines the resulting maps and records the relation between them.

\begin{proposition}\label{prop:Gl-maps}
  Assume that $(\M,g)$ and $(\tM,\tg)$ have equivalent cross-spectral data relative to $(\Phi_\src,\Phi_\rcv)$. Then, for each $\ell$, there exist unique maps $\Glsrc,\Glrcv\colon\El\to\tEl$ that satisfy
  \begin{equation}\label{eq:Gl-maps}
    \phi(x) = (\Glsrc\phi)(\Phi_\src(x))
    \quad\text{and}\quad
    \phi(y) = (\Glrcv\phi)(\Phi_\rcv(y))
  \end{equation}
  for every $\phi\in \El$, $x\in\src$, and $y\in\rcv$. These maps are linear isomorphisms, and they satisfy
  \begin{equation}\label{eq:Gl-adjoint}
    (\Glrcv)^\ast=(\Glsrc)^{-1},
  \end{equation}
  where the adjoint $(\Glrcv)^\ast$ is taken with respect to the inner products inherited from $\LM$ and $\tLM$. Furthermore, for every $\phi\in\El$,
  \begin{equation}\label{eq:Gl-norms}
    \tnormsrc[\big]{(\Glsrc\phi)|_{\tsrc}} = \normsrc{\phi|_\src}
    \quad\text{and}\quad
    \tnormrcv[\big]{(\Glrcv\phi)|_{\trcv}} = \normrcv{\phi|_\rcv}.
  \end{equation}
\end{proposition}
\begin{proof}
  Fix $\ell$. For any $\phi\in \El$, equation~\labelcref{eq:Gl-maps} determines the values of $\Glsrc\phi$ on the open set $\tsrc$, and hence $\Glsrc\phi\in\tEl$ itself by unique continuation. The same argument applies to $\Glrcv$. Thus the maps, if they exist, are unique.

  By \cref{item:trace-spaces} of \cref{prop:trace-spaces}, every $\phi\in\El$ satisfies $\phi|_\rcv=\wt{\phi}\circ\Phi_\rcv$ for some $\wt{\phi}\in\tEl$; this is the second identity in \cref{eq:trace-spaces}. Setting $\Glrcv\phi:=\wt{\phi}$, and defining $\Glsrc$ analogously from the first identity, gives maps satisfying \cref{eq:Gl-maps}. They are linear by uniqueness, injective by unique continuation, and surjective by the reverse inclusion in \cref{eq:trace-spaces}.

  To prove \cref{eq:Gl-adjoint}, we use two properties of the kernels recorded above: the reproducing identities~\labelcref{eq:K-reproducing} on $\M$ and on $\tM$; and, by \cref{item:K-spans} of \cref{prop:trace-spaces}, that the slices $K_\ell(x,\cdot)$, $x\in\src$, span $\El$, as do the slices $K_\ell(y,\cdot)$, $y\in\rcv$.

  Next, for $x\in\src$ and $y\in\rcv$,
  \begin{equation}\label{eq:Gl-on-slices}
    \Glrcv\bigl(K_\ell(x,\cdot)\bigr)=\wt{K}_\ell\bigl(\Phi_\src(x),\cdot\bigr)
    \quad\text{and}\quad
    \Glsrc\bigl(K_\ell(y,\cdot)\bigr)=\wt{K}_\ell\bigl(\Phi_\rcv(y),\cdot\bigr).
  \end{equation}
  Indeed, the right-hand sides belong to $\tEl$, and the equivalence of the cross-spectral data together with the symmetry of the kernels shows that they satisfy the corresponding matching identities in \cref{eq:Gl-maps}.

  As $\Glsrc$ is bijective, \cref{eq:Gl-adjoint} is equivalent to $\tinnerM{\Glrcv\phi}{\Glsrc\psi}=\innerM{\phi}{\psi}$ for all $\phi,\psi\in\El$. By the spanning property, it suffices to verify this for $\phi=K_\ell(x,\cdot)$ and $\psi=K_\ell(y,\cdot)$ with $x\in\src$ and $y\in\rcv$. For such $\phi$ and $\psi$, the identities~\labelcref{eq:Gl-on-slices}, the reproducing identities, and the equivalence of the data give
  \begin{align*}
    \tinnerM{\Glrcv\phi}{\Glsrc\psi}
    &= \tinnerM{\wt{K}_\ell(\Phi_\src(x),\cdot)}{\wt{K}_\ell(\Phi_\rcv(y),\cdot)}\\
    &= \wt{K}_\ell\bigl(\Phi_\src(x),\Phi_\rcv(y)\bigr)\\
    &= K_\ell(x,y)\\
    &= \innerM{\phi}{\psi}.
  \end{align*}

  Finally, \cref{eq:Gl-maps} and the measure-preserving property in \cref{rem:local-isometry} of \cref{rem:data-equivalence} give \cref{eq:Gl-norms}.
\end{proof}

Our main uniqueness result, \cref{prop:main-uniqueness}, is formulated in terms of the singular values of the receiver matching maps $\Glrcv$.

%% file: uniqueness.tex
\section{Uniqueness}\label{sec:uniqueness}

In this \namecref{sec:uniqueness}, we prove that, under a suitable spectral bound, the cross-spectral data of a manifold determine it uniquely up to a Riemannian isometry. By \cref{prop:data-summary}, the same conclusion holds under any of the equivalent conditions \crefrange{data:cross-spectral}{data:projectors-opposite}. The geometric core of our proof is a result of Helin, Lassas, Oksanen, and Saksala~\cite{MR3826551}. They consider the non-disjoint-data case; in \cref{sec:HLOS-balls}, we recast their geometric result in a form suitable for disjoint data. This reduces the inverse problem to determining, from the cross-spectral data, the ball-inclusion relations in \cref{item:HLOS-2} of \cref{prop:HLOS}.

The starting point for detecting the inclusion relations is the same as in~\cite{MR3826551}: we probe the part of the manifold outside $\src\cup\rcv$ with waves. At a fixed time, waves generated by suitably localized sources are supported in geodesic balls, but arbitrary $L^2$-functions supported in those balls can be produced only in the limit. With disjoint data, we cannot detect this limiting process directly from $L^2$-norms. Instead, we use wave mollification to construct localized functions in a spectrally weighted space $\HWM$ and test the waves against them. We define $\HWM$ in \cref{sec:HWM} and carry out this construction in \cref{sec:bump-functions}. In \cref{sec:ball-inclusion}, we use these functions to characterize the ball inclusions, and in \cref{sec:uniqueness-results}, we recover the inclusions from the data and prove our main uniqueness results.

The main analytic obstruction is that the eigenfunction matching maps $\Glrcv$ of \cref{prop:Gl-maps} need not preserve global $L^2$-norms. Although each $\Glrcv$ is a bounded isomorphism, its norm and that of its inverse need not be bounded uniformly in $\ell$. We compensate by choosing spectral weights subject to two competing requirements. They must make the functions in the two weighted spaces correspond and their $L^2$-pairings with waves agree. At the same time, they must grow slowly enough that $\HWM$ contains nonzero functions supported in arbitrary nonempty open subsets of $\M$. The latter requirement hinges on the decay of Fourier transforms of compactly supported functions on $\R$. We begin with this in \cref{sec:decay-bounds}.

\subsection{Admissible Fourier-decay bounds}\label{sec:decay-bounds}

For mollification purposes (see the proof of \cref{prop:wave-mollification}), we need a family of nonnegative functions with shrinking supports whose Fourier transforms decay in a controlled fashion. In fact, only the decay at the frequencies $\omega_\ell=\sqrt{\lambda_\ell}$ will matter. The convention used below for the Fourier transform of $\rho\in L^1(\R)$ is
\begin{equation*}
  \widehat{\rho}(\omega) := \int_\R e^{-i\omega x}\,\rho(x)\,dx.
\end{equation*}

\begin{definition}[Admissible Fourier-decay bounds]\label{def:decay-bound}
  Recall that $(\lambda_\ell)_\ell$ are the distinct eigenvalues of $-\Delta_g$ and that $\omega_\ell=\sqrt{\lambda_\ell}$. An \emph{admissible Fourier-decay bound} (for the manifold $(\M,g)$) is a sequence $(\beta_\ell)_\ell\subset\rinterval{0}{\infty}$ such that there exist functions $\rho_\delta\in L^1(\R)$ and constants $C_\delta>0$, indexed by $\delta\in\ointerval{0}{1}$, that satisfy
  \begin{enumerate}
    \item \label{item:decay-bound-rho}$\rho_\delta\ge 0$, $\int_\R\rho_\delta\,dx=1$, and $\supp\rho_\delta\subset\interval{-\delta}{\delta}$;
    \item \label{item:decay-bound-rho-hat} $\abs*{\widehat{\rho}_\delta(\omega_\ell)}\le C_\delta\,e^{-\beta_\ell}$ for all $\ell$.
  \end{enumerate}
\end{definition}
\begin{remark}\label{rem:same-sequences}
  If $(\M,g)$ and $(\tM,\tg)$ have equivalent cross-spectral data, then $\lambda_\ell=\wt{\lambda}_\ell$ for every $\ell$, and hence a sequence is an admissible Fourier-decay bound for $(\M,g)$ if and only if it is one for $(\tM,\tg)$.
\end{remark}

A supply of examples follows from Ingham's classical result.

\begin{proposition}[after Ingham~\cite{MR1574706}]\label{prop:Ingham}
  Fix $\delta>0$ and let $b\colon\rinterval{0}{\infty}\to\rinterval{0}{\infty}$ satisfy
  \begin{equation*}
    b(\omega^2)=\omega\,a(\omega),
    \qquad \omega\ge1,
  \end{equation*}
  where $a\colon\rinterval{1}{\infty}\to\ointerval{0}{\infty}$ is nonincreasing. Then the following are equivalent:
  \begin{enumerate}
    \item\label{item:Ingham-integral} The integral of $a(\omega)/\omega$ converges:
      \begin{equation}\label{eq:Ingham}
        \int_1^\infty \frac{a(\omega)}{\omega} \, d\omega < \infty.
      \end{equation}
    \item\label{item:Ingham-mollifier} There exist a constant $C>0$ and a function $\rho\in L^1(\R)$ such that $\rho\ge 0$, $\int_\R\rho\,dx=1$, $\supp\rho\subset\interval{-\delta}{\delta}$, and
      \begin{equation*}
        \abs*{\widehat{\rho}(\omega)}\le Ce^{-b(\omega^2)}=Ce^{-\abs{\omega}a(\abs{\omega})},
        \qquad \abs{\omega}\ge1.
      \end{equation*}
  \end{enumerate}

  In particular, if \cref{eq:Ingham} holds, then the sequence $(\beta_\ell)_\ell:=(b(\lambda_\ell))_\ell$ is an admissible Fourier-decay bound.
\end{proposition}

\begin{proof}
  The equivalence is the result of~\cite{MR1574706} together with minor observations about its proof. If $a$ does not tend to $0$, then $a\ge c$ for some $c>0$, and both sides of the equivalence fail: the integral in \cref{eq:Ingham} diverges, and a compactly supported $\rho\in L^1(\R)$ with $\abs*{\widehat{\rho}(\omega)}\le Ce^{-c\abs{\omega}}$ for $\abs{\omega}\ge1$ vanishes identically~\cite[p.~29]{MR1574706}, so no such $\rho$ exists. We may therefore assume $a(\omega)\to 0$ as $\omega\to\infty$.

  The necessity of~\labelcref{eq:Ingham} is stated and proved in~\cite{MR1574706}. For sufficiency, Ingham constructs an explicit nonzero function $\rho\in L^1(\R)$ such that $\abs{\widehat{\rho}(\omega)}\le Ce^{-\abs{\omega}a(\abs{\omega})}$ for $\abs{\omega}\ge 1$. The construction shows that $\rho$ is the pointwise limit, as $k\to\infty$, of the convolutions $h_1\ast\dots\ast h_k$, where $h_j\ge 0$. This implies $\rho\ge 0$, and $\int_\R\rho\,dx=1$ is obtained after normalization.

  The final assertion follows by applying the equivalence to each $\delta\in\ointerval{0}{1}$ and enlarging $C_\delta$ to cover the finitely many indices $\ell$ for which $\lambda_\ell<1$.
\end{proof}

Note that the construction in \cref{prop:Ingham} does not use any property of the eigenvalue sequence $(\lambda_\ell)_\ell$, since the Fourier-decay estimate holds for every real $\omega$ with $\abs{\omega}\ge 1$.

For every $\theta\in\rinterval{0}{1/2}$, the sequence $(\lambda_\ell^\theta)_\ell$ is an admissible Fourier-decay bound by \cref{prop:Ingham}, applied to $b(\lambda):=\lambda^\theta$. At the endpoint $\theta=1/2$, the corresponding function $a$ is identically $1$, and the integral in \cref{eq:Ingham} diverges.

\Cref{prop:Ingham} also yields admissible Fourier-decay bounds closer to the critical scale. For example, given $\varepsilon>0$, a nonnegative sequence $(\beta_\ell)_\ell$ is admissible whenever
\begin{equation*}
  \beta_\ell
  =\frac{\sqrt{\lambda_\ell}}
  {(\log\lambda_\ell)^{1+\varepsilon}}
\end{equation*}
holds for all sufficiently large $\ell$.

We will also use the following elementary closure properties of admissible Fourier-decay bounds.

\begin{lemma}\label{prop:decay-bound-properties}
  Let $\beta=(\beta_\ell)_\ell$ and $\wt{\beta}=(\wt{\beta}_\ell)_\ell$ be admissible Fourier-decay bounds for $(\M,g)$, let $\gamma=(\gamma_\ell)_\ell$ be a nonnegative sequence satisfying $\gamma_\ell\le\beta_\ell$ for every $\ell$, and let $C,c\ge0$. Then $\gamma$, $\beta+\wt{\beta}$, and $C\beta+c$ are all admissible Fourier-decay bounds (the operations are defined elementwise).
\end{lemma}

\begin{proof}
  For each $\delta\in\ointerval{0}{1}$, choose functions $\rho_\delta,\wt{\rho}_\delta\in L^1(\R)$ and constants $C_\delta,\wt{C}_\delta>0$ satisfying \cref{item:decay-bound-rho,item:decay-bound-rho-hat} of \cref{def:decay-bound} with respect to the sequences $\beta$ and $\wt{\beta}$, respectively.

  The same function $\rho_\delta$ with the same constant $C_\delta$ satisfies \cref{item:decay-bound-rho-hat} of \cref{def:decay-bound} with respect to $\gamma$ as well, so $\gamma$ is an admissible Fourier-decay bound.

  Put $\sigma_\delta := \rho_{\delta/2}\ast\wt{\rho}_{\delta/2}$. By the standard properties of convolutions, the function $\sigma_\delta\in L^1(\R)$ satisfies \cref{item:decay-bound-rho} of \cref{def:decay-bound}, and
  \begin{equation*}
    \abs*{\widehat{\sigma}_\delta (\omega_\ell)}
    = \abs*{\widehat{\rho}_{\delta/2}(\omega_\ell)\,\widehat{\wt{\rho}}_{\delta/2}(\omega_\ell)}
    \le C_{\delta/2}\wt{C}_{\delta/2}\, e^{-(\beta_\ell+\wt{\beta}_\ell)},
  \end{equation*}
  so it satisfies \cref{item:decay-bound-rho-hat} as well. Thus, $\beta+\wt{\beta}$ is admissible.

  For an additive constant $c\ge 0$, $\abs[\big]{\widehat{\rho}_{\delta}(\omega_\ell)}\le C_\delta e^{-\beta_\ell}=C_\delta e^{c}\, e^{-(\beta_\ell+c)}$, so $\beta+c$ is admissible.

  Finally, choose a positive integer $k\ge C$. Then $k\beta$ is admissible (a $k$-fold sum of $\beta$) and $C\beta\le k\beta$, hence so is $C\beta$, and therefore so is $C\beta + c$.
\end{proof}

\subsection{The spectrally weighted space \texorpdfstring{$\HWM$}{H\_W(M)}}\label{sec:HWM}

We define a subspace $\HWM\subset\LM$ by weighting the eigenspaces $\El$ separately. If the weights are uniformly bounded from below, then $\HWM$ is a Hilbert space under its weighted inner product. When two manifolds have equivalent cross-spectral data and spectral weight sequences related by \cref{eq:widetilde-W} below, the eigenfunction matching maps $\Glrcv$ of \cref{prop:Gl-maps} induce an isometric isomorphism between the weighted spaces that preserves $L^2$-pairings with wave solutions.

\begin{definition}[Spectrally weighted space $\HWM$]\label{def:HWM}
  Let $(\M,g)$ be as in \cref{ass:standing}.
  \begin{enumerate}
    \item A \emph{spectral weight sequence} for $(\M,g)$ is a sequence $W=(W_\ell)_\ell$ such that each $W_\ell\in\calB(\El)$ is self-adjoint and positive definite.

    \item For a spectral weight sequence $W$ and $\psi\in\LM$, define $\normW{\psi}\in\interval{0}{\infty}$ by
      \begin{equation*}
        \normW{\psi}^2
        := \sum_{\ell\ge 0} \innerM{W_\ell \psi_\ell}{\psi_\ell},
      \end{equation*}
      where $\psi_\ell := P_\ell\psi \in \El$.
    \item For a spectral weight sequence $W$, the \emph{spectrally weighted space} $\HWM$ consists of those functions $\psi\in\LM$ for which $\normW{\psi}<\infty$.
  \end{enumerate}
\end{definition}

The following elementary properties of $\HWM$ are standard; we omit the proofs.

\begin{proposition}\label{prop:HW}
  Let $W=(W_\ell)_\ell$ be a spectral weight sequence for $(\M,g)$.
  \begin{enumerate}
    \item $\HWM$ is an inner product space under the pairing induced by $\normW{\cdot}$: $\inner{\psi}{\psi'}_W:=\sum_{\ell\ge 0}\innerM{W_\ell\psi_\ell}{\psi'_\ell}$, where $\psi_\ell:=P_\ell\psi$ and $\psi'_\ell:=P_\ell\psi'$; the series converges absolutely.
    \item\label{item:HW-Hilbert} If $W$ is uniformly bounded from below, i.e., there exists $c>0$ such that $W_\ell\ge c\,\id_{\El}$ for every $\ell$, then $\HWM$ is a Hilbert space that embeds continuously into $\LM$.
  \end{enumerate}
\end{proposition}

\subsubsection*{Relating spectral weight sequences.}
Let $\wt{W}=(\wt{W}_\ell)_\ell$ be a spectral weight sequence for $(\tM,\tg)$, and assume that $(\M,g)$ and $(\tM,\tg)$ have equivalent cross-spectral data. Define $W=(W_\ell)_\ell$ by
\begin{equation}\label{eq:widetilde-W}
  W_\ell := (\Glrcv)^\ast \wt{W}_\ell \Glrcv
  \in \calB(\El).
\end{equation}
Then $W$ is a spectral weight sequence for $(\M,g)$.

We record how the eigenfunction matching maps $\Glrcv$ and $\Glsrc$ act on waves generated by matched sources.

\begin{lemma}\label{prop:Gl-on-wave}
  Assume that the manifolds $(\M,g)$ and $(\tM,\tg)$ have equivalent cross-spectral data. Fix $T>0$.

  For matched sources $s\in L^1(\interval{0}{T};\Lsrc)$ and $\tilde{s}\in L^1(\interval{0}{T};\tLsrc)$, the corresponding wave solutions satisfy, for all $\ell$,
  \begin{equation}\label{eq:Glrcv-on-wave}
    \Glrcv P_\ell v^s(T)
    = \wt{P}_\ell\wt{v}^{\tilde{s}}(T).
  \end{equation}
  For matched sources $r\in L^1(\interval{0}{T};\Lrcv)$ and $\tilde{r}\in L^1(\interval{0}{T};\tLrcv)$, the corresponding wave solutions satisfy, for all $\ell$,
  \begin{equation}\label{eq:Glsrc-on-wave}
    \Glsrc P_\ell v^r(T)
    = \wt{P}_\ell\wt{v}^{\tilde{r}}(T).
  \end{equation}
\end{lemma}

\begin{proof}
  It suffices to prove \cref{eq:Glrcv-on-wave}, since \cref{eq:Glsrc-on-wave} follows from the same argument with the roles of the source and receiver sets reversed, using \cref{data:projectors-opposite} and its $L^2$-extension in \cref{item:data-extension} of \cref{rem:data-equivalence}.

  Fix $\ell$, and consider matched $\psi\in\Lsrc$ and $\wt{\psi}\in\tLsrc$. By \cref{item:data-extension} of \cref{rem:data-equivalence}, the functions $\PlRS\psi$ and $\tPlRS\wt{\psi}$ are matched. Combining this with \cref{eq:Gl-maps}, applied to $\phi=P_\ell\psi$, shows that $\wt{P}_\ell\wt{\psi}$ and $\Glrcv P_\ell\psi$ agree on $\trcv$, hence everywhere by unique continuation. Therefore, $\Glrcv P_\ell \psi = \wt{P}_\ell\wt{\psi}$. Using this, \cref{eq:wave-solution}, and $\lambda_\ell=\wt{\lambda}_\ell$, we get \cref{eq:Glrcv-on-wave}:
  \begin{align*}
    \Glrcv P_\ell v^s(T)
    &= \Glrcv P_\ell \int_0^T \sinprop_{T-\tau}s(\tau)\,d\tau\\
    &= \int_0^T \sigma_\ell(T-\tau) \wt{P}_\ell \tilde{s}(\tau)\,d\tau\\
    &= \wt{P}_\ell\wt{v}^{\tilde{s}}(T).\qedhere
  \end{align*}
\end{proof}

The relation~\labelcref{eq:widetilde-W} for the spectral weight sequences compensates for the failure of $\Glrcv$ and $\Glsrc$ to preserve the global $L^2$-norms: it allows the maps $\Glrcv$ to assemble into an isometric isomorphism between the weighted spaces. Moreover, the adjoint relation~\labelcref{eq:Gl-adjoint} and \cref{prop:Gl-on-wave} imply that this isomorphism preserves $L^2$-pairings with waves generated by matched sources in $\rcv$. This is the mechanism used in the proof of \cref{prop:main-uniqueness} to transfer the ball-inclusion criterion of \cref{prop:inf-balls} between the two manifolds.

\begin{proposition}\label{prop:HW-G-rcv}
  Assume that the manifolds $(\M,g)$ and $(\tM,\tg)$ have equivalent cross-spectral data, and let $W=(W_\ell)_\ell$ and $\wt{W}=(\wt{W}_\ell)_\ell$ be spectral weight sequences for $(\M,g)$ and $(\tM,\tg)$, respectively. Assume that they are related by \cref{eq:widetilde-W} and that both are uniformly bounded from below (see \cref{item:HW-Hilbert} of \cref{prop:HW} for the definition).
  \begin{enumerate}
    \item\label{item:G-rcv} Let $\psi\in\HWM$ and write $\psi = \sum_{\ell\ge 0} \psi_\ell$, where $\psi_\ell\in \El$. The series
      \begin{equation}\label{eq:HW-Grcv}
        \Grcv\,\psi := \sum_{\ell\ge 0} \Glrcv \psi_\ell
      \end{equation}
      is orthogonal in the Hilbert space $\tHWM$, converges in its norm, and defines an isometric isomorphism $\Grcv\colon\HWM\to\tHWM$.
    \item\label{item:HW-Grcv-inner-product} For every $T>0$, every $\psi\in\HWM$, and every matched pair of sources $r\in\Lloc{\rinterval{0}{\infty};\Lrcv}$ and $\tilde{r}\in\Lloc{\rinterval{0}{\infty};\tLrcv}$, with corresponding solutions $v^r$ and $\wt{v}^{\tilde{r}}$ to the wave equations on $\M$ and $\tM$, respectively,
      \begin{equation}\label{eq:HW-Grcv-inner-product}
        \innerM{\psi}{v^r(T)}
        = \tinnerM{\Grcv\psi}{\wt{v}^{\tilde{r}}(T)}.
      \end{equation}
  \end{enumerate}
\end{proposition}
\begin{proof}
  \begin{enumerate}
    \item The series is orthogonal, because $\Glrcv$ maps into $\tEl$, and $\tEl\perp\wt{E}_{\ell'}$ in $\tHWM$ if $\ell\neq\ell'$. It converges, since
      \begin{align*}
        \tnormW[\big]{\Glrcv\,\psi_\ell}^2
        = \tinnerM{\wt{W}_\ell \Glrcv\,\psi_\ell}{\Glrcv\,\psi_\ell}
        = \innerM{W_\ell\psi_\ell}{\psi_\ell},
      \end{align*}
      and summing over $\ell$ yields $\tnormW[\big]{\Grcv\psi}^2=\normW{\psi}^2<\infty$, which also proves that $\Grcv$ is norm-preserving. Since $\HWM$ is complete, the range of $\Grcv$ is closed. It is also dense, since it contains every eigenspace $\tEl$; hence $\Grcv$ is an isometric isomorphism.

    \item Let $\psi\in\HWM$ and write $\psi=\sum_{\ell\ge 0} \psi_\ell$ with $\psi_\ell\in \El$. By the definition of $\Grcv$, \cref{eq:Glsrc-on-wave}, and \cref{eq:Gl-adjoint}, we obtain
      \begin{align*}
        \tinnerM{\Grcv\psi}{\wt{v}^{\tilde{r}}(T)}
        &= \sum_{\ell\ge 0} \tinnerM{\Glrcv\psi_\ell}{\wt{P}_\ell\,\wt{v}^{\tilde{r}}(T)}\\
        &= \sum_{\ell\ge 0} \tinnerM{\Glrcv\psi_\ell}{\Glsrc P_\ell\,v^{r}(T)}\\
        &= \sum_{\ell\ge 0} \innerM{\psi_\ell}{P_\ell\,v^{r}(T)}\\
        &= \innerM{\psi}{v^{r}(T)}.\qedhere
      \end{align*}
  \end{enumerate}
\end{proof}

\subsection{Bump functions in \texorpdfstring{$\HWM$}{H\_W(M)}}\label{sec:bump-functions}

In this subsection, we show that if a spectral weight sequence $W$ for $(\M,g)$ satisfies a suitable upper bound, then $\HWM$ contains bump functions supported in any prescribed nonempty open subset of $\M$. These functions are used in \cref{sec:ball-inclusion} to test inclusion relations between geodesic balls.

We construct these bump functions by wave mollification~\cite[Equation~(1.4)]{MR658471}, using the mollifiers $\rho_\delta$ associated with an admissible Fourier-decay bound (\cref{def:decay-bound}); see~\cite[Sections~2--3]{MR3112375} for related wave-equation regularizations.
\begin{proposition}\label{prop:wave-mollification}
  Let $(\beta_\ell)_\ell$ be an admissible Fourier-decay bound, and choose associated functions $\rho_\delta\in L^1(\R)$ and constants $C_\delta>0$, $\delta\in\ointerval{0}{1}$, as in \cref{def:decay-bound}. Fix a spectral weight sequence $W=(W_\ell)_\ell$ for $(\M,g)$ and assume that it satisfies, for every $\ell$, the bound
  \begin{equation}\label{eq:beta-bounded-W}
    W_\ell \le e^{2\beta_\ell}\,\id_{\El}.
  \end{equation}

  For every $\psi\in\LM$ and $\delta\in\ointerval{0}{1}$, the Bochner integral
  \begin{equation}
    \label{eq:wave-mollification}
    \psi_\delta := \int_\R \rho_\delta(t)\, \cosprop_t\psi\,dt \in\LM
  \end{equation}
  exists, and
  \begin{enumerate}
    \item\label{item:wave-mollification-space} $\psi_\delta\in\HWM$,
    \item\label{item:wave-mollification-convergence} $\psi_\delta\to\psi$ in $\LM$ as $\delta\to 0^+$, and
    \item\label{item:wave-mollification-support} $\supp\psi_\delta\subset\{x\in\M : \dist(x,\supp\psi)\le\delta\}$.
  \end{enumerate}
\end{proposition}
\begin{proof}
  Since $t\mapsto \cosprop_t\psi$ is continuous and bounded by $\normM{\psi}$, and $\rho_\delta\in L^1(\R)$, the Bochner integral~\labelcref{eq:wave-mollification} exists. Recall that $\omega_\ell=\sqrt{\lambda_\ell}$. For every $\ell$,
  \begin{equation}
    \label{eq:wave-mollification-2}
    P_\ell\psi_\delta
    = \left(\int_\R\rho_\delta(t)\cos(\omega_\ell t)\,dt\right) P_\ell\psi
    = \re\bigl(\widehat{\rho}_\delta(\omega_\ell)\bigr)\,P_\ell\psi,
  \end{equation}
  so, by \cref{item:decay-bound-rho-hat} of \cref{def:decay-bound},
  \begin{equation*}
    \normM{P_\ell\psi_\delta}
    = \abs*{\re\bigl(\widehat{\rho}_\delta(\omega_\ell)\bigr)}\,\normM{P_\ell\psi}
    \le C_\delta\,e^{-\beta_\ell}\,\normM{P_\ell\psi}.
  \end{equation*}
  Therefore,
  \begin{align*}
    \normW{\psi_\delta}^2
    &= \sum_{\ell\ge 0} \innerM{W_\ell P_\ell\psi_\delta}{P_\ell\psi_\delta}\\
    &\le \sum_{\ell\ge 0} e^{2\beta_\ell} \normM{P_\ell\psi_\delta}^2\\
    &\le C_\delta^2\sum_{\ell\ge 0} \normM{P_\ell\psi}^2\\
    &= C_\delta^2\normM{\psi}^2\\
    &< \infty,
  \end{align*}
  and $\psi_\delta\in\HWM$. This proves the first item.

  The second item follows from the strong continuity of $\cosprop$ at $t=0$. Indeed, by \cref{item:decay-bound-rho} of \cref{def:decay-bound},
  \begin{equation*}
    \normM{\psi_\delta-\psi}
    \le \int_\R \rho_\delta(t)\,\normM[\big]{\cosprop_t\psi-\psi}\,dt
    \le \sup_{\abs{t}\le\delta}\normM[\big]{\cosprop_t\psi-\psi}
    \to 0
  \end{equation*}
  as $\delta\to 0^+$.

  By the finite speed of wave propagation~\cite[Chapter~6, Proposition~1.3]{MR2744150}, the third item follows, since $\rho_\delta(t)$ vanishes for $\abs{t}>\delta$.
\end{proof}

\begin{corollary}\label{prop:HW-dense-in-U}
  Let $U\subset\M$ be a nonempty open set and $W=(W_\ell)_\ell$ a spectral weight sequence for $(\M,g)$. Suppose that \cref{eq:beta-bounded-W} holds for some admissible Fourier-decay bound $(\beta_\ell)_\ell$. Then $\HWM\cap\Cc{U}$ is dense in $L^2(U)$.
\end{corollary}
\begin{proof}
  Let $\psi\in L^2(U)$ and $\varepsilon>0$. Choose $\chi\in\Cc{U}$ such that $\norm{\psi-\chi}_{L^2(U)}<\varepsilon/2$. Let $(\beta_\ell)_\ell$ be as in the hypothesis and, for $\delta\in\ointerval{0}{1}$, let $\chi_\delta$ be the wave mollification of $\chi$ defined in \cref{prop:wave-mollification}. By \cref{item:wave-mollification-convergence} of \cref{prop:wave-mollification}, $\normM{\chi-\chi_\delta}<\varepsilon/2$ for all sufficiently small $\delta>0$. Since $\supp\chi$ is a compact subset of $U$, \cref{item:wave-mollification-support} of \cref{prop:wave-mollification} also gives $\supp\chi_\delta\subset U$ for all sufficiently small $\delta>0$.

  Moreover, each $\chi_\delta$ is smooth: the normalization and nonnegativity of $\rho_\delta$ give $\abs{\widehat{\rho}_\delta}\le 1$, so \cref{eq:wave-mollification-2} and $\chi\in\Cc{U}$ imply that $\chi_\delta\in H^k(\M)$ for every integer $k\ge0$. Fixing a small enough $\delta>0$, we have $\chi_\delta\in\HWM\cap\Cc{U}$ by \cref{item:wave-mollification-space} of \cref{prop:wave-mollification}, and the triangle inequality gives $\norm{\psi-\chi_\delta|_U}_{L^2(U)}<\varepsilon$. Since $\varepsilon>0$ was arbitrary, this proves the claim.
\end{proof}

\subsection{Ball inclusion and the singular-value condition}\label{sec:ball-inclusion}

The following proposition extends~\cite[Lemma~10]{MR3826551} to allow for disjoint $\src$ and $\rcv$, and is a key ingredient for our main uniqueness result (\cref{prop:main-uniqueness}). When $\src=\rcv$, the Blagoveščenskiĭ identity~\cite{MR282060,MR298226} (see also~\cite[Theorem~12]{MR3826551}) computes the norm of a wave directly from the data. This makes it possible to detect convergence of a sequence of waves. For disjoint $\src$ and $\rcv$, only the inner product of a pair of waves---one sourced on $\src$ and observed on $\rcv$, the other sourced on $\rcv$ and observed on $\src$---is computable from the data~\cite[Lemma~3.2]{MR4833220}. However, detecting even weak convergence of a sequence of waves by testing against other waves runs into the difficulty that, in general, the reachable waves are only dense in $L^2$ of their domain of influence. Concluding convergence therefore requires knowing that the sequence is bounded, but that is not known a priori.

The key idea below is to sidestep these difficulties by testing waves against functions in suitably related weighted spaces $\HWM$ and $\tHWM$. Their spectral weight sequences must balance two requirements: they must be uniformly bounded from below, so that the eigenfunction matching maps induce an isometric isomorphism that preserves $L^2$-pairings with matched waves (see \cref{prop:HW-G-rcv}), but also from above in terms of an admissible Fourier-decay bound, so that the spaces contain nonzero functions supported in arbitrary nonempty open subsets of $\M$ and $\tM$, respectively (see \cref{prop:HW-dense-in-U}). The ball-inclusion criterion in \cref{prop:inf-balls} below uses only the upper bound; the lower bound enters in \cref{sec:uniqueness-results}, where the two manifolds are compared.

Let $B_{\M}(x,\rho)\subset\M$ denote the open geodesic ball of radius $\rho>0$ centered at $x\in\M$; for $\tx\in\tM$, define $B_{\tM}(\tx,\rho)\subset\tM$ analogously.

\begin{proposition}\label{prop:inf-balls}
  Let $(\beta_\ell)_\ell$ be an admissible Fourier-decay bound and $W=(W_\ell)_\ell$ a spectral weight sequence for $(\M,g)$, and assume that, for every $\ell$,
  \begin{equation}\label{eq:W-assumption}
    W_\ell \le e^{2\beta_\ell}\,\id_{\El}.
  \end{equation}
  Let $x_j\in\rcv$ and $\rho_j\in\ointerval{0}{\infty}$, $j=0,1,2$, and let $\varepsilon,T>0$ be such that $B_{\M}(x_j,\varepsilon)\subset\rcv$ and $\rho_j-\varepsilon\in\ointerval{0}{T}$ for $j=0,1,2$. For $j=0,1,2$, set
  \begin{equation}\label{eq:calR-j}
    \calR_j := \Cc[\big]{\ointerval{T-(\rho_j-\varepsilon)}{T} \times B_{\M}(x_j,\varepsilon)}
    \subset \Cc[\big]{\ointerval{0}{T}\times\rcv}.
  \end{equation}

  The inclusion
  \begin{equation}\label{eq:ball-inclusion}
    B_{\M}(x_0,\rho_0) \subset \overline{B_{\M}(x_1,\rho_1) \cup B_{\M}(x_2,\rho_2)}
  \end{equation}
  holds if and only if, for every source $r_0\in\calR_0$ and every $\psi\in\HWM$,
  \begin{equation}\label{eq:ball-inf}
    \inf\left\{ \abs*{\innerM{\psi}{v^{r_0}(T) - v^{r_1+r_2}(T)}}
    : r_1\in\calR_1,\, r_2\in\calR_2 \right\} = 0.
  \end{equation}
\end{proposition}
\begin{proof}
  Denote $B_j:=B_{\M}(x_j,\rho_j)$ for $j=0,1,2$. By the finite speed of wave propagation~\cite[Chapter~6, Proposition~1.3]{MR2744150} and \cref{eq:wave-solution}, we have $\supp v^{r_j}(T)\subset B_j$ for every $j$ and every source $r_j\in\calR_j$. Moreover, approximate controllability~\cite[Theorem~11]{MR3826551}, which follows from Tataru's unique continuation theorem~\cite{MR1326909}, implies that the set
  \begin{equation}\label{eq:approximate-controllability}
    \{v^{r_j}(T)|_{\overline{B}_j} : r_j\in\calR_j\}
  \end{equation}
  is dense in $L^2(\overline{B}_j)$. The formulations with open and closed balls are equivalent, since geodesic spheres have measure zero~\cite[Lemma~3.2]{MR2825736}.

  First assume that inclusion~\labelcref{eq:ball-inclusion} holds, and fix a source $r_0\in\calR_0$, a function $\psi\in\HWM$, and $\delta>0$. Then
  \begin{equation*}
    \supp v^{r_0}(T)
    \subset B_0
    \subset \overline{B_1 \cup B_2}
    = \overline{B}_1 \cup \overline{B}_2.
  \end{equation*}
  Writing $\gamma_1$ and $\gamma_2$ for the restrictions of $v^{r_0}(T)$ to $\overline{B}_1$ and to $\overline{B}_2\setminus\overline{B}_1$, extended by zero, we have $v^{r_0}(T)=\gamma_1+\gamma_2$ with $\gamma_j\in L^2(\overline{B}_j)$. By the density of the set~\labelcref{eq:approximate-controllability} in $L^2(\overline{B}_j)$, we can choose $r_j\in\calR_j$, $j=1,2$, such that
  \begin{equation*}
    \normM{v^{r_0}(T) - v^{r_1+r_2}(T)}\le\normM{\gamma_1 - v^{r_1}(T)}+\normM{\gamma_2 - v^{r_2}(T)}<\delta.
  \end{equation*}
  By Cauchy--Schwarz, the infimum~\labelcref{eq:ball-inf} is at most $\delta\normM{\psi}$ for every $\delta>0$, hence zero.

  For the reverse implication, suppose that inclusion~\labelcref{eq:ball-inclusion} does not hold. Set $U := B_0\setminus \overline{B_1 \cup B_2}$ and observe that $U$ is a nonempty open set. By the density of the set~\labelcref{eq:approximate-controllability} in $L^2(\overline{B}_0)$, we can choose a source $r_0\in\calR_0$ such that $v^{r_0}(T)|_U\neq 0$ in $L^2(U)$. By \cref{eq:W-assumption} and \cref{prop:HW-dense-in-U}, we can find $\psi\in\HWM$ with $\supp\psi\subset U$ and $\innerM{\psi}{v^{r_0}(T)} \neq 0$.

  Now let $r_1\in\calR_1$ and $r_2\in\calR_2$ be any sources. Then $\supp v^{r_1+r_2}(T)\subset B_1\cup B_2$, so $\supp v^{r_1+r_2}(T) \cap \supp\psi = \emptyset$, and it follows that
  \begin{align*}
    \innerM{\psi}{v^{r_0}(T) - v^{r_1+r_2}(T)}
    &= \innerM{\psi}{v^{r_0}(T)}
    - \innerM{\psi}{v^{r_1+r_2}(T)}\\
    &= \innerM{\psi}{v^{r_0}(T)}\\
    &\neq 0.
  \end{align*}
  Hence, the infimum~\labelcref{eq:ball-inf} is strictly positive.
\end{proof}

Let $\sv(L)$ denote the set of singular values of an operator $L$. We next characterize, in terms of singular values, when the upper and lower bounds on the two spectral weight sequences can be imposed simultaneously.

\begin{proposition}\label{prop:main-assumption-iff}
  Let $(\M,g)$ and $(\tM,\tg)$ be as in \cref{ass:standing}, and assume that they have equivalent cross-spectral data.

  For every admissible Fourier-decay bound $(\beta_\ell)_\ell$, the following are equivalent:
  \begin{enumerate}
    \item\label{item:iff-sv} For every $\ell$,
      \begin{equation}\label{eq:sg-Glrcv}
        \sv(\Glrcv) \subset \interval{e^{-\beta_\ell}}{e^{\beta_\ell}}.
      \end{equation}
    \item\label{item:iff-sv-2} For every $\ell$,
      \begin{equation}\label{eq:sg-Glsrc}
        \sv(\Glsrc) \subset \interval{e^{-\beta_\ell}}{e^{\beta_\ell}}.
      \end{equation}
    \item\label{item:iff-bounds} There exist spectral weight sequences $W=(W_\ell)_\ell$ and $\wt{W}=(\wt{W}_\ell)_\ell$ for $(\M,g)$ and $(\tM,\tg)$, respectively, that are related by \cref{eq:widetilde-W} and satisfy, for every $\ell$,
      \begin{equation}\label{eq:W-bounds}
        \id_{\El} \le W_\ell \le e^{2\beta_\ell}\,\id_{\El}
        \quad\text{and}\quad
        \id_{\tEl} \le \wt{W}_\ell \le e^{2\beta_\ell}\,\id_{\tEl}.
      \end{equation}
  \end{enumerate}
\end{proposition}
\begin{proof}
  Fix an admissible Fourier-decay bound $(\beta_\ell)_\ell$ and an index $\ell$, and set $A:=(\Glrcv)^\ast\Glrcv\in\calB(\El)$ and $B:=\Glrcv(\Glrcv)^\ast\in\calB(\tEl)$, both self-adjoint and positive definite.

  (\labelcref{item:iff-sv} $\Leftrightarrow$ \labelcref{item:iff-sv-2}) This follows immediately from $(\Glsrc)^\ast \Glsrc = \bigl((\Glrcv)^\ast\Glrcv\bigr)^{-1}$ (see \cref{eq:Gl-adjoint}) and the fact that the interval $\interval{e^{-\beta_\ell}}{e^{\beta_\ell}}$ is closed under reciprocation.

  (\labelcref{item:iff-sv} $\Rightarrow$ \labelcref{item:iff-bounds}) Suppose \cref{eq:sg-Glrcv} holds. Then
  \begin{equation*}
    e^{-2\beta_\ell}\, \id_{\El}
    \le A
    \le e^{2\beta_\ell}\, \id_{\El}
    \quad\text{and}\quad
    e^{-2\beta_\ell}\, \id_{\tEl}
    \le B
    \le e^{2\beta_\ell}\, \id_{\tEl}.
  \end{equation*}
  Define $h,\wt{h}\colon\ointerval{0}{\infty}\to\ointerval{0}{\infty}$ by $h(\lambda)=\max(1,\lambda)$ and $\wt{h}(\lambda)=\max(1,1/\lambda)$, and set $W_\ell := h(A)$ and $\wt{W}_\ell := \wt{h}(B)$. Since $h$ and $\wt{h}$ map $\interval{e^{-2\beta_\ell}}{e^{2\beta_\ell}}$ into $\interval{1}{e^{2\beta_\ell}}$, both inequalities~\labelcref{eq:W-bounds} hold.

  Since $B\Glrcv = \Glrcv A$, induction gives $B^k\Glrcv = \Glrcv A^k$ for every $k\ge0$, and hence $p(B)\Glrcv = \Glrcv p(A)$ for every polynomial $p$. The set $\spec(A)\cup\spec(B)\subset\ointerval{0}{\infty}$ is finite, so we may choose a polynomial $p$ that agrees with $\wt{h}$ on it. Then $p(A)=\wt{h}(A)$ and $p(B)=\wt{h}(B)$, so $\wt{h}(B)\Glrcv = \Glrcv \wt{h}(A)$ and \cref{eq:widetilde-W} holds:
  \begin{equation*}
    (\Glrcv)^\ast\wt{W}_\ell \Glrcv
    = (\Glrcv)^\ast \wt{h}(B) \Glrcv
    = A \wt{h}(A)
    = h(A)
    = W_\ell,
  \end{equation*}
  where the third equality follows from $\lambda\max(1,1/\lambda)=\max(1,\lambda)$ for $\lambda>0$.

  (\labelcref{item:iff-bounds} $\Rightarrow$ \labelcref{item:iff-sv}) The map $X\mapsto (\Glrcv)^\ast X\Glrcv$ preserves order, so from $\wt{W}_\ell\ge \id_{\tEl}$ and $W_\ell\le e^{2\beta_\ell}\id_{\El}$ it follows that
  \begin{equation*}
    A \le (\Glrcv)^\ast \wt{W}_\ell \Glrcv
    = W_\ell
    \le e^{2\beta_\ell} \id_{\El}.
  \end{equation*}
  From $W_\ell\ge \id_{\El}$ and $\wt{W}_\ell\le e^{2\beta_\ell}\id_{\tEl}$, we similarly obtain
  \begin{equation*}
    \id_{\El}
    \le W_\ell
    = (\Glrcv)^\ast \wt{W}_\ell \Glrcv
    \le e^{2\beta_\ell} A.
  \end{equation*}
  It follows that $\spec(A)\subset \interval{e^{-2\beta_\ell}}{e^{2\beta_\ell}}$; hence $\sv(\Glrcv)\subset\interval{e^{-\beta_\ell}}{e^{\beta_\ell}}$.
\end{proof}

\subsection{Uniqueness results}\label{sec:uniqueness-results}

In this subsection, we combine the preceding results to show that two connected, closed Riemannian manifolds with equivalent cross-spectral data are isometric, provided that the singular values of the eigenfunction matching maps $\Glrcv$ defined in \cref{prop:Gl-maps} are controlled by an admissible Fourier-decay bound. As a consequence, we obtain uniqueness from lower bounds for the $L^2$-masses of eigenfunctions imposed either on the receiver sets of both manifolds or on the source and receiver sets of a single manifold. By \cref{prop:data-summary}, these results hold under any of the equivalent conditions \crefrange{data:cross-spectral}{data:projectors-opposite}.

\begin{theorem}\label{prop:main-uniqueness}
  Let $(\M,g)$ and $(\tM,\tg)$ be Riemannian manifolds as in \cref{ass:standing}, and assume that they have equivalent cross-spectral data. The following conditions are equivalent:
  \begin{cond}
  \item\label{cond:main-inclusion} There exists an admissible Fourier-decay bound $(\beta_\ell)_\ell$ for $(\M,g)$ such that, for every $\ell$, the singular values of $\Glrcv$ satisfy
    \begin{equation}\label{eq:main-inclusion}
      \sv(\Glrcv) \subset \interval{e^{-\beta_\ell}}{e^{\beta_\ell}}.
    \end{equation}

  \item\label{cond:main-isometry} There exists a Riemannian isometry $\Phi\colon (\M,g)\to (\tM,\tg)$ such that $\Phi|_\src=\Phi_\src$ and $\Phi|_\rcv=\Phi_\rcv$.
  \end{cond}

  If either condition holds, then every Riemannian isometry $\Phi$ as in \cref{cond:main-isometry} satisfies $\Glrcv\phi=\phi\circ\Phi^{-1}$ for every $\ell$ and $\phi\in\El$; consequently, $\sv(\Glrcv)=\{1\}$ for every $\ell$.
\end{theorem}

\begin{proof}
  (\labelcref{cond:main-inclusion} $\Rightarrow$ \labelcref{cond:main-isometry}) We show that \cref{item:HLOS-2} of \cref{prop:HLOS} holds for the manifolds $(\M,g)$ and $(\tM,\tg)$ with $U=\rcv$, $\wt{U}=\trcv$, and $\Phi_U=\Phi_\rcv$. By \cref{prop:main-assumption-iff}, we can find spectral weight sequences $W$ and $\wt{W}$ for $(\M,g)$ and $(\tM,\tg)$, respectively, that are related by \cref{eq:widetilde-W} and satisfy \cref{eq:W-bounds}.

  Fix $x_j\in\rcv$ and $\rho_j\in\ointerval{0}{\infty}$ for $j=0,1,2$, and define $\tx_j:=\Phi_\rcv(x_j)\in\trcv$. Choose $\varepsilon,T>0$ such that $B_{\M}(x_j,\varepsilon)\subset\rcv$, $B_{\tM}(\tx_j,\varepsilon)\subset\trcv$, and $\rho_j-\varepsilon\in\ointerval{0}{T}$ for every $j$. Let $\calR_j$ be defined by \cref{eq:calR-j} and set
  \begin{equation*}
    \tcalR_j := \Cc[\big]{\ointerval{T-(\rho_j-\varepsilon)}{T} \times B_{\tM}(\tx_j,\varepsilon)}
    \subset \Cc[\big]{\ointerval{0}{T}\times\trcv}.
  \end{equation*}
  With these choices, \cref{prop:inf-balls} applies to both $(\M,g)$ and $(\tM,\tg)$; in the latter case, admissibility of $(\beta_\ell)_\ell$ follows from \cref{rem:same-sequences}.

  We need to show that
  \begin{equation}\label{eq:main-proof-iff-1}
    B_{\M}(x_0,\rho_0) \subset \overline{B_{\M}(x_1,\rho_1) \cup B_{\M}(x_2,\rho_2)}
  \end{equation}
  if and only if
  \begin{equation}\label{eq:main-proof-iff-2}
    B_{\tM}(\tx_0,\rho_0) \subset \overline{B_{\tM}(\tx_1,\rho_1) \cup B_{\tM}(\tx_2,\rho_2)}.
  \end{equation}
  To compare these inclusions, define functions $\alpha\colon \calR_0\times\HWM\to\rinterval{0}{\infty}$ and $\wt{\alpha}\colon \tcalR_0\times\tHWM\to\rinterval{0}{\infty}$ by
  \begin{align*}
    \alpha(r_0,\psi)
    &:= \inf\left\{ \abs*{\innerM{\psi}{v^{r_0}(T) - v^{r_1+r_2}(T)}}
    : r_1\in\calR_1,\, r_2\in\calR_2 \right\},\\
    \wt{\alpha}(\tilde{r}_0,\wt{\psi})
    &:= \inf\left\{ \abs*{\tinnerM{\wt{\psi}}{\wt{v}^{\tilde{r}_0}(T) - \wt{v}^{\tilde{r}_1+\tilde{r}_2}(T)}}
    : \tilde{r}_1\in\tcalR_1,\, \tilde{r}_2\in\tcalR_2 \right\}.
  \end{align*}
  By \cref{prop:inf-balls}, \cref{eq:main-proof-iff-1} is equivalent to $\alpha$ vanishing identically, while \cref{eq:main-proof-iff-2} is equivalent to $\wt{\alpha}$ vanishing identically.

  Recall that $\Phi_\rcv\colon\rcv\to\trcv$ is a Riemannian isometry by \cref{prop:data-isometry}. For each $j=0,1,2$, the ambient balls $B_{\M}(x_j,\varepsilon)$ and $B_{\tM}(\tx_j,\varepsilon)$ coincide with the corresponding balls for the intrinsic distances on $\rcv$ and $\trcv$. Hence, $\Phi_\rcv$ restricts to an isometry $B_{\M}(x_j,\varepsilon)\to B_{\tM}(\tx_j,\varepsilon)$, and $\Phi_\rcv$-matching of sources defines an isomorphism $\calR_j\ni r_j\mapsto\tilde{r}_j\in\tcalR_j$.

  Combining these isomorphisms with \cref{item:HW-Grcv-inner-product} of \cref{prop:HW-G-rcv} yields $\alpha(r_0,\psi) = \wt{\alpha}(\tilde{r}_0,\Grcv\psi)$ for all $r_0\in\calR_0$ and $\psi\in\HWM$. By \cref{prop:HW-G-rcv}, $\Grcv\colon\HWM\to\tHWM$ is also an isomorphism. Therefore,
  \begin{equation*}
    \alpha \text{ vanishes identically}
    \Longleftrightarrow
    \wt{\alpha} \text{ vanishes identically},
  \end{equation*}
  so \cref{eq:main-proof-iff-1,eq:main-proof-iff-2} are equivalent. Since $x_0,x_1,x_2\in\rcv$ and $\rho_0,\rho_1,\rho_2\in\ointerval{0}{\infty}$ were arbitrary, \cref{item:HLOS-2} holds. By \cref{prop:HLOS}, there exists a Riemannian isometry $\Phi\colon (\M,g)\to(\tM,\tg)$ such that $\Phi|_\rcv = \Phi_\rcv$. It follows from \cref{prop:isometry-identifications} that $\Phi|_\src = \Phi_\src$.

  (\labelcref{cond:main-isometry} $\Rightarrow$ \labelcref{cond:main-inclusion}) Let $\Phi$ be as in \cref{cond:main-isometry} and fix $\ell$. Since $\Phi$ is a Riemannian isometry, the map $\El\ni\phi\mapsto\phi\circ\Phi^{-1}\in\tEl$ is a linear isomorphism; it is norm-preserving because $\Phi$ is measure-preserving. Moreover, $\Phi|_\rcv=\Phi_\rcv$ gives
  \begin{equation*}
    \bigl(\phi\circ\Phi^{-1}\bigr)\bigl(\Phi_\rcv(y)\bigr)
    = \bigl(\phi\circ\Phi^{-1}\bigr)\bigl(\Phi(y)\bigr)
    = \phi(y),
    \qquad \phi\in\El,\quad y\in\rcv,
  \end{equation*}
  which is the second relation in \cref{eq:Gl-maps}. By the uniqueness of the maps in \cref{prop:Gl-maps}, $\Glrcv\phi=\phi\circ\Phi^{-1}$ for every $\phi\in\El$; in particular, $\Glrcv$ is norm-preserving and $\sv(\Glrcv)=\{1\}$. Hence \cref{eq:main-inclusion} holds with the zero sequence, which is an admissible Fourier-decay bound. This proves the equivalence. Since \cref{cond:main-isometry} holds under either condition, the argument just given also proves the final assertions.
\end{proof}

\Cref{item:classical-result} of the following \namecref{prop:main-corollaries} recovers, within the present framework, the non-disjoint-data result~\cite[Corollary~2]{MR3826551} (which also follows from~\cite[Theorem~1]{MR2435213}). Neither the condition $\tsrc\cap\trcv\neq\emptyset$ nor the agreement of $\Phi_\src$ and $\Phi_\rcv$ on $\src\cap\rcv$ is assumed a priori; both follow from $\src\cap\rcv\neq\emptyset$ and the equivalence of the data. \Cref{item:connected-U} shows that if the metric is known on a connected open set meeting both $\src$ and $\rcv$, then uniqueness up to isometry holds without any further assumptions.

\begin{corollary}\label{prop:main-corollaries}
  Let $(\M,g)$ and $(\tM,\tg)$ be Riemannian manifolds as in \cref{ass:standing}, and assume that they have equivalent cross-spectral data.

  Suppose that one of the following two conditions holds:
  \begin{enumerate}
    \item\label{item:classical-result} $\src\cap\rcv\neq\emptyset$, or
    \item\label{item:connected-U} there exist connected open sets $U\subset\M$, $\wt{U}\subset\tM$, and a Riemannian isometry $\Psi\colon U\to\wt{U}$ such that $U\cap\src\neq\emptyset$, $U\cap\rcv\neq\emptyset$, and
      \begin{equation*}
        \Psi|_{U\cap\src} = \Phi_\src|_{U\cap\src}
        \quad\text{and}\quad
        \Psi|_{U\cap\rcv} = \Phi_\rcv|_{U\cap\rcv}.
      \end{equation*}
  \end{enumerate}
  Then there exists a Riemannian isometry $\Phi\colon (\M,g)\to (\tM,\tg)$ such that $\Phi|_\src=\Phi_\src$ and $\Phi|_\rcv=\Phi_\rcv$.
\end{corollary}

\begin{proof}
  In either case it suffices to prove that $\Glsrc=\Glrcv$ for every $\ell$. Indeed, \cref{eq:Gl-adjoint} then gives $(\Glrcv)^\ast\Glrcv=\id_{\El}$, so $\sv(\Glrcv)=\{1\}$. Since the zero sequence is an admissible Fourier-decay bound, \cref{eq:main-inclusion} holds with it, and \cref{prop:main-uniqueness} applies.
  \begin{enumerate}
    \item Set $V:=\src\cap\rcv$. We first show that $\Phi_\src|_V=\Phi_\rcv|_V$.

      Let $\psi\in\Cc{\src}$ and $\wt{\psi}\in\Cc{\tsrc}$ be $\Phi_\src$-matched. By \cref{data:projectors}, the functions $\PlRS\psi$ and $\tPlRS\wt{\psi}$ are $\Phi_\rcv$-matched for every $\ell$, that is,
      \begin{equation*}
        (P_\ell\psi)(y)
        = (\wt{P}_\ell\wt{\psi})\bigl(\Phi_\rcv(y)\bigr),
        \qquad y\in\rcv.
      \end{equation*}
      The expansions $\sum_{\ell\ge0}P_\ell\psi=\psi$ and $\sum_{\ell\ge0}\wt{P}_\ell\wt{\psi}=\wt{\psi}$ converge in $\LM$ and in $\tLM$, respectively, and $\Phi_\rcv$ is measure-preserving by \cref{prop:data-isometry}, so composition with $\Phi_\rcv$ is an isometry $\tLrcv\to\Lrcv$. Summing the displayed identity over $\ell$ therefore gives $\psi=\wt{\psi}\circ\Phi_\rcv$ in $\Lrcv$, hence everywhere on $\rcv$ by continuity. In particular,
      \begin{equation*}
        \wt{\psi}\bigl(\Phi_\rcv(x)\bigr)
        = \psi(x)
        = \wt{\psi}\bigl(\Phi_\src(x)\bigr),
        \qquad x\in V,
      \end{equation*}
      where the second equality holds since $\psi$ and $\wt{\psi}$ are $\Phi_\src$-matched. Since the matched pair was arbitrary and compactly supported smooth functions separate points, we conclude that $\Phi_\src|_V=\Phi_\rcv|_V$.

      Fix $\ell$ and $\phi\in\El$. By \cref{eq:Gl-maps}, the functions $\Glsrc\phi$ and $\Glrcv\phi$ agree on the nonempty open set $\Phi_\src(V)=\Phi_\rcv(V)$, and hence everywhere on $\tM$ by unique continuation. Since $\ell$ and $\phi$ were arbitrary, $\Glsrc=\Glrcv$ for every $\ell$.

    \item Fix $\ell$ and $\phi\in \El$. Define $\psi_\src := (\Glsrc\phi)\circ\Psi$ and $\psi_\rcv := (\Glrcv\phi)\circ\Psi$ on $U$. Since $\Psi$ is a Riemannian isometry and $\Glsrc\phi\in \tEl$, we have $-\Delta_g\psi_\src=\lambda_\ell\psi_\src$ on $U$. The function $\psi_\src-\phi|_U$ therefore solves the eigenvalue equation on $U$ and, by \cref{eq:Gl-maps}, vanishes on the nonempty open set $U\cap\src$. Since $U$ is connected, unique continuation gives $\psi_\src=\phi|_U$. The same argument gives $\psi_\rcv=\phi|_U$. Since $\Psi(U)=\wt{U}$ and $\psi_\src=\psi_\rcv$, the functions $\Glsrc\phi$ and $\Glrcv\phi$ agree on $\wt{U}$, and therefore everywhere on $\tM$ by unique continuation. Since $\ell$ and $\phi$ were arbitrary, $\Glsrc=\Glrcv$ for every $\ell$.\qedhere
  \end{enumerate}
\end{proof}

The following \namecref{prop:main-beta-bounded} converts lower bounds for the $L^2$-masses of eigenfunctions into the singular-value condition of \cref{prop:main-uniqueness}. In \cref{item:two-manifold-bounds}, the bounds are imposed on the receiver sets of both manifolds. In \cref{item:single-manifold-bounds}, they are imposed on the source and receiver sets of $(\M,g)$ alone; uniqueness then holds within the class of all connected, closed manifolds. Since the data are symmetric (see \cref{rem:data-symmetry}), one may replace both receiver sets in the first condition by the corresponding source sets.

\begin{theorem}\label{prop:main-beta-bounded}
  Let $(\M,g)$ and $(\tM,\tg)$ be Riemannian manifolds as in \cref{ass:standing}. Suppose that one of the following two conditions holds:
  \begin{enumerate}
    \item\label{item:two-manifold-bounds} there exist admissible Fourier-decay bounds $(\beta_\ell)_\ell$ and $(\wt{\beta}_\ell)_\ell$ for $(\M,g)$ and $(\tM,\tg)$, respectively, such that, for every $\ell$, $\phi\in \El$, and $\wt{\phi}\in\tEl$,
      \begin{equation}\label{eq:beta-bound}
        \normrcv[\big]{\phi|_\rcv}
        \ge
        e^{-\beta_\ell}\normM{\phi}
        \quad\text{and}\quad
        \tnormrcv[\big]{\wt{\phi}|_{\trcv}}
        \ge
        e^{-\wt{\beta}_\ell}\tnormM[\big]{\wt{\phi}},
      \end{equation}
      or
    \item\label{item:single-manifold-bounds} there exists an admissible Fourier-decay bound $(\beta_\ell)_\ell$ for $(\M,g)$ such that, for every $\ell$ and $\phi\in\El$,
      \begin{equation}\label{eq:single-manifold-bounds}
        \normsrc[\big]{\phi|_\src}
        \ge e^{-\beta_\ell}\normM{\phi}
        \quad\text{and}\quad
        \normrcv[\big]{\phi|_\rcv}
        \ge e^{-\beta_\ell}\normM{\phi}.
      \end{equation}
  \end{enumerate}

  If the manifolds $(\M,g)$ and $(\tM,\tg)$ have equivalent cross-spectral data, then there exists a Riemannian isometry $\Phi\colon (\M,g)\to (\tM,\tg)$ such that $\Phi|_\src=\Phi_\src$ and $\Phi|_\rcv=\Phi_\rcv$.
\end{theorem}
\begin{proof}
  Assume that the manifolds have equivalent cross-spectral data.
  \begin{enumerate}
    \item Replacing both sequences by $(\beta_\ell+\wt{\beta}_\ell)_\ell$, which is an admissible Fourier-decay bound for both manifolds by \cref{rem:same-sequences,prop:decay-bound-properties}, we may assume that $(\beta_\ell)_\ell=(\wt{\beta}_\ell)_\ell$. Fix $\ell$ and $\phi\in \El$. By \cref{eq:Gl-norms} and the first inequality in \cref{eq:beta-bound},
      \begin{equation*}
        e^{-\beta_\ell}\normM{\phi}
        \le \normrcv[\big]{\phi|_\rcv}
        = \tnormrcv[\big]{(\Glrcv\phi)|_{\trcv}}
        \le \tnormM[\big]{\Glrcv\phi}.
      \end{equation*}
      Hence, $\sv(\Glrcv)\subset\rinterval{e^{-\beta_\ell}}{\infty}$.

      Applying the second inequality in \cref{eq:beta-bound} to $\wt{\phi}:=\Glrcv\phi\in\tEl$ and again using \cref{eq:Gl-norms}, we obtain
      \begin{equation*}
        e^{-\beta_\ell}\tnormM[\big]{\Glrcv\phi}
        \le \tnormrcv[\big]{(\Glrcv\phi)|_{\trcv}}
        = \normrcv[\big]{\phi|_\rcv}
        \le \normM{\phi}.
      \end{equation*}
      Hence, $\sv(\Glrcv)\subset\linterval{0}{e^{\beta_\ell}}$. Thus, \cref{eq:main-inclusion} holds, and \cref{prop:main-uniqueness} applies.

    \item Fix $\ell$ and $\phi\in\El$. By \cref{eq:Gl-norms,eq:single-manifold-bounds},
      \begin{equation*}
        \tnormM[\big]{\Glrcv\phi}
        \ge \tnormrcv[\big]{(\Glrcv\phi)|_{\trcv}}
        = \normrcv[\big]{\phi|_\rcv}
        \ge e^{-\beta_\ell}\normM{\phi}.
      \end{equation*}
      Thus, $\sv(\Glrcv)\subset\rinterval{e^{-\beta_\ell}}{\infty}$. The same argument with $\src$ and $\Glsrc$ shows that $\sv(\Glsrc)\subset\rinterval{e^{-\beta_\ell}}{\infty}$. By \cref{eq:Gl-adjoint}, the elements of $\sv(\Glsrc)$ are the reciprocals of those of $\sv(\Glrcv)$; hence \cref{eq:main-inclusion} holds, and \cref{prop:main-uniqueness} applies.\qedhere
  \end{enumerate}
\end{proof}

In the following \namecref{prop:subcritical-result}, the case $\theta=\wt{\theta}=0$ of \cref{item:subcritical-two-manifold} corresponds to condition~(C) in~\cite[Theorem~1.1]{MR4833220}. The \namecref{prop:subcritical-result} extends this result to $\theta,\wt{\theta}\in\ointerval{0}{1/2}$. Even when $\theta=\wt{\theta}=0$, the statement is stronger in two respects: the maps $\Phi_\src$ and $\Phi_\rcv$ are required only to be diffeomorphisms, rather than isometries, and the global isometry in the conclusion extends both maps.

By the spectral inequality~\cite[Theorem~{14.6}]{MR1743865} (see also~\cite{MR943927,MR1312710}), every connected, closed manifold satisfies the bounds~\labelcref{eq:subcritical-bounds,eq:subcritical-bounds-single} below at exponent $1/2$, with suitable constants. Thus the endpoint bounds impose no restriction beyond the standing assumptions. Together with \cref{prop:nonuniqueness}, this shows that the exponent range in both parts of \cref{prop:subcritical-result} is sharp.

\begin{corollary}\label{prop:subcritical-result}
  Let $(\M,g)$ and $(\tM,\tg)$ be Riemannian manifolds as in \cref{ass:standing}. Suppose that at least one of the following two conditions holds:
  \begin{enumerate}
    \item\label{item:subcritical-two-manifold} there exist constants $c,C,\tilde{c},\wt{C}>0$ and $\theta,\wt{\theta}\in\rinterval{0}{1/2}$ such that the following inequalities hold for every $\ell$, $\phi\in \El$, and $\wt{\phi}\in\tEl$:
      \begin{equation}\label{eq:subcritical-bounds}
        \normrcv[\big]{\phi|_\rcv} \ge ce^{-C\lambda_\ell^\theta}\normM{\phi}
        \quad\text{and}\quad
        \tnormrcv[\big]{\wt{\phi}|_{\trcv}} \ge \tilde{c}e^{-\wt{C}\wt{\lambda}_\ell^{\wt{\theta}}}\tnormM[\big]{\wt{\phi}},
      \end{equation}
      or
    \item\label{item:subcritical-single-manifold} there exist constants $c,C>0$ and $\theta\in\rinterval{0}{1/2}$ such that the following inequalities hold for every $\ell$ and $\phi\in\El$:
      \begin{equation}\label{eq:subcritical-bounds-single}
        \normsrc[\big]{\phi|_\src} \ge ce^{-C\lambda_\ell^\theta}\normM{\phi}
        \quad\text{and}\quad
        \normrcv[\big]{\phi|_\rcv} \ge ce^{-C\lambda_\ell^\theta}\normM{\phi}.
      \end{equation}
  \end{enumerate}

  If the manifolds $(\M,g)$ and $(\tM,\tg)$ have equivalent cross-spectral data, then there exists a Riemannian isometry $\Phi\colon (\M,g)\to (\tM,\tg)$ such that $\Phi|_\src=\Phi_\src$ and $\Phi|_\rcv=\Phi_\rcv$.
\end{corollary}

\begin{proof}
  By decreasing the constants if necessary, we may assume that $c\le 1$ and, under \cref{item:subcritical-two-manifold}, that $\tilde{c}\le 1$.

  Under \cref{item:subcritical-two-manifold}, \cref{prop:Ingham,prop:decay-bound-properties} show that $(C\lambda_\ell^\theta-\log c)_\ell$ and $(\wt{C}\wt{\lambda}_\ell^{\wt{\theta}}-\log\tilde{c})_\ell$ are admissible Fourier-decay bounds for $(\M,g)$ and $(\tM,\tg)$, respectively; hence \cref{eq:subcritical-bounds} takes the form~\labelcref{eq:beta-bound}.

  Under \cref{item:subcritical-single-manifold}, the same reasoning shows that $(C\lambda_\ell^\theta-\log c)_\ell$ is an admissible Fourier-decay bound for $(\M,g)$; hence \cref{eq:subcritical-bounds-single} takes the form~\labelcref{eq:single-manifold-bounds}.

  In either case, \cref{prop:main-beta-bounded} applies.
\end{proof}

%% file: nonuniqueness.tex
\section{Nonuniqueness}\label{sec:nonuniqueness}

In every dimension $n\ge 2$, we construct a smooth, connected, closed $n$-manifold $\M$, nonempty open sets $\src,\rcv\subset\M$, and two Riemannian metrics $g$ and $\tg$ on $\M$. The manifolds $(\M,g)$ and $(\M,\tg)$ have equivalent cross-spectral data relative to the identity maps $\id_\src$ and $\id_\rcv$, but are not related by any Riemannian isometry (\cref{prop:nonuniqueness}).

Our construction shares similarities with that of Daudé, Kamran, and Nicoleau for the anisotropic Calderón problem with data measured on disjoint sets (see~\cite{MR3973447,MR3916964}). Both constructions exploit a product structure that permits separation of variables in the second factor, reducing the inverse problem to a family of one-dimensional equations. For potentials on an interval, nonuniqueness with both Dirichlet and Neumann spectra prescribed was studied by Pierce~\cite{MR1907903} and Korotyaev~\cite{MR3962693}.

We construct the two warped-product metrics from ``warping profiles'' assembled by concatenating two types of fixed blocks, reducing nonuniqueness to a question about binary words and $2\times 2$ matrices. The construction works simultaneously at all frequencies, while that of Daudé, Kamran, and Nicoleau works at a single frequency.

The word mechanism of the construction has two kinds of antecedents. The first is Horowitz's work on universal trace identities between matrix words~\cite{MR314993}, with subsequent applications in other settings~\cite{MR553396,MR4704631}. Our construction uses identities for off-diagonal entries.

The second kind of antecedent is the analytic device of realizing letters as one-dimensional potential blocks and words as products of their transfer matrices. It appears in continuum substitution models of Damanik, Lenz, and Stolz~\cite{MR2244377} and, in a general finite-alphabet setting, in Damanik, Fillman, and Gorodetski~\cite{MR3205747}.

We proceed in five subsections. In \cref{sec:ingredients}, we introduce the parameters and define the product manifold, the two warped-product metrics determined by a pair of warping profiles, and the source and receiver sets. The next two subsections reduce the problem to conditions on that pair. In \cref{sec:reduction}, we show that the manifolds have equivalent cross-spectral data if and only if the transfer matrices of an associated family of one-dimensional Schrödinger equations have equal off-diagonal entries. In \cref{sec:isometry-reduction}, we show that, under a condition on the second factor of the product manifold, every isometry between the manifolds is a product of isometries of the two factors, so that the absence of an isometry becomes a condition on the warping profiles. In \cref{sec:profile-construction}, we construct profiles satisfying both conditions. Finally, \cref{sec:nonuniqueness-result} combines these ingredients and proves the nonuniqueness result, \cref{prop:nonuniqueness}.

\subsection{Setting}\label{sec:ingredients}

The manifolds are determined by the following parameters:
\begin{enumerate}
  \item $n\ge 2$, an integer, the dimension of the product manifold.

  \item $L>0$, a length, determining the first factor of the product manifold: the circle $C:=\R/(2L\Z)$. We equip $C$ with the flat metric $dr^2$ induced by the Euclidean metric on $\R$, and write $\pi_C\colon \R\to C$ for the quotient map, $\partial_r$ for the global vector field induced by $d/dx$, and $\partial_r^2$ for the Laplace--Beltrami operator. The map $\iota_C\colon C\to C$ is the reflection defined by $\iota_C(\pi_C(x)):=\pi_C(-x)$.

  \item $\eta\in\ointerval{0}{L/2}$, a length, determining the half-width of two opposite open arcs of $C$:
    \begin{equation}\label{eq:C-src-rcv}
      C_\src := \pi_C\bigl(\ointerval{-\eta}{\eta}\bigr)\subset C;
      \qquad
      C_\rcv := \pi_C\bigl(\ointerval{L-\eta}{L+\eta}\bigr)\subset C.
    \end{equation}

  \item $(F,h)$, the \emph{fiber manifold}, a smooth, connected, closed Riemannian manifold of dimension $m:=n-1$, forming the second factor of the product manifold. The distinct eigenvalues of $-\Delta_h$ are denoted by $0=\mu_0<\mu_1<\cdots$, with corresponding eigenspaces $Y_k$.

  \item $q,\tq\in C^\infty(C)$, the \emph{warping profiles}, determining the two metrics on the product manifold. They are assumed to satisfy
    \begin{align}
      q>0\quad&\text{and}\quad \tq>0 \quad\text{on}\quad C;\label{eq:q-positive}\\
      q = q\circ\iota_C\quad&\text{and}\quad \tq = \tq\circ\iota_C;\label{eq:q-even}\\
      q = \tq \quad&\text{on}\quad C_\src\cup C_\rcv.\label{eq:q-matching}
    \end{align}
\end{enumerate}
The parameters $n,L,\eta,(F,h),q,\tq$ are chosen explicitly in \cref{sec:nonuniqueness-result}, subject to these assumptions. The geometry of the circle is summarized in \cref{fig:nonuniqueness-geometry}.

\begin{figure}[tbp]
  \centering
  \input{nonuniqueness-geometry.tikz}
  \caption{The circle $C$ and its associated interval coordinates. (a) The sets $C_\src$ and $C_\rcv$ are $\iota_C$-invariant neighborhoods of the fixed points $\pi_C(0)$ and $\pi_C(L)$, respectively. The maps $j_\pm(x)=\pi_C(\pm x)$, $x\in\interval{0}{L}$, parametrize the positive and negative semicircles, respectively, where $\pi_C$ is the quotient map. (b) The endpoints of the fundamental interval $\interval{-L}{L}$ are identified, and the schematic even warping profiles $q$ and $\tq$ agree on the shaded representatives of $C_\src\cup C_\rcv$. The lower bars mark the intervals $I_\src,I_\rcv\subset\interval{0}{L}$ defined in \cref{eq:interval-ports}.}
  \label{fig:nonuniqueness-geometry}
\end{figure}
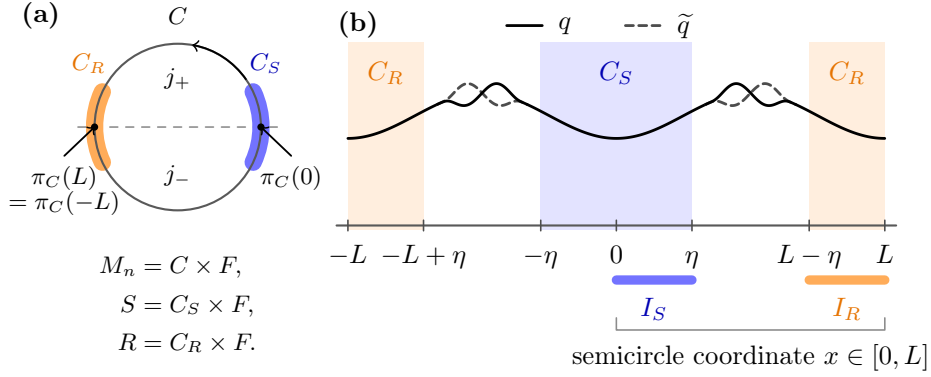

We use the above parameters to define the common underlying product manifold, the source and receiver sets $\src$ and $\rcv$, and the two Riemannian metrics as follows:
\begin{enumerate}
  \item The underlying manifold is
    \begin{equation}\label{eq:M-n}
      \M_n:=C\times F,
    \end{equation}
    which is smooth, connected, and closed, being a product of two such manifolds.

  \item The source and receiver sets $\src$ and $\rcv$ are the disjoint open sets
    \begin{equation}\label{eq:src-rcv}
      \src := C_\src \times F \subset \M_n
      \quad\text{and}\quad
      \rcv := C_\rcv \times F \subset \M_n.
    \end{equation}

  \item The two Riemannian metrics on $\M_n$ are the warped-product metrics
    \begin{equation}\label{eq:warped-product-metrics}
      g_q := dr^2 + a_q(r)^2 h
      \quad\text{and}\quad
      g_{\tq} := dr^2 + a_{\tq}(r)^2 h,
    \end{equation}
    where $a_q:=q^{2/m}$ and $a_{\tq}:={\tq}^{2/m}$, with $m=n-1$. (The exponent $2/m$ is chosen only to simplify some expressions below.)
\end{enumerate}

In the notation of \cref{ass:standing}, set
\begin{equation}\label{eq:standing-manifolds}
  (\M,g):=(\M_n,g_q),\qquad
  (\tM,\tg):=(\M_n,g_{\tq}),\qquad
  \tsrc:=\src,\qquad
  \trcv:=\rcv,
\end{equation}
and take $\Phi_\src:=\id_\src$ and $\Phi_\rcv:=\id_\rcv$. Then \cref{ass:standing} holds. Note that because of \cref{eq:q-matching}, the maps $\Phi_\src$ and $\Phi_\rcv$ are Riemannian isometries.

\subsubsection*{Notation.}
Functions on $C$ are identified with their pullbacks to $\M_n$ under the projection $\pr\colon\M_n=C\times F\to C$, given by $\pr(r,y)=r$.

The fixed points of $\iota_C$ are $\pi_C(0)$ and $\pi_C(L)$, and the sets $C_\src$ and $C_\rcv$ are $\iota_C$-invariant. A measurable function $w$ on $C$, or on an $\iota_C$-invariant measurable subset of $C$, is called \emph{even} if $w\circ\iota_C=w$ almost everywhere, and \emph{odd} if $w\circ\iota_C=-w$ almost everywhere; this is well defined on $L^2$-equivalence classes, since $\iota_C$ preserves null sets. Assumption~\labelcref{eq:q-even} states that the warping profiles are even; this is what lets the operators on $C$ split into a Dirichlet and a Neumann operator on an interval (\cref{prop:parity-decomposition}).

\subsection{Reduction to a transfer-matrix problem}\label{sec:reduction}

In this subsection, we reduce the comparison of the cross-spectral data of $(\M_n,g_q)$ and $(\M_n,g_{\tq})$ to equality of the corresponding off-diagonal entries of the transfer matrices for a family of one-dimensional Schrödinger equations on the semicircle $\interval{0}{L}$.

\subsubsection*{Separation of variables.}
Recall from \cref{eq:RSR} the cross-compressed resolvent $\RSR(z)$ of $-\Delta_{g_q}$. By \cref{prop:data-summary}, its values for $z<0$ determine and are determined by the cross-spectral data. Our goal here is to relate the cross-compressed resolvents of the two manifolds in \cref{eq:standing-manifolds} to those of a family of one-dimensional Schrödinger operators on $C$, indexed by the eigenvalues of $-\Delta_h$.

Since $a_q^m=q^2$, a direct computation gives the volume density and the Laplace--Beltrami operator of the warped product $(\M_n,g_q)$:
\begin{equation}\label{eq:warped-laplacian}
  dV_{g_q}=q^2\,dr\,dV_{h}
  \quad\text{and}\quad
  \Delta_{g_q}u=\frac{1}{q^2}\,\partial_r\bigl(q^2\,\partial_r u\bigr)+q^{-4/m}\,\Delta_{h}u.
\end{equation}
The eigenspace decomposition $L^2(F,dV_h)=\bigoplus_{k\ge0}Y_k$ and the product structure of the measure give the orthogonal decomposition
\begin{equation}\label{eq:L2-decomposition}
  L^2\bigl(\Omega\times F,\;dr\,dV_h\bigr)=\bigoplus_{k\ge0}L^2\bigl(\Omega,dr\bigr)\otimes Y_k,
  \qquad\Omega\in\{C,\,C_\src,\,C_\rcv\}.
\end{equation}
Here $\bigoplus$ denotes the orthogonal Hilbert-space direct sum, and the measures are the \emph{unwarped} product measures.

For $\mu\ge 0$ and any warping profile $q$, define
\begin{equation}\label{eq:potential}
  V_\mu(\cdot\,;q):=\frac{\partial_r^2q}{q}+\mu q^{-4/m}\in C^\infty(C),
\end{equation}
the smooth potential associated with $q$. We abbreviate $V_\mu:=V_\mu(\cdot\,;q)$ for the warping profile $q$ fixed in \cref{sec:ingredients}. Let $H_\mu^P$ ($P$ for periodic) be the self-adjoint operator on $L^2(C,dr)$ defined by
\begin{equation}\label{eq:operator-H-P}
  H_\mu^P := -\partial_r^2 + V_\mu
\end{equation}
with domain $\calD(H_\mu^P):=H^2(C)$. Since $q$ is even by \cref{eq:q-even} and $\partial_r^2$ preserves parity, the potential $V_\mu$ is even.

To diagonalize the warped Laplacian in terms of the operators $H_{\mu_k}^P$, we combine \cref{eq:L2-decomposition} with the Liouville transform given by multiplication by $q$. For the general Sturm--Liouville reduction, see, e.g.,~\cite[Lemma~1.4]{MR1889089}. By \cref{eq:warped-laplacian}, the map
\begin{equation*}
  U_{\M_n}\colon L^2\bigl(\M_n,dV_{g_q}\bigr)\to L^2\bigl(\M_n,dr\,dV_h\bigr)
\end{equation*}
given by this multiplication is unitary (a surjective linear isometry), and its adjoint is multiplication by $q^{-1}$.

\begin{lemma}\label{prop:Delta-diagonalization}
  Set
  \begin{equation*}
    B:=U_{\M_n}(-\Delta_{g_q})U_{\M_n}^\ast
    \quad\text{on }L^2\bigl(\M_n,dr\,dV_h\bigr),
    \qquad
    \calH_k:=L^2(C,dr)\otimes Y_k.
  \end{equation*}
  Then $\calD(B)=H^2(\M_n)$, each $\calH_k$ reduces $B$, and
  \begin{align}
    B|_{\calH_k\cap\calD(B)} &= H_{\mu_k}^P\otimes\id_{Y_k},\label{eq:Delta-block-operator}\\
    \calH_k\cap\calD(B) &= H^2(C)\otimes Y_k.\label{eq:Delta-block-domain}
  \end{align}
  Moreover, $H_\mu^P$ is nonnegative for every $\mu\ge0$.
\end{lemma}

\begin{proof}
  For $\mu\ge0$, let $L_\mu$ be the nonnegative operator on the space $L^2(C,q^2\,dr)$ defined by
  \begin{equation*}
    L_\mu:=-\frac{1}{q^2}\,\partial_r\bigl(q^2\,\partial_r\,\cdot\,\bigr)+\mu q^{-4/m},
    \qquad
    \calD(L_\mu):=H^2(C).
  \end{equation*}
  Let $U_C\colon L^2(C,q^2\,dr)\to L^2(C,dr)$ be multiplication by $q$. It is unitary, its adjoint is multiplication by $q^{-1}$, and it maps $H^2(C)$ onto itself. Computing first on $C^\infty(C)$ and then extending to $H^2(C)$ by density and continuity gives the Liouville identity
  \begin{equation}\label{eq:liouville}
    U_C L_\mu U_C^\ast=H_\mu^P.
  \end{equation}
  Consequently, $L_\mu$ is self-adjoint and $H_\mu^P$ is nonnegative.

  Since $\calD(-\Delta_{g_q})=H^2(\M_n)$ and multiplication by $q$ maps $H^2(\M_n)$ onto itself, the domain of $B$ is $H^2(\M_n)$. The operator $B$ is self-adjoint by unitary equivalence. It is also continuous from $H^2(\M_n)$ to $L^2(\M_n,dr\,dV_h)$, being a composition of smooth multipliers and $-\Delta_{g_q}$.

  Fix $k$, let $\Pi_k$ be the orthogonal projector onto $\calH_k$, and choose an orthonormal basis $\phi_1,\dots,\phi_d$ of $Y_k$, where $d:=\dim Y_k$. For $v\in L^2(\M_n,dr\,dV_h)$, the function $a_j[v]\in L^2(C,dr)$ is defined by $a_j[v](r):=\inner{v(r,\cdot)}{\phi_j}_{L^2(F,dV_h)}$ for a.e.\ $r\in C$, and $\Pi_kv=\sum_{j=1}^d a_j[v]\otimes\phi_j$. The maps $L^2(\M_n,dr\,dV_h)\ni v\mapsto a_j[v]\in L^2(C,dr)$ are bounded. The finite smooth basis ensures that $\Pi_k C^\infty(\M_n)\subset C^\infty(\M_n)$, and the eigenfunction equations and \cref{eq:warped-laplacian} give $\Pi_k\Delta_{g_q}u=\Delta_{g_q}\Pi_ku$ for smooth $u$. Since the weight $q^2$ depends only on $r$, $\Pi_k$ is also an orthogonal projector in $L^2(\M_n,dV_{g_q})$. Using the spectral $H^2(\M_n)$ norm, we obtain, for $u\in C^\infty(\M_n)$,
  \begin{equation}\label{eq:pi-ku-bounded}
    \norm{\Pi_ku}_{H^2(\M_n)}
    =\norm{\Pi_k(\id-\Delta_{g_q})u}_{L^2(\M_n,dV_{g_q})}
    \le\norm{u}_{H^2(\M_n)}.
  \end{equation}
  By density, $\Pi_k$ preserves $H^2(\M_n)$, satisfies the estimate~\labelcref{eq:pi-ku-bounded} for every $u\in H^2(\M_n)$, and commutes with $\Delta_{g_q}$ on that domain. Since it also commutes with multiplication by $q$ and $q^{-1}$, we obtain $\Pi_k\calD(B)\subset\calD(B)$ and $B\Pi_ku=\Pi_kBu$ for every $u\in\calD(B)$. Hence $\calH_k$ reduces $B$.

  To identify $\calH_k\cap\calD(B)$ and the action of $B$ on it, take $a_j\in C^\infty(C)$. The identities~\labelcref{eq:warped-laplacian,eq:liouville} give
  \begin{equation}\label{eq:B-action}
    B\left(\sum_{j=1}^d a_j\otimes\phi_j\right)
    =\sum_{j=1}^d(H_{\mu_k}^Pa_j)\otimes\phi_j.
  \end{equation}
  For $a_j\in H^2(C)$, approximate each coefficient by smooth functions in $H^2(C)$. By \cref{eq:B-action} and continuity of $H_{\mu_k}^P\colon H^2(C)\to L^2(C,dr)$, the approximating tensor sums and their images under $B$ converge in $L^2(\M_n,dr\,dV_h)$. Closedness of $B$ therefore gives $H^2(C)\otimes Y_k\subset\calH_k\cap\calD(B)$ and extends the displayed action to this space.

  Conversely, let $u\in\calH_k\cap\calD(B)$, choose smooth $v_\ell\to u$ in $H^2(\M_n)$, and put $u_\ell:=\Pi_kv_\ell$. Estimate~\labelcref{eq:pi-ku-bounded} gives $u_\ell\to u$ in $H^2(\M_n)$, so continuity of $B$ gives $Bu_\ell\to Bu$ in $L^2(\M_n,dr\,dV_h)$. Therefore, $a_j[u_\ell]\to a_j[u]$ and $H_{\mu_k}^Pa_j[u_\ell]=a_j[Bu_\ell]\to a_j[Bu]$ in $L^2(C,dr)$. Closedness of $H_{\mu_k}^P$ now gives $a_j[u]\in H^2(C)$ for every $j$, and hence $u\in H^2(C)\otimes Y_k$. This proves the reverse inclusion in the domain identity~\labelcref{eq:Delta-block-domain}, while \cref{eq:B-action} proves the operator identity~\labelcref{eq:Delta-block-operator}.
\end{proof}

For $\mu\ge0$ and $z\in\R\setminus\spec(H_\mu^P)$, define the cross-compressed resolvents on $C$ by
\begin{equation}\label{eq:cross-compressed-resolvent-P}
  \Rport{P}(z\portsC)
  :=\res_{C_\rcv}\circ(H_\mu^P-z)^{-1}\circ\ext_{C_\src}
  \colon L^2\bigl(C_\src,dr\bigr)\to L^2\bigl(C_\rcv,dr\bigr),
\end{equation}
where $\ext_{C_\src}\colon L^2(C_\src,dr)\to L^2(C,dr)$ and $\res_{C_\rcv}\colon L^2(C,dr)\to L^2(C_\rcv,dr)$ are the extension-by-zero and restriction operators, respectively.

Let $U_\src$ and $U_\rcv$ denote multiplication by $q|_\src$ and $q|_\rcv$, respectively. By \cref{eq:warped-laplacian}, the maps
\begin{align}
  U_\src&\colon L^2\bigl(\src,dV_{g_q}\bigr)\to L^2\bigl(\src,dr\,dV_h\bigr),\label{eq:U_src}\\
  U_\rcv&\colon L^2\bigl(\rcv,dV_{g_q}\bigr)\to L^2\bigl(\rcv,dr\,dV_h\bigr),\label{eq:U_rcv}
\end{align}
are unitary, with inverses given by multiplication by $q^{-1}$ restricted to $\src$ and $\rcv$, respectively.

We can now decompose $\RSR(z)$ into the cross-compressed resolvents~\labelcref{eq:cross-compressed-resolvent-P} on $C$.

\begin{lemma}\label{prop:RSR-resolvent-decomposition}
  For $z<0$, the cross-compressed resolvent decomposes as
  \begin{equation}\label{eq:cross-compressed-resolvent-decomposition}
    \RSR(z)
    =U_\rcv^\ast
    \left[\bigoplus_{k\ge0}\Rport[\mu_k]{P}(z\portsC)\otimes\id_{Y_k}\right]
    U_\src,
  \end{equation}
  where the operator in brackets maps $L^2\bigl(\src,dr\,dV_h\bigr)$ to $L^2\bigl(\rcv,dr\,dV_h\bigr)$, decomposed by \cref{eq:L2-decomposition} with $\Omega=C_\src$ and $\Omega=C_\rcv$, respectively.
\end{lemma}

\begin{proof}
  Let $z<0$, and let $B$ and $\calH_k$ be as in \cref{prop:Delta-diagonalization}. By unitary equivalence and that lemma, $B$ and every $H_{\mu_k}^P$ are self-adjoint and nonnegative, so their resolvents at $z$ exist. Let $\Pi_k$ be the orthogonal projector onto $\calH_k$. Since $\calH_k$ reduces $B$, the projector $\Pi_k$ commutes with $(B-z)^{-1}$. The description of $B|_{\calH_k\cap\calD(B)}$ in \cref{prop:Delta-diagonalization} therefore gives
  \begin{equation}\label{eq:block-resolvent}
    (B-z)^{-1}|_{\calH_k}=(H_{\mu_k}^P-z)^{-1}\otimes\id_{Y_k}.
  \end{equation}

  Since $\sum_{k\ge0}\Pi_k=\id$ in the strong operator topology on $L^2(\M_n,dr\,dV_h)$ by \cref{eq:L2-decomposition} with $\Omega=C$, the identity~\labelcref{eq:block-resolvent} yields
  \begin{equation*}
    (B-z)^{-1}=\bigoplus_{k\ge0}(H_{\mu_k}^P-z)^{-1}\otimes\id_{Y_k}.
  \end{equation*}
  Conjugating by $U_{\M_n}$ gives
  \begin{equation}\label{eq:resolvent-decomposition}
    (-\Delta_{g_q}-z)^{-1}
    =U_{\M_n}^\ast
    \left[\bigoplus_{k\ge0}(H_{\mu_k}^P-z)^{-1}\otimes\id_{Y_k}\right]
    U_{\M_n}.
  \end{equation}
  Finally, restriction to $\rcv=C_\rcv\times F$ and extension by zero from $\src=C_\src\times F$ act fiberwise and intertwine $U_{\M_n}$ with $U_\rcv$ and $U_\src$. Applying them to \cref{eq:resolvent-decomposition} proves \cref{eq:cross-compressed-resolvent-decomposition}.
\end{proof}

Set $\wt{V}_\mu:=V_\mu(\cdot\,;\tq)$, and let $\wt{H}_\mu^P$ be the self-adjoint operator defined by \cref{eq:operator-H-P} with $V_\mu$ replaced by $\wt{V}_\mu$. Both $H_\mu^P$ and $\wt{H}_\mu^P$ act on $L^2(C,dr)$ with domain $H^2(C)$, independently of the warping profile. The operator $\tRSR(z)$, defined in \cref{eq:tRSR}, is the cross-compressed resolvent of $-\Delta_{g_{\tq}}$. Define $\tRport{P}(z\portsC)$ by \cref{eq:cross-compressed-resolvent-P} with $H_\mu^P$ replaced by $\wt{H}_\mu^P$. Since $\tq$ satisfies the same assumptions as $q$, the constructions and results of this subsection apply verbatim with tildes.

Together with \cref{prop:data-equivalence}, the following \namecref{prop:resolvent-reduction} reduces the equivalence of the cross-spectral data to a family of identities on the circle, one for each $z<0$ and each eigenvalue of $-\Delta_{h}$.

\begin{proposition}\label{prop:resolvent-reduction}
  Let $z<0$. The identity $\RSR(z)=\tRSR(z)$ holds if and only if, for every $k=0,1,\dots$,
  \begin{equation}\label{eq:resolvent-reduction}
    \Rport[\mu_k]{P}(z\portsC)=\tRport[\mu_k]{P}(z\portsC).
  \end{equation}
\end{proposition}

\begin{proof}
  The spaces in \cref{eq:U_src,eq:U_rcv} are the same for both warping profiles, because the codomains do not involve the profile and the domains agree by \cref{eq:q-matching}. For the same reason, $U_\src$ and $U_\rcv$ coincide with their tilded counterparts.

  If $z<0$, then $z$ is outside the spectra of the nonnegative operators $-\Delta_{g_q}$ and $-\Delta_{g_{\tq}}$. Therefore, the decomposition~\labelcref{eq:cross-compressed-resolvent-decomposition} holds, and it also holds if $\RSR(z)$ and $\Rport[\mu_k]{P}(z\portsC)$ are replaced by $\tRSR(z)$ and $\tRport[\mu_k]{P}(z\portsC)$, respectively. The operators are equal precisely when all the corresponding summands are, which is \cref{eq:resolvent-reduction}.
\end{proof}

\subsubsection*{Parity reduction from $C$ to the semicircle.}
Let $j_\pm$ parametrize the two closed semicircles:
\begin{equation}\label{eq:semicircle-parametrization}
  j_\pm\colon \interval{0}{L} \to C,
  \qquad
  j_\pm(x) := \pi_C(\pm x).
\end{equation}
We call $j_+(\interval{0}{L})$ the \emph{positive semicircle} and $j_-(\interval{0}{L})$ the \emph{negative semicircle}. We write $V_\mu(\cdot\,;q)$ also for its pullback along $j_+$, for any warping profile $q$, and in particular for $V_\mu$ and $\wt{V}_\mu$. The potentials are even as functions on $C$, so their values on $\interval{0}{L}$ determine them on $C$. In all interval equations on $\interval{0}{L}$ below, primes denote ordinary derivatives with respect to $x\in\interval{0}{L}$.

We aim first to express $H_\mu^P$ in terms of the Dirichlet and Neumann realizations of a differential equation on $\interval{0}{L}$, and then to relate the cross-compressed resolvents of the operators on $\interval{0}{L}$ to those of the operator on the circle (cf.\ \cref{prop:resolvent-reduction}). To that end, let $H_\mu^D$ and $H_\mu^N$ be the Dirichlet and Neumann realizations, respectively, of $f\mapsto -f''+V_\mu f$ on $\interval{0}{L}$, with the domains
\begin{align*}
  \calD(H_\mu^D) &:= \{w \in H^2(\interval{0}{L}) : w(0)=w(L)=0\}, \\
  \calD(H_\mu^N) &:= \{w \in H^2(\interval{0}{L}) : w'(0)=w'(L)=0\},
\end{align*}
where, for a compact interval $\interval{a}{b}$, we write $H^2(\interval{a}{b})$ for the Sobolev space $H^2\bigl(\ointerval{a}{b}\bigr)$ whose elements are identified with their $C^1(\interval{a}{b})$-representatives. For $L^2$-spaces on intervals, we always use Lebesgue measure and omit it from the notation. These operators are self-adjoint on $L^2(\interval{0}{L})$.

For $\Omega\in\{C,\,C_\src,\,C_\rcv\}$, write $\Leven{\Omega,dr}$ and $\Lodd{\Omega,dr}$ for the spaces of even and odd $L^2$-functions on $\Omega\subset C$. In each case the two subspaces are orthogonal and $L^2(\Omega,dr)$ is their orthogonal direct sum.

Let $E_N,E_D\colon L^2(\interval{0}{L})\to L^2(C,dr)$ be the normalized even and odd \emph{parity maps}, defined for $x\in\ointerval{0}{L}$ by
\begin{equation}\label{eq:parity-maps}
  \bigl(E_Nw\bigr)\bigl(j_\pm(x)\bigr) := \frac{1}{\sqrt{2}}\,w(x)
  \quad\text{and}\quad
  \bigl(E_Dw\bigr)\bigl(j_\pm(x)\bigr) := \pm\frac{1}{\sqrt{2}}\,w(x).
\end{equation}
The values at the two fixed points of $\iota_C$ are immaterial, and the subscripts $N$ and $D$ anticipate \cref{prop:parity-decomposition}. The operators $E_N$ and $E_D$ map $L^2(\interval{0}{L})$ isometrically onto $\Leven{C,dr}$ and onto $\Lodd{C,dr}$, respectively.

\begin{lemma}\label{prop:parity-decomposition}
  Fix $\mu\ge 0$ and recall that $V_\mu\in C^\infty(C)$ is even. The map
  \begin{equation*}
    E\colon L^2(\interval{0}{L})\oplus L^2(\interval{0}{L})\to L^2(C,dr),\quad
    E(f_N,f_D) := E_Nf_N+E_Df_D,
  \end{equation*}
  is unitary and satisfies
  \begin{equation}\label{eq:parity-expression}
    H_\mu^P E = E(H_\mu^N\oplus H_\mu^D).
  \end{equation}
  Consequently,
  \begin{equation}\label{eq:spectral-identity}
    \spec(H_\mu^P) = \spec(H_\mu^N)\cup\spec(H_\mu^D),
  \end{equation}
  and for every $z\in\R\setminus\spec(H_\mu^P)$,
  \begin{equation}\label{eq:resolvent-fold}
    (H_\mu^P-z)^{-1}
    = E_N(H_\mu^N-z)^{-1}E_N^\ast
    + E_D(H_\mu^D-z)^{-1}E_D^\ast.
  \end{equation}
\end{lemma}

\begin{proof}
  Pullback by $\pi_C$ provides a unitary identification
  \begin{equation*}
    J\colon L^2(C,dr)\to L^2(\interval{-L}{L}),
    \qquad
    (J w)(x) := w(\pi_C(x)).
  \end{equation*}
  Under this identification, $H_\mu^P$ transports to $JH_\mu^P J^{-1}=-\frac{d^2}{dx^2}+V_\mu\circ\pi_C$ with domain
  \begin{equation*}
    H_{\mathrm{per}}^2(\interval{-L}{L})
    := \Big\{ w\in H^2(\interval{-L}{L}) : w(-L)=w(L),\,w'(-L)=w'(L) \Big\}.
  \end{equation*}
  We continue to denote the transported operator and potential by $H_\mu^P$ and $V_\mu$, respectively.

  We also suppress $J$ in the notation for $JE$, $JE_N$, and $JE_D$. Thus, $E_N$ and $E_D$ become the normalized even and odd extension operators from $L^2(\interval{0}{L})$ to $L^2(\interval{-L}{L})$. Their ranges are the even and odd subspaces $\Leven{\interval{-L}{L}}$ and $\Lodd{\interval{-L}{L}}$, respectively, so $E$ is unitary. Since all assertions are invariant under this unitary identification, it suffices to prove the \namecref{prop:parity-decomposition} for the transported operators.

  Let $w\in H^2(\interval{0}{L})$. Its even and odd extensions are $H^2$ on $\ointerval{-L}{0}$ and on $\ointerval{0}{L}$. If $w'(0)=w'(L)=0$, then the even extension belongs to $H_{\mathrm{per}}^2(\interval{-L}{L})$: its one-sided values and first derivatives agree at $0$, and its values and first derivatives at $-L$ and $L$ coincide. Conversely, if the even extension belongs to $H_{\mathrm{per}}^2(\interval{-L}{L})$, then its derivative is odd, continuous at $0$, and takes the same value at $-L$ and $L$, so it vanishes at $0$ and at $L$; that is, $w'(0)=w'(L)=0$. Similarly, the odd extension belongs to $H_{\mathrm{per}}^2(\interval{-L}{L})$ if and only if $w(0)=w(L)=0$. Consequently, the operator domains satisfy
  \begin{align}
    E_N\calD(H_\mu^N)
    &= \calD(H_\mu^P)\cap \Leven{\interval{-L}{L}},\label{eq:even-domain}\\
    E_D\calD(H_\mu^D)
    &= \calD(H_\mu^P)\cap \Lodd{\interval{-L}{L}}.\label{eq:odd-domain}
  \end{align}

  The reflection $x\mapsto-x$ preserves $\calD(H_\mu^P)=H_{\mathrm{per}}^2(\interval{-L}{L})$, and it commutes with $H_\mu^P$ because $V_\mu$ is even. Together with \cref{eq:even-domain,eq:odd-domain}, this gives
  \begin{equation*}
    E^\ast H_\mu^P E=H_\mu^N\oplus H_\mu^D.
  \end{equation*}
  Since $E$ is unitary, this is equivalent to \cref{eq:parity-expression}, and \cref{eq:spectral-identity} follows.

  Let $z\in\R\setminus\spec(H_\mu^P)$. By \cref{eq:spectral-identity}, $z$ is outside $\spec(H_\mu^N)$ and $\spec(H_\mu^D)$, and \cref{eq:parity-expression} gives $(H_\mu^P-z)^{-1}=E\bigl((H_\mu^N-z)^{-1}\oplus(H_\mu^D-z)^{-1}\bigr)E^\ast$. Expanding $E$ and $E^\ast$ gives \cref{eq:resolvent-fold}.
\end{proof}

As illustrated in \cref{fig:nonuniqueness-geometry}, put
\begin{equation}\label{eq:interval-ports}
  I_\src:=\ointerval{0}{\eta},
  \qquad
  I_\rcv:=\ointerval{L-\eta}{L}.
\end{equation}
For $X\in\{D,N\}$ and $z\in\R\setminus\spec(H_\mu^X)$, the cross-compressed resolvent of $H_\mu^X$ is denoted by
\begin{equation*}
  \Rport{X}(z\portsI)
  :=\res_{I_\rcv}\circ\,(H_\mu^X-z)^{-1}\,\circ\,\ext_{I_\src}
  \colon L^2\bigl(I_\src\bigr)\to L^2\bigl(I_\rcv\bigr).
\end{equation*}
Here $\res_{I_\rcv}\colon L^2(\interval{0}{L})\to L^2(I_\rcv)$ and $\ext_{I_\src}\colon L^2(I_\src)\to L^2(\interval{0}{L})$ are the restriction and extension-by-zero operators. Define $\wt{H}_\mu^X$ and $\tRport{X}(z\portsI)$ analogously from $\tq$, the latter for $z\in\R$ outside the spectrum of $\wt{H}_\mu^X$.

Together with \cref{prop:data-equivalence,prop:resolvent-reduction}, the following \namecref{prop:parity-resolvent-reduction} reduces the equivalence of the cross-spectral data to the Dirichlet and Neumann problems on $\interval{0}{L}$.

\begin{proposition}\label{prop:parity-resolvent-reduction}
  Let $\mu\ge0$ and $z<0$. The following are equivalent:
  \begin{cond}
  \item\label{cond:circle-resolvents} $\Rport{P}(z\portsC)=\tRport{P}(z\portsC)$.
  \item\label{cond:interval-resolvents} $\Rport{N}(z\portsI)=\tRport{N}(z\portsI)$ and $\Rport{D}(z\portsI)=\tRport{D}(z\portsI)$.
  \end{cond}
\end{proposition}

\begin{proof}
  The operators $H_\mu^P$ and $\wt{H}_\mu^P$ are nonnegative by \cref{prop:Delta-diagonalization}. Thus, $z<0$ lies outside their spectra, and by \cref{eq:spectral-identity}, it also lies outside the spectra of $H_\mu^N$, $H_\mu^D$, $\wt{H}_\mu^N$ and $\wt{H}_\mu^D$. All the resolvents below therefore exist.

  Define the four maps
  \begin{equation*}
    E^Z_X\colon L^2(I_Z)\to L^2(C_Z,dr),
    \qquad Z\in\{\src,\rcv\},\quad X\in\{D,N\},
  \end{equation*}
  by the formulas~\labelcref{eq:parity-maps}, restricted to $x\in I_Z$. For each $Z$, the operators $E^Z_N$ and $E^Z_D$ map $L^2(I_Z)$ isometrically onto $\Leven{C_Z,dr}$ and $\Lodd{C_Z,dr}$, respectively. Since the even and odd subspaces are orthogonal,
  \begin{equation}\label{eq:localized-parity-orthogonality}
    (E^Z_X)^\ast E^Z_Y=
    \begin{cases}
      \id_{L^2(I_Z)} & \text{if } X=Y,\\
      0   & \text{if } X\neq Y,
    \end{cases}
    \quad Z\in\{\src,\rcv\},\quad X,Y\in\{D,N\}.
  \end{equation}
  The following intertwining relations hold:
  \begin{equation}\label{eq:localized-parity-intertwining}
    \res_{C_\rcv}E_X=E^\rcv_X\,\res_{I_\rcv}
    \enspace\text{and}\enspace
    E_X^\ast\ext_{C_\src}=\ext_{I_\src}\,(E^\src_X)^\ast,
    \enspace X\in\{D,N\}.
  \end{equation}

  Substituting \cref{eq:resolvent-fold} into the definition~\labelcref{eq:cross-compressed-resolvent-P} and applying \cref{eq:localized-parity-intertwining} to each summand yields
  \begin{equation}\label{eq:localized-fold}
    \Rport{P}(z\portsC)
    = \sum_{X\in\{D,N\}} E^\rcv_X\,\Rport{X}(z\portsI)\,(E^\src_X)^\ast,
  \end{equation}
  and the analogous equation for $\tRport{P}(z\portsC)$ in terms of $\tRport{X}(z\portsI)$. This proves \labelcref{cond:interval-resolvents} $\Rightarrow$ \labelcref{cond:circle-resolvents}.

  Conversely, multiplying \cref{eq:localized-fold} from the left by $(E^\rcv_X)^\ast$ and from the right by $E^\src_X$ and using \cref{eq:localized-parity-orthogonality} gives
  \begin{equation*}
    \Rport{X}(z\portsI)=(E^\rcv_X)^\ast\,\Rport{P}(z\portsC)\,E^\src_X,
  \end{equation*}
  for $X\in\{D,N\}$, and likewise for the tilded operators. This proves \labelcref{cond:circle-resolvents} $\Rightarrow$ \labelcref{cond:interval-resolvents}.
\end{proof}

\subsubsection*{Reduction to the off-diagonal transfer-matrix entries.}
\Cref{prop:data-equivalence,prop:resolvent-reduction,prop:parity-resolvent-reduction} show that it only remains to compare the Dirichlet and Neumann cross-compressed resolvents associated with the two warping profiles. We first compute the Green's functions of $H_\mu^D$ and $H_\mu^N$, that is, the integral kernels of their resolvents, and then relate the corresponding cross-compressed resolvents to the off-diagonal entries of the transfer matrix across $\interval{0}{L}$ for the equation
\begin{equation}\label{eq:schrodinger}
  -f''+(V_\mu-z)f=0
  \qquad\text{on}\qquad\interval{0}{L}.
\end{equation}
Recall that $V_\mu\in C^\infty(\interval{0}{L})$ denotes the potential of \cref{eq:potential} pulled back along the parametrization $j_+$ defined in \cref{eq:semicircle-parametrization}.

For $\mu\ge0$, $z\in\R$, and $x_0,x_1\in\interval{0}{L}$, let $T_{x_1\shortleftarrow x_0}(\mu,z)\in\R^{2\times2}$ denote the transfer matrix of \cref{eq:schrodinger}, characterized by
\begin{equation}\label{eq:transfer-matrix}
  T_{x_1\shortleftarrow x_0}(\mu,z)
  \begin{bmatrix}
    f(x_0)\\f'(x_0)
  \end{bmatrix}
  =
  \begin{bmatrix}
    f(x_1)\\f'(x_1)
  \end{bmatrix},
\end{equation}
for every solution $f$ of this equation. We define $\wt{T}_{x_1\shortleftarrow x_0}(\mu,z)$ analogously using $\wt{V}_\mu$.

For $x_0,x_1,x_2\in\interval{0}{L}$, the characterization~\labelcref{eq:transfer-matrix} also gives
\begin{equation}\label{eq:transfer-composition}
  T_{x_2\shortleftarrow x_0}(\mu,z)=T_{x_2\shortleftarrow x_1}(\mu,z)\,T_{x_1\shortleftarrow x_0}(\mu,z).
\end{equation}

Let $f_1$ and $f_2$ be the solutions of \cref{eq:schrodinger} whose state vectors $(f_j,f'_j)$ at $x_0$ are the standard basis vectors $e_1$ and $e_2$ of $\R^2$, respectively. Then
\begin{equation*}
  T_{x_1\shortleftarrow x_0}(\mu,z)
  =
  \begin{bmatrix}
    f_1(x_1)&f_2(x_1)\\
    f_1'(x_1)&f_2'(x_1)
  \end{bmatrix}.
\end{equation*}
Since the equation has no first-order term, the Wronskian $W(f_1,f_2):=f_1f_2'-f_1'f_2$ is constant. Hence
\begin{equation}\label{eq:det-T-1}
  \det\bigl(T_{x_1\shortleftarrow x_0}(\mu,z)\bigr)
  =W(f_1,f_2)(x_1)
  =W(f_1,f_2)(x_0)
  =1.
\end{equation}

\begin{lemma}\label{prop:interval-green}
  Fix $\mu\ge 0$ and $z\in\R$, and let $d_0,d_L,n_0,n_L\in C^2(\interval{0}{L})$ be the unique solutions of \cref{eq:schrodinger} determined by the Cauchy data
  \begin{align*}
    \bigl(d_0(0),d_0'(0)\bigr) &= (0,1),
    &\bigl(d_L(L),d_L'(L)\bigr) &= (0,-1),\\
    \bigl(n_0(0),n_0'(0)\bigr) &= (1,0),
    &\bigl(n_L(L),n_L'(L)\bigr) &= (1,0),
  \end{align*}
  with their dependence on $\mu$ and $z$ suppressed in the notation. Then:
  \begin{enumerate}
    \item\label{item:G-D-green} if $d_0(L)\neq 0$, then $z\notin\spec(H_\mu^D)$ and $(H_\mu^D-z)^{-1}$ is an integral operator with symmetric Green's function $G_{\mu,z}^D\in C(\interval{0}{L}^2)$ given by
      \begin{equation}\label{eq:G-D-green}
        G_{\mu,z}^D(x,y) =
        \frac{1}{d_0(L)}\cdot
        \begin{cases}
          d_0(x)d_L(y)   & \text{if } 0\le x \le y \le L, \\
          d_0(y)d_L(x)  & \text{if } 0\le y < x \le L;
        \end{cases}
      \end{equation}
    \item\label{item:G-N-green} if $n_0'(L)\neq 0$, then $z\notin\spec(H_\mu^N)$ and $(H_\mu^N-z)^{-1}$ is an integral operator with symmetric Green's function $G_{\mu,z}^N\in C(\interval{0}{L}^2)$ given by
      \begin{equation}\label{eq:G-N-green}
        G_{\mu,z}^N(x,y) =
        \frac{1}{n_0'(L)}\cdot
        \begin{cases}
          n_0(x)n_L(y) & \text{if } 0\le x \le y \le L, \\
          n_0(y)n_L(x) & \text{if } 0\le y < x \le L;
        \end{cases}
      \end{equation}
    \item\label{item:spec} $z\in\spec(H_\mu^D)$ if and only if $d_0(L)=0$, and $z\in\spec(H_\mu^N)$ if and only if $n_0'(L)=0$.
  \end{enumerate}
\end{lemma}
\begin{proof}
  \begin{enumerate}
    \item Evaluating the Wronskian $W(d_L,d_0)$ at $L$ yields $W(d_L,d_0) = d_0(L)$. Since $d_0(L)\neq 0$, we have $W(d_L,d_0)\neq0$. The standard formula for the Green's function of a regular Sturm--Liouville problem~\cite[Lemma~9.7 and the example following it]{MR3243083} therefore gives $z\notin\spec(H_\mu^D)$ and shows that the Dirichlet resolvent has kernel~\labelcref{eq:G-D-green}. The displayed formula is symmetric and continuous.

    \item The Neumann case is proved in the same way, using the identity $W(n_L,n_0)=n_0'(L)$.

    \item If $d_0(L)=0$, then $d_0$ is a nonzero element of $\calD(H_\mu^D)$ with $H_\mu^D d_0=zd_0$, whence $z\in\spec(H_\mu^D)$. If $d_0(L)\neq 0$, then $z\notin\spec(H_\mu^D)$ by \cref{item:G-D-green}. The Neumann case is similar.\qedhere
  \end{enumerate}
\end{proof}

For the potentials $V_\mu$ and $\wt{V}_\mu$, we denote the off-diagonal entries of the transfer matrices across $\interval{0}{L}$, defined by \cref{eq:transfer-matrix}, as follows:
\begin{equation}\label{eq:transfer-matrix-entries}
  \Tmat{0}{L}(\mu,z) =
  \begin{bmatrix}
    \ast     & b(\mu,z) \\
    c(\mu,z) & \ast
  \end{bmatrix}
  \enspace\text{ and }\enspace
  \tTmat{0}{L}(\mu,z) =
  \begin{bmatrix}
    \ast                 & \tilde{b}(\mu,z) \\
    \tilde{c}(\mu,z) & \ast
  \end{bmatrix}
\end{equation}
(we will need no notation for the diagonal entries).

The definitions of the transfer matrix and of the solutions $d_0$ and $n_0$ in \cref{prop:interval-green} give
\begin{equation}\label{eq:b-c}
  d_0(L)=b(\mu,z)
  \quad\text{and}\quad
  n_0'(L)=c(\mu,z)
\end{equation}
for every $\mu\ge0$ and $z\in\R$.

\begin{lemma}\label{prop:localized-transfer-entries}
  For every $\mu\ge0$ and $z<0$,
  \begin{align}
    \Rport{D}(z\portsI)=\tRport{D}(z\portsI)
    &\Longleftrightarrow b(\mu,z)=\tilde{b}(\mu,z),\label{eq:localized-transfer-D}\\
    \Rport{N}(z\portsI)=\tRport{N}(z\portsI)
    &\Longleftrightarrow c(\mu,z)=\tilde{c}(\mu,z).\label{eq:localized-transfer-N}
  \end{align}
\end{lemma}

\begin{proof}
  We prove \cref{eq:localized-transfer-D}; the proof of \cref{eq:localized-transfer-N} is the same with $n_0$, $n_L$ and $c$ in place of $d_0$, $d_L$ and $b$.

  By \cref{prop:Delta-diagonalization,eq:spectral-identity}, the operator $H_\mu^D$ is nonnegative, so $z<0$ lies outside its spectrum. Hence \cref{item:spec} of \cref{prop:interval-green} and \cref{eq:b-c} imply that $b(\mu,z)\neq 0$, and \cref{item:G-D-green} of the same lemma applies. The same argument gives $\tilde{b}(\mu,z)\neq 0$.

  Let $\wt{d}_0$, $\wt{d}_L$ denote the solutions defined as in \cref{prop:interval-green}, with $V_\mu$ replaced by $\wt{V}_\mu$. By \cref{eq:q-matching} and continuity, $V_\mu=\wt{V}_\mu$ on $\interval{0}{\eta}$ and on $\interval{L-\eta}{L}$. The solutions $d_0$ and $\wt{d}_0$ have the same Cauchy data at $0$, and $d_L$ and $\wt{d}_L$ at $L$, so uniqueness for the initial value problem gives
  \begin{equation}\label{eq:ends-agree}
    d_0=\wt{d}_0 \quad\text{on }\interval{0}{\eta}
    \qquad\text{and}\qquad
    d_L=\wt{d}_L \quad\text{on }\interval{L-\eta}{L}.
  \end{equation}

  Since $\eta<L/2$, we have $y<x$ whenever $y\in I_\src$ and $x\in I_\rcv$, so the second branch of \cref{eq:G-D-green} applies. Together with \cref{eq:b-c,eq:ends-agree}, this shows that $\Rport{D}(z\portsI)-\tRport{D}(z\portsI)$ is the integral operator whose kernel, for $x\in I_\rcv$ and $y\in I_\src$, is
  \begin{equation}\label{eq:difference-kernel}
    \biggl(\frac{1}{b(\mu,z)}-\frac{1}{\tilde{b}(\mu,z)}\biggr)d_L(x)d_0(y).
  \end{equation}

  Since $d_0$ and $d_L$ have nonzero Cauchy data at $0$ and $L$, respectively, neither $d_0|_{I_\src}$ nor $d_L|_{I_\rcv}$ vanishes identically. Hence the operator with kernel~\labelcref{eq:difference-kernel} is zero if and only if the scalar prefactor vanishes, that is, if and only if $b(\mu,z)=\tilde{b}(\mu,z)$.
\end{proof}

Combining this with the earlier reductions gives the main result of this subsection.

\begin{proposition}\label{prop:transfer-reduction}
  The following are equivalent:
  \begin{cond}[beginpenalty=10000]
  \item\label{cond:nonuniqueness-equivalent-data} The manifolds $(\M_n,g_q)$ and $(\M_n,g_{\tq})$ have equivalent cross-spectral data relative to the identity maps $\id_\src$ and $\id_\rcv$.
  \item\label{cond:b-and-c} The off-diagonal transfer-matrix entries in \cref{eq:transfer-matrix-entries} satisfy
    \begin{equation}\label{eq:b-and-c}
      b(\mu_k,z) = \tilde{b}(\mu_k,z)
      \quad\text{and}\quad
      c(\mu_k,z) = \tilde{c}(\mu_k,z)
    \end{equation}
    for every $k=0,1,\dots$ and every $z<0$.
  \end{cond}
\end{proposition}

\begin{proof}
  Since the identifications $\Phi_\src$ and $\Phi_\rcv$ are the identities, the characterization \cref{data:resolvents} of \cref{prop:data-equivalence} and density (see \cref{item:data-extension} of \cref{rem:data-equivalence}) show that \cref{cond:nonuniqueness-equivalent-data} of the current \namecref{prop:transfer-reduction} is equivalent to $\RSR(z)=\tRSR(z)$ for every $z<0$. By \cref{prop:resolvent-reduction,prop:parity-resolvent-reduction,prop:localized-transfer-entries}, this is equivalent to \cref{cond:b-and-c}.
\end{proof}

\subsection{Reduction of isometries to symmetries of the circle}\label{sec:isometry-reduction}

In this subsection, we prove that, under certain conditions on the fiber $(F,h)$, every Riemannian isometry $\Phi\colon(\M_n,g_q)\to(\M_n,g_{\tq})$ has the form $(r,y)\mapsto(\tau(r),\Psi(y))$, where $\tau$ and $\Psi$ are Riemannian isometries of $(C,dr^2)$ and $(F,h)$, respectively. Of the three standing assumptions~\labelcref{eq:q-positive,eq:q-even,eq:q-matching} on the warping profiles, only positivity~\labelcref{eq:q-positive} is needed here.

Some condition on the fiber is necessary. If $q=\tq\equiv 1$ and $(F,h)$ is isometric to $(C,dr^2)$, then $\M_2$ is a flat torus whose two factors are interchanged by an isometry, and that isometry is not of the above form.

\subsubsection*{Preliminaries.}
A \emph{loop} in a manifold is a piecewise smooth map $\gamma$ defined on $\interval{0}{1}$ such that $\gamma(0)=\gamma(1)$. Two loops $\gamma_0$ and $\gamma_1$ in the same manifold are \emph{freely homotopic} if there exists a continuous map $H$ from the square $\interval{0}{1}^2$ into the manifold such that, for all $s,t\in\interval{0}{1}$,
\begin{equation*}
  H(s,0)=\gamma_0(s),
  \qquad
  H(s,1)=\gamma_1(s),
  \qquad
  H(0,t)=H(1,t).
\end{equation*}
If $\gamma_0$ and $\gamma_1$ are freely homotopic, then
\begin{equation}\label{eq:freely-homotopic-integral}
  \int_{\gamma_0}\omega=\int_{\gamma_1}\omega
\end{equation}
for every closed $1$-form $\omega$ on the manifold~\cite[proof of Theorem~16.26, Equation~(16.9)]{MR2954043}.

The sets $\{r\}\times F$, $r\in C$, are the \emph{fibers} of $\M_n$. A loop on $\M_n=C\times F$ is called \emph{horizontal} if its image is included in $C\times\{y\}$ for some $y\in F$, and \emph{vertical} if its image is included in a fiber. For every loop $\sigma$ in the circle $C$,
\begin{equation}\label{eq:circle-loop-period}
  \int_\sigma dr \in 2L\Z,
\end{equation}
since a lift of $\sigma$ to $\R$ has endpoints differing by an element of $2L\Z$.

At every $(r,y)\in\M_n$, the tangent space $T_{(r,y)}\M_n$ decomposes into the \emph{horizontal subspace} $T_rC$ and the \emph{vertical subspace} $T_yF$:
\begin{equation*}
  T_{(r,y)}\M_n=T_rC\oplus T_yF.
\end{equation*}
The decomposition is orthogonal with respect to both metrics $g_q$ and $g_{\tq}$.

Recall that $\pr\colon\M_n\to C$ is the projection onto the first factor and define
\begin{equation*}
  dr_{\M_n}:=\pr^\ast dr.
\end{equation*}
Write
\begin{equation*}
  \partial_r^{\M_n}:=(\partial_r,0)
\end{equation*}
for the horizontal lift of $\partial_r$ to $\M_n$. This field spans the horizontal subspace at every point. Note that $dr_{\M_n}$ annihilates the vertical summand, satisfies $dr_{\M_n}(\partial_r^{\M_n})=1$, and has pointwise norm $1$ with respect to both $g_q$ and $g_{\tq}$.

Define a family $(\sigma_y)_{y\in F}$ of horizontal loops parametrized by $F$:
\begin{equation}\label{eq:horizontal-circle}
  \sigma_y\colon\interval{0}{1}\to\M_n=C\times F,
  \qquad
  \sigma_y(s)=(\pi_C(2Ls),y).
\end{equation}
Each $\sigma_y$ has velocity $\sigma_y'(s)=2L(\partial_r^{\M_n})_{\sigma_y(s)}$ and therefore length $2L$ with respect to both $g_q$ and $g_{\tq}$.

\subsubsection*{Splitting of isometries.}
The next \namecref{prop:unit-form-rigidity} identifies certain closed unit $1$-forms on $\M_n$. Applying it to the pullback of $dr_{\M_n}$ under an isometry will then yield the desired splitting.

\begin{lemma}\label{prop:unit-form-rigidity}
  Let $y_0\in F$, and let $\omega$ be a smooth closed $1$-form on $\M_n$ such that
  \begin{equation}\label{eq:unit-form-assumptions}
    \pointnorm{\omega}_{g_q}=1
    \quad\text{and}\quad
    \int_{\sigma_{y_0}}\omega\in 2L\Z.
  \end{equation}
  Suppose also that
  \begin{equation}\label{eq:vertical-loop-integral}
    \int_\zeta\omega=0
  \end{equation}
  for every vertical loop $\zeta$. Then there exists $\varepsilon\in\{-1,1\}$ such that
  \begin{equation*}
    \omega=\varepsilon\,dr_{\M_n}.
  \end{equation*}
\end{lemma}

\begin{proof}
  Since $F$ is connected, the loops $\sigma_y$, $y\in F$, are all freely homotopic. Hence \cref{eq:freely-homotopic-integral} shows that $\int_{\sigma_y}\omega$ is independent of $y\in F$. Moreover, $\pointnorm{\omega}_{g_q}=1$ implies
  \begin{equation*}
    \abs*{\int_{\sigma_y}\omega}
    \le 2L.
  \end{equation*}
  By \cref{eq:unit-form-assumptions}, the common value lies in $2L\Z$, hence in $2L\Z\cap\interval{-2L}{2L}$. Therefore there exists $\varepsilon\in\{-1,0,1\}$ such that
  \begin{equation}\label{eq:horizontal-period}
    \int_{\sigma_y}\omega=2L\varepsilon
    \qquad\text{for every }y\in F.
  \end{equation}

  Suppose that $\varepsilon=0$. Every loop $\gamma=(\alpha,\beta)$ in $\M_n=C\times F$ is freely homotopic to the reparametrized concatenation of the horizontal loop $\gamma_H:=(\alpha,\beta(0))$ and the vertical loop $\gamma_V:=(\alpha(0),\beta)$. For some $j\in\Z$, the loop $\gamma_H$ is freely homotopic within $C\times\{\beta(0)\}$ to $\sigma_{\beta(0)}^j$ (a negative power denoting an iterate of the reversed loop). Hence, by \cref{eq:freely-homotopic-integral}, equality~\labelcref{eq:horizontal-period} with $\varepsilon=0$ gives
  \begin{equation*}
    \int_{\gamma_H}\omega
    =j\int_{\sigma_{\beta(0)}}\omega
    =0.
  \end{equation*}
  Since $\gamma$ is freely homotopic to the concatenation of $\gamma_H$ and $\gamma_V$, \cref{eq:freely-homotopic-integral} and the hypothesis~\labelcref{eq:vertical-loop-integral} give
  \begin{equation*}
    \int_\gamma\omega
    =\int_{\gamma_H}\omega+\int_{\gamma_V}\omega
    =0.
  \end{equation*}
  It follows that $\omega=df$ for a smooth function $f$ on $\M_n$~\cite[Theorem~11.42]{MR2954043}. Since $\M_n$ is compact, $f$ attains a maximum at some point, where $\omega=df=0$. This contradicts $\pointnorm{\omega}_{g_q}=1$, so $\varepsilon\in\{-1,1\}$.

  Fix $y\in F$. By the velocity formula for $\sigma_y$, \cref{eq:horizontal-period} gives
  \begin{equation*}
    2L
    =2L\int_0^1\varepsilon\,\omega_{\sigma_y(s)}\bigl((\partial_r^{\M_n})_{\sigma_y(s)}\bigr)\,ds.
  \end{equation*}
  The continuous integrand satisfies
  \begin{equation*}
    \varepsilon\,\omega(\partial_r^{\M_n})
    \le \abs*{\omega(\partial_r^{\M_n})}
    \le \pointnorm{\omega}_{g_q}\pointnorm{\partial_r^{\M_n}}_{g_q}
    =1,
  \end{equation*}
  so the preceding integral equality forces $\omega(\partial_r^{\M_n})=\varepsilon$ on $C\times\{y\}$, hence everywhere. At a point $(r,y)$, let $\partial_r^{\M_n},e_1,\dots,e_m$ be a $g_q$-orthonormal basis with $e_j\in T_yF$. Then
  \begin{equation*}
    1=\pointnorm{\omega}_{g_q}^2
    =\abs*{\omega(\partial_r^{\M_n})}^2+\sum_{j=1}^m\abs*{\omega(e_j)}^2
    =1+\sum_{j=1}^m\abs*{\omega(e_j)}^2.
  \end{equation*}
  Thus $\omega$ vanishes on $T_yF$. Since $\omega(\partial_r^{\M_n})=\varepsilon$, we conclude that $\omega=\varepsilon\,dr_{\M_n}$.
\end{proof}

\begin{proposition}\label{prop:isometry-splitting}
  Recall that $a_q=q^{2/m}$ (see \cref{eq:warped-product-metrics}). Assume that one of the following holds:
  \begin{cond}
  \item\label{cond:fiber-circle} $m=1$ and, for some $L_F>0$, the manifold $(F,h)$ is isometric to $(\R/(2L_F\Z),d\theta^2)$, where $d\theta^2$ is induced by the Euclidean metric on $\R$, and $\min_{r\in C} a_q(r)<L/L_F$;
  \item\label{cond:fiber-simply-connected} $m\ge 2$ and $F$ is simply connected.
  \end{cond}
  Then every Riemannian isometry $\Phi\colon(\M_n,g_q)\to(\M_n,g_{\tq})$ is of the form
  \begin{equation}\label{eq:isometry-split}
    \Phi(r,y)=\bigl(\tau(r),\Psi(y)\bigr),
    \qquad (r,y)\in C\times F,
  \end{equation}
  where $\tau\colon C\to C$ and $\Psi\colon F\to F$ are Riemannian isometries of $(C,dr^2)$ and $(F,h)$, respectively. In particular,
  \begin{equation}\label{eq:isometry-profile-relation}
    q=\tq\circ\tau .
  \end{equation}
\end{proposition}

\begin{remark}\label{rem:fiber-circle-condition}
  A connected, closed Riemannian manifold of dimension $1$ is automatically isometric to the circle $(\R/(2L_F\Z),d\theta^2)$ for exactly one $L_F>0$. \Cref{cond:fiber-circle} of \cref{prop:isometry-splitting} compares the lengths of circles: the fiber $\{r\}\times F$ has length $2L_Fa_q(r)$ in $g_q$, while $C$ has length $2L$. So the condition says precisely that the shortest fiber is strictly shorter than $C$.
\end{remark}

\begin{proof}[Proof of \cref{prop:isometry-splitting}]
  Set $\omega:=\Phi^\ast dr_{\M_n}$ and fix $y_0\in F$. We first aim to apply \cref{prop:unit-form-rigidity} to $\omega$ to obtain $\varepsilon\in\{-1,1\}$ with
  \begin{equation}\label{eq:isometry-pullback-dr}
    \Phi^\ast dr_{\M_n}=\varepsilon\,dr_{\M_n}.
  \end{equation}
  Verifying the hypotheses of \cref{prop:unit-form-rigidity} is where the condition on $F$ enters.

  Since pullback commutes with exterior differentiation and $\Phi$ is an isometry, $\omega$ is closed and $\pointnorm{\omega}_{g_q}=1$. For every piecewise smooth loop $\gamma$ in $\M_n$,
  \begin{equation}\label{eq:isometry-period-winding}
    \int_\gamma\omega
    =\int_{\pr\circ\Phi\circ\gamma}dr
    \in 2L\Z,
  \end{equation}
  where the last inclusion follows from \cref{eq:circle-loop-period}. In particular, \cref{eq:isometry-period-winding} holds with $\gamma=\sigma_{y_0}$. Thus the conditions in \cref{eq:unit-form-assumptions} are satisfied.

  We next verify the remaining hypothesis~\labelcref{eq:vertical-loop-integral}. First suppose that \cref{cond:fiber-circle} holds. For $r\in C$, let $\zeta_r\colon\interval{0}{1}\to\M_n$ be the loop $\zeta_r(s)=(r,\pi_F(2L_F\,s))$, where $\pi_F\colon\R\to F$ is the composition of the quotient map $\R\to\R/(2L_F\Z)$ with a Riemannian isometry $(\R/(2L_F\Z),d\theta^2)\to(F,h)$, and choose $r_\ast\in C$ such that $a_q(r_\ast)=\min_Ca_q$. By $\pointnorm{\omega}_{g_q}=1$,
  \begin{equation*}
    \abs*{\int_{\zeta_{r_\ast}}\omega}
    \le 2L_F\,a_q(r_\ast)
    <2L,
  \end{equation*}
  while by \cref{eq:isometry-period-winding} the integral belongs to $2L\Z$. Hence the integral vanishes. Now let $\zeta$ be a vertical loop, included in the fiber $\{r\}\times F$ for some $r\in C$. Within its own fiber, $\zeta$ is freely homotopic to $\zeta_r^j$ for some $j\in\Z$, and, since $C$ is connected, $\zeta_r^j$ is in turn freely homotopic to $\zeta_{r_\ast}^j$ in $\M_n$. Hence
  \begin{equation*}
    \int_{\zeta}\omega=j\int_{\zeta_{r_\ast}}\omega=0.
  \end{equation*}

  If \cref{cond:fiber-simply-connected} holds, then every loop $\zeta$ included in a fiber $\{r\}\times F$ is contractible within that fiber, so $\int_\zeta \omega$ vanishes because $\omega$ is closed.

  This verifies the remaining hypothesis, so \cref{eq:isometry-pullback-dr} holds.

  We next deduce from \cref{eq:isometry-pullback-dr} that $\Phi$ has the form~\labelcref{eq:isometry-split}, with $\tau$ a Riemannian isometry of $C$ and $\Psi$ a diffeomorphism of $F$. Let $\Phi_H$ and $\Phi_V$ be the horizontal and vertical components of $\Phi$, so that $\Phi(r,y)=(\Phi_H(r,y),\Phi_V(r,y))$. Since $\Phi_H=\pr\circ\Phi$, \cref{eq:isometry-pullback-dr} says that $\Phi_H^\ast dr=\varepsilon\,dr_{\M_n}$. As $dr$ is nowhere vanishing on the one-dimensional manifold $C$, it follows that $d\Phi_H$ vanishes on the vertical subspaces and that $d\Phi_H(\partial_r^{\M_n})=\varepsilon\partial_r$. Since $F$ is connected, $\Phi_H$ is independent of $y$. Hence $\Phi_H(r,y)=\tau(r)$ for a smooth map $\tau\colon C\to C$ with $d\tau(\partial_r)=\varepsilon\partial_r$. Any such map is of the form
  \begin{equation*}
    \tau(\pi_C(x))=\pi_C(\varepsilon x+r_0)
  \end{equation*}
  for some $r_0\in\R$, so $\tau$ is a Riemannian isometry of $C$. Moreover, since
  \begin{equation*}
    d\Phi(\partial_r^{\M_n})=\varepsilon\partial_r^{\M_n}+d\Phi_V(\partial_r^{\M_n}),
  \end{equation*}
  we have
  \begin{equation*}
    1=\pointnorm[\big]{d\Phi(\partial_r^{\M_n})}_{g_{\tq}}^2
    =1+a_{\tq}(\tau(r))^2\pointnorm[\big]{d\Phi_V(\partial_r^{\M_n})}_h^2,
  \end{equation*}
  so $d\Phi_V(\partial_r^{\M_n})=0$. Thus $\Phi_V(r,y)=\Psi(y)$ for a smooth map $\Psi\colon F\to F$. Since $\Phi$ is a diffeomorphism, so is $\Psi$.

  It remains to show that $\Psi$ is a Riemannian isometry of $(F,h)$ and that \cref{eq:isometry-profile-relation} holds. Restricting $\Phi^\ast g_{\tq}=g_q$ to the vertical subspaces, on which $d\Phi$ acts as $d\Psi$, gives
  \begin{equation*}
    a_{\tq}(\tau(r))^2\Psi^\ast h=a_q(r)^2h.
  \end{equation*}
  Hence $\Psi^\ast h=c^2h$ and $a_q=c\,a_{\tq}\circ\tau$ for some constant $c>0$.

  Since $\Psi$ is a diffeomorphism, $(F,\Psi^\ast h)$ is isometric to $(F,h)$ and therefore has the same diameter, whereas $\diam(F,c^2h)=c\diam(F,h)$. The diameter is finite and nonzero, so $c=1$. Thus $\Psi^\ast h=h$, so $\Psi$ is a Riemannian isometry, and $a_q=a_{\tq}\circ\tau$. Finally, $a_q=q^{2/m}$ and $a_{\tq}=\tq^{2/m}$, so \cref{eq:isometry-profile-relation} follows.
\end{proof}

Thus, under either condition of \cref{prop:isometry-splitting}, proving that $(\M_n,g_q)$ and $(\M_n,g_{\tq})$ are not related by a Riemannian isometry reduces to showing that $q\ne\tq\circ\tau$ for every Riemannian isometry $\tau$ of $C$.

\subsection{Warping profiles from binary words}\label{sec:profile-construction}

The preceding two subsections isolate two requirements on the warping profiles: by \cref{prop:transfer-reduction}, their transfer matrices must have the same off-diagonal entries, and by \cref{prop:isometry-splitting}, no Riemannian isometry of $C$ may relate them. We construct profiles satisfying both requirements from binary words, one block per letter. The transfer matrix of such a profile factors into a product of the two single-block matrices, one factor per letter (\cref{prop:block-transfer}), and two such profiles are related by a Riemannian isometry of $C$ exactly when their words are equivalent (\cref{prop:profile-recovery}). The two requirements thus become, respectively, an identity between matrix products and inequivalence of the words; \cref{prop:nine-letter-words} supplies a pair of words meeting both.

\subsubsection*{Words, profiles, and transfer matrices.}
We begin with the required word notation.

A \emph{binary word}, or simply a \emph{word}, is a finite, possibly empty, string over the alphabet $\{\bzero,\bone\}$. A nonempty word $s$ of length $N\ge1$ is written as $s=s_1\cdots s_N$, where $s_j\in\{\bzero,\bone\}$ for $j=1,\dots,N$. We use typewriter font to distinguish letters and literal words, such as $\bzero$, $\bone$, and $\word{100\,101}$, from the numbers $0$, $1$, and $100101$. In long words we insert thin spaces for legibility; they are not part of the word.

The \emph{concatenation} of words $s=s_1\cdots s_N$ and $t=t_1\cdots t_{N'}$ is the word $st=s_1\cdots s_Nt_1\cdots t_{N'}$. The \emph{reverse complement} $s^\sharp$ of $s$ is the reverse of its bitwise complement, that is,
\begin{equation}\label{eq:word-sharp}
  s^{\sharp} := \overline{s_N} \cdots \overline{s_1},
  \quad\text{where}\quad
  \overline{\bzero} := \bone
  \quad\text{and}\quad
  \overline{\bone} := \bzero.
\end{equation}
The reverse complement of the empty word is the empty word.

We construct warping profiles piecewise from two functions $p_\bzero,p_\bone\in C^\infty(\interval{0}{1})$, called \emph{blocks}. The blocks are chosen to be positive and equal to $1$ near the endpoints, so that profiles assembled from multiple blocks are smooth across the seams. The two blocks are reflections of each other. A common central plateau identifies the block centers, while an off-center marker distinguishes the two orientations.

Fix $0<\rho<\delta<1/12$, and choose functions $\chi_{\mathrm{pl}},\chi_{\mathrm{mk}}\in \Cc{\R}$ such that
\begin{enumerate}
  \item the \emph{plateau function} $\chi_{\mathrm{pl}}$ is supported in $\interval{-\delta}{\delta}$, valued in $\interval{0}{1}$, and equal to $1$ exactly on $\interval{-\rho}{\rho}$;

  \item the \emph{marker function} $\chi_{\mathrm{mk}}$ is supported in $\interval{-\delta}{\delta}$, valued in $\interval{0}{1}$, with $\chi_{\mathrm{mk}}(0)=1$.
\end{enumerate}
Define $p_{\bzero},p_{\bone}\in C^\infty(\interval{0}{1})$ by
\begin{equation}\label{eq:marked-block}
  \begin{aligned}
    p_{\bzero}(x) &:= 1+2\,\chi_{\mathrm{pl}}\bigl(x-\tfrac{1}{2}\bigr)+\chi_{\mathrm{mk}}\bigl(x-\tfrac{1}{3}\bigr),\\
    p_{\bone}(x) &:= p_{\bzero}(1-x).
  \end{aligned}
\end{equation}
Note that $p_{\bzero}$ and $p_{\bone}$ are identically $1$ on $\interval{0}{1/4}\cup\interval{3/4}{1}$, and that they attain their maximum $3$ exactly on $\interval{1/2-\rho}{1/2+\rho}$. The marker distinguishes the two blocks: $p_{\bzero}(1/3)=2\neq1=p_{\bone}(1/3)$; in particular, $p_{\bzero}\neq p_{\bone}$.

\begin{definition}\label{def:word-profile}
  Fix an integer $L\ge 1$. Let $C=C(L):=\R/(2L\Z)$ be the circle as defined in \cref{sec:ingredients}, with its associated quotient map $\pi_C\colon\R\to C$ and reflection $\iota_C\colon C\to C$. Let $s=s_1\cdots s_L$ be a binary word. The \emph{warping profile associated with $s$} is the function $q_s\in C^\infty(C)$ determined by
  \begin{equation*}
    q_s\bigl(\pi_C(j+x)\bigr)=p_{s_{j+1}}(x),
    \qquad j=0,\dots,L-1,\quad x\in\interval{0}{1},
    \qquad q_s\circ\iota_C=q_s.
  \end{equation*}
\end{definition}

The construction is illustrated in \cref{fig:word-profile-construction}.

\begin{remark}\label{rem:word-profiles-admissible}
  Since the blocks are positive and identically $1$ near their endpoints, $q_s$ is a well-defined smooth positive function on $C$, and it is even by construction. Thus $q_s$ satisfies the requirements~\labelcref{eq:q-positive,eq:q-even} for warping profiles. With $\eta=1/4$ in \cref{eq:C-src-rcv}, the warping profiles associated with any two words $s$ and $t$ of equal length agree on $C_\src\cup C_\rcv$, and hence satisfy \cref{eq:q-matching}.
\end{remark}

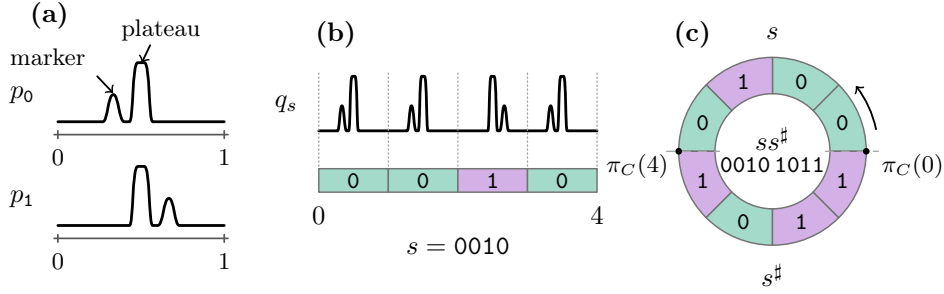
\begin{figure}[tbp]
  \centering
  \input{nonuniqueness-word-profile.tikz}
  \caption{Construction of a warping profile associated with a word. (a) The two blocks $p_{\bzero}$ and $p_{\bone}$ are shown schematically. Their common central plateau identifies block centers, while the off-center marker distinguishes the two orientations. (b) The profile $q_s$ is shown on $\interval{0}{L}$ for $L=4$ and $s=\word{0010}$. (c) The even reflection completes the profile on $C$; reading in the positive direction from $\pi_C(0)$ gives $ss^\sharp=\word{0010\,1011}$.}
  \label{fig:word-profile-construction}
\end{figure}

Two words $s$ and $t$ are said to \emph{differ by a cyclic shift} if there exist words $u$ and $v$ such that $s=uv$ and $t=vu$. We define the following equivalence relation on all binary words:
\begin{equation*}
  s\sim_\sharp t
  \Longleftrightarrow
  ss^\sharp \quad\text{and}\quad tt^\sharp \quad \text{differ by a cyclic shift}.
\end{equation*}

Reading the profile $q_s$ block by block along $C$, in the positive direction from $\pi_C(0)$, yields the word $ss^\sharp$; that is,
\begin{equation*}
  q_s\bigl(\pi_C(j+x)\bigr)=p_{(ss^\sharp)_{j+1}}(x),
  \qquad j=0,\dots,2L-1,\quad x\in\interval{0}{1}.
\end{equation*}
Indeed, the positive semicircle carries the blocks $s_1,\dots,s_L$ in order, while the reflection $\iota_C$ reverses both their order and the orientation within each block, so the negative semicircle carries $\overline{s_L}\cdots\overline{s_1}=s^\sharp$. Consequently, for words $s$ and $t$ of equal length $L\ge 1$,
\begin{equation}\label{eq:equivalence-shift}
  s\sim_\sharp t
  \Longleftrightarrow
  q_t\circ\pi_C=q_s\bigl(\pi_C(\cdot+\xi)\bigr)
  \quad\text{for some}\quad \xi\in\Z.
\end{equation}

The plateau also rules out noninteger translations. Together with the evenness of the profiles, this yields the following recovery result.

\begin{lemma}\label{prop:profile-recovery}
  Let $s$ and $t$ be binary words of equal length $L\ge 1$, and let $q_s$ and $q_t$ be their associated warping profiles on $C=C(L)$. Then
  \begin{equation}\label{eq:profile-recovery}
    s\sim_\sharp t
    \Longleftrightarrow
    q_t=q_s\circ\tau
    \text{ for some Riemannian isometry }\tau\text{ of }(C,dr^2).
  \end{equation}
\end{lemma}

\begin{proof}
  If $s\sim_\sharp t$, then \cref{eq:equivalence-shift} shows that the right-hand side of \cref{eq:profile-recovery} holds for the isometry that maps $\pi_C(x)$ to $\pi_C(x+\xi)$, where $\xi\in\Z$ is as in \cref{eq:equivalence-shift}.

  To show the reverse implication, assume that $q_t=q_s\circ\tau$, where $\tau\colon C\to C$ is an isometry. Then $\tau$ is of the form $\pi_C(x)\mapsto\pi_C(\varepsilon x+\xi)$ for some $\varepsilon\in\{-1,1\}$ and $\xi\in\R$.

  If $\varepsilon=-1$, define the isometry $\widehat{\tau}(\pi_C(x))=\pi_C(x-\xi)$. Then $\tau = \iota_C\circ\widehat{\tau}$, so using $q_s=q_s\circ\iota_C$, we obtain
  \begin{equation*}
    q_t = q_s\circ\tau = q_s\circ\widehat{\tau}.
  \end{equation*}
  After replacing $\tau$ by $\widehat{\tau}$ if necessary and renaming the translation parameter, we may assume that
  \begin{equation*}
    \tau(\pi_C(x))=\pi_C(x+\xi).
  \end{equation*}

  By construction,
  \begin{equation}\label{eq:reverse-1}
    (q_t\circ\pi_C)^{-1}(3) = \Z + \frac{1}{2} + \interval{-\rho}{\rho},
  \end{equation}
  and
  \begin{equation}\label{eq:reverse-2}
    (q_s\circ\tau\circ\pi_C)^{-1}(3) = \Z + \frac{1}{2}-\xi + \interval{-\rho}{\rho}.
  \end{equation}
  By assumption, the sets in \cref{eq:reverse-1,eq:reverse-2} are equal. This forces $\xi\in\Z$, and \cref{eq:equivalence-shift} gives $s\sim_\sharp t$.
\end{proof}

We next translate the reverse-complement operation into matrix algebra. Define an operation $\sharp$ on $\R^{2\times2}$ by
\begin{equation}\label{eq:matrix-sharp}
  \begin{bmatrix}
    a & b\\
    c & d
  \end{bmatrix}
  ^\sharp
  :=
  \begin{bmatrix}
    d & b\\
    c & a
  \end{bmatrix}.
\end{equation}
Clearly this operation is involutive and a calculation shows it reverses products: for all $P,Q\in\R^{2\times2}$,
\begin{equation}\label{eq:matrix-sharp-product}
  (PQ)^\sharp=Q^\sharp P^\sharp.
\end{equation}
Given matrices $M_\bzero,M_\bone\in\R^{2\times2}$ indexed by $\bzero$ and $\bone$, and a binary word $s=s_1\cdots s_N$, define the \emph{word matrix}
\begin{equation}\label{eq:word-matrix}
  M_s:=M_{s_N}\cdots M_{s_1},
\end{equation}
where the right-hand side is the identity matrix if $s$ is empty. Then
\begin{equation}\label{eq:word-matrix-product}
  M_{st}=M_tM_s
\end{equation}
for all words $s$ and $t$. If $M_\bone=(M_\bzero)^\sharp$, then involutivity and \cref{eq:matrix-sharp-product} give
\begin{equation}\label{eq:word-matrix-sharp}
  M_{s^\sharp}=(M_s)^\sharp
\end{equation}
for every word $s$.

Given $\mu\ge 0$, $z\in\R$, and a word $s$ of length $L\ge1$, let
\begin{equation*}
  T_{x_1\shortleftarrow x_0}(\mu,z; s) \in \R^{2\times 2},
  \qquad x_0,x_1\in\interval{0}{L},
\end{equation*}
denote the transfer matrices of \cref{eq:schrodinger} on $\interval{0}{L}$ with potential $V_\mu=V_\mu(\cdot\,;q_s)\in C^\infty(\interval{0}{L})$. For fixed $\mu$ and $z$, let
\begin{equation}\label{eq:block-matrices}
  T_\bzero:=T_{1\shortleftarrow0}(\mu,z;\bzero)
  \quad\text{and}\quad
  T_\bone:=T_{1\shortleftarrow0}(\mu,z;\bone)
\end{equation}
be the transfer matrices across $\interval{0}{1}$ of the two words of length one, omitting $\mu$ and $z$ from the notation. Let $T_s$ be the word matrix~\labelcref{eq:word-matrix} determined by the word $s$ and the matrices $T_\bzero$ and $T_\bone$.

The following \namecref{prop:block-transfer} identifies the transfer matrix of a word with its word matrix.

\begin{proposition}\label{prop:block-transfer}
  Fix $\mu\ge 0$ and $z\in\R$, and let $T_\bzero$ and $T_\bone$ be as in \cref{eq:block-matrices}.
  \begin{enumerate}
    \item\label{item:block-sharp} $T_\bone=(T_\bzero)^\sharp$.

    \item\label{item:word-transfer} For every word $s=s_1\cdots s_L$ of length $L\ge 1$,
    \begin{equation}\label{eq:block-transfer}
      T_{L\shortleftarrow0}(\mu,z; s)=T_s.
    \end{equation}
  \end{enumerate}
\end{proposition}

\begin{proof}
\begin{enumerate}
\item Write $V_\bzero:=V_\mu(\cdot\,;q_\bzero)$ and $V_\bone:=V_\mu(\cdot\,;q_\bone)$ for the potentials determining $T_\bzero$ and $T_\bone$. Explicitly,
  \begin{equation*}
    V_\bzero=\frac{p_\bzero''}{p_\bzero}+\mu\,p_\bzero^{-4/m}
    \quad\text{and}\quad
    V_\bone=\frac{p_\bone''}{p_\bone}+\mu\,p_\bone^{-4/m}.
  \end{equation*}
  Since $p_\bone(x)=p_\bzero(1-x)$, it follows that $V_\bone(x)=V_\bzero(1-x)$.

  Let $f$ solve \cref{eq:schrodinger} with potential $V_\bzero$, and set $g(x):=f(1-x)$ and $\Sigma:=\diag(1,-1)$. Then $g$ solves \cref{eq:schrodinger} with potential $V_\bone$, and the state vectors of $f$ and $g$ satisfy
  \begin{equation*}
    T_\bone \Sigma T_\bzero
    \begin{bmatrix}
      f(0)\\f'(0)
    \end{bmatrix}
    = T_\bone \Sigma
    \begin{bmatrix}
      f(1)\\f'(1)
    \end{bmatrix}
    = T_\bone
    \begin{bmatrix}
      g(0)\\g'(0)
    \end{bmatrix}
    =
    \begin{bmatrix}
      g(1)\\g'(1)
    \end{bmatrix}
    = \Sigma
    \begin{bmatrix}
      f(0)\\f'(0)
    \end{bmatrix}.
  \end{equation*}
  Consequently, $T_\bone\Sigma T_\bzero=\Sigma$. Since $\det T_\bzero=1$ by \cref{eq:det-T-1}, we obtain $T_\bone = \Sigma T_\bzero^{-1}\Sigma=(T_\bzero)^\sharp$.

  \item By the composition rule~\labelcref{eq:transfer-composition},
  \begin{equation*}
    T_{L\shortleftarrow0}(\mu,z; s)
    = T_{L\shortleftarrow(L-1)}(\mu,z; s) \cdots T_{1\shortleftarrow0}(\mu,z; s).
  \end{equation*}
  On $\interval{j-1}{j}$, the potential is the translate of the single-block potential of type $s_j$. Hence
  \begin{equation*}
    T_{j\shortleftarrow(j-1)}(\mu,z; s)=T_{s_j},
  \end{equation*}
  and \cref{eq:block-transfer} follows, since $T_s=T_{s_L}\cdots T_{s_1}$ by \cref{eq:word-matrix}.\qedhere
\end{enumerate}
\end{proof}

\subsubsection*{Word algebra and combinatorics.}
The following proposition provides two binary words that are inequivalent under $\sim_\sharp$, but whose word matrices have the same off-diagonal entries whenever $M_\bone=(M_\bzero)^\sharp$. The corresponding warping profiles, displayed in \cref{fig:nine-letter-profiles}, are the ones used for the nonuniqueness result in \cref{sec:nonuniqueness-result}.

\begin{proposition}\label{prop:nine-letter-words}
  Let $M_\bzero\in\R^{2\times 2}$, set $M_\bone:=(M_\bzero)^\sharp$, and define $M_s$ for each binary word $s$ by \cref{eq:word-matrix}. Define
  \begin{equation}\label{eq:nine-letter-words}
    w:=\word{001\,011\,100}
    \quad\text{and}\quad
    \wt{w}:=\word{011\,001\,110}.
  \end{equation}
  Then:
  \begin{enumerate}
    \item\label{item:off-diagonal} the off-diagonal entries of $M_w$ and $M_{\wt{w}}$ agree;
    \item\label{item:no-isometry} $w\not\sim_\sharp\wt{w}$.
  \end{enumerate}
\end{proposition}

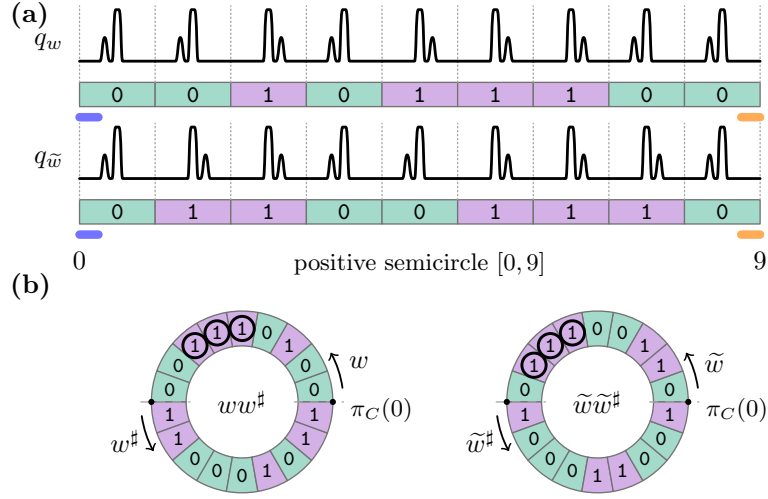
\begin{figure}[tbp]
  \centering
  \input{nonuniqueness-nine-letter-words.tikz}
  \caption{The warping profiles associated with the two nine-letter words. (a) The warping profiles $q_w$ and $q_{\wt{w}}$ on the positive semicircle; the blue and orange endpoint bars mark the intervals $I_\src$ and $I_\rcv$, on which both profiles equal $1$. (b) The words $ww^\sharp$ and $\wt{w}\wt{w}^\sharp$, arranged cyclically as they are read off the full profiles $q_w$ and $q_{\wt{w}}$ on $C$. The circled letters form the unique cyclic occurrence of $\word{111}$ in each word.}
  \label{fig:nine-letter-profiles}
\end{figure}

\begin{proof}
  \begin{enumerate}
    \item Let $\off(M)$ denote the off-diagonal part of a matrix $M$. We first prove the following auxiliary fact: if $P,Q\in\R^{2\times 2}$ have the same diagonal, then $QP^\sharp P$ and $Q^\sharp PP^\sharp$ have the same off-diagonal. Write $P=D+U$ and $Q=D+V$, where $D$ is diagonal and $U,V$ are off-diagonal. Then, for some $\alpha\in\R$,
      \begin{equation*}
        P^\sharp P=\alpha I_2+2D^\sharp U
        \quad\text{and}\quad
        PP^\sharp=\alpha I_2+2DU,
      \end{equation*}
      where $I_2$ is the $2\times 2$ identity matrix. Since $DD^\sharp=D^\sharp D$ and $V^\sharp=V$,
      \begin{align*}
        \off(QP^\sharp P)
        &= 2\off(DD^\sharp U) + \alpha \off(V)\\
        &= 2\off(D^\sharp DU) + \alpha \off(V^\sharp)\\
        &= \off(Q^\sharp PP^\sharp).
      \end{align*}

      Next, set $A:=M_\bzero$, $B:=M_\bone=A^\sharp$, $x:=\word{001}$, $\overleftarrow{x}:=\word{100}$, and $\overline{x}:=\word{110}$. Note that $x^\sharp=\word{011}$ and $\overline{x}=(\overleftarrow{x})^\sharp$, so $w=xx^\sharp\overleftarrow{x}$ and $\wt{w}=x^\sharp x\overline{x}$. Put $P:=M_x=BA^2$ and $Q:=M_{\overleftarrow{x}}=A^2B$. Then, by \cref{eq:word-matrix-product,eq:word-matrix-sharp},
      \begin{equation*}
        M_w = Q P^\sharp P
        \quad\text{and}\quad
        M_{\wt{w}} = Q^\sharp P P^\sharp,
      \end{equation*}
      so it only remains to show that $P$ and $Q$ have the same diagonal.

      Because
      \begin{equation*}
        (AB-BA)^\sharp
        = B^\sharp A^\sharp - A^\sharp B^\sharp
        = AB-BA,
      \end{equation*}
      the diagonal entries of $AB-BA$ are equal. Since $\trace(AB-BA)=0$, both diagonal entries vanish. The Cayley--Hamilton identity
      \begin{equation*}
        A^2-\trace(A)A+\det(A)I_2=0
      \end{equation*}
      gives
      \begin{equation*}
        P-Q = BA^2-A^2B = \trace(A)(BA-AB),
      \end{equation*}
      from which it follows that $P$ and $Q$ have the same diagonal.

    \item Observe that
      \begin{align*}
        ww^\sharp &= \word{001\,011\,100\,110\,001\,011},\\
        \wt{w}\wt{w}^\sharp &= \word{011\,001\,110\,100\,011\,001}.
      \end{align*}
      Note that both words contain a single cyclic occurrence of $\word{111}$. In other words, there exists a unique word $s$ such that $\word{111}s$ differs by a cyclic shift from $ww^\sharp$, and there exists a unique word $t$ such that $\word{111}t$ differs by a cyclic shift from $\wt{w}\wt{w}^\sharp$. Then $ww^\sharp$ and $\wt{w}\wt{w}^\sharp$ differ by a cyclic shift if and only if $s=t$. But $s_2=\bzero\neq\bone=t_2$, so that $w\not\sim_\sharp\wt{w}$.\qedhere
  \end{enumerate}
\end{proof}

\subsection{Nonisometric manifolds with equivalent cross-spectral data}\label{sec:nonuniqueness-result}

Recall that the construction of the pair of Riemannian manifolds $(\M_n,g_q)$ and $(\M_n,g_{\tq})$ in \cref{sec:ingredients} is determined by the parameters $n,L,\eta,(F,h),q,\tq$. We now specialize these parameters so that the pair realizes nonuniqueness.

\begin{param}
\item\label{par:dimension} The dimension $n\ge 2$ of $\M_n$ can be chosen freely.

\item\label{par:length} The length $L=9$, so that $C=\R/(2L\Z)=\R/(18\Z)$.

\item\label{par:width} The length $\eta=1/4$, so that the width of the source set $\src$ and receiver set $\rcv$ along $C$ is $1/2$.

\item\label{par:fiber} If $m=n-1\ge 2$, the fiber manifold $(F,h)$ is any smooth, closed, \emph{simply connected} Riemannian manifold of dimension $m$. If $m=1$, it is the unit circle $\bbS^1$ with the Riemannian metric induced by the Euclidean metric on $\R^2$.

\item\label{par:profiles} The functions $q=q_w$ and $\tq=q_{\wt{w}}$ are chosen to be the warping profiles associated with the nine-letter binary words $w,\wt{w}$ defined in \cref{eq:nine-letter-words} (see \cref{def:word-profile}). By \cref{rem:word-profiles-admissible}, they satisfy the requirements~\labelcref{eq:q-positive,eq:q-even,eq:q-matching} for warping profiles.
\end{param}
Substituting these parameters into \cref{eq:M-n,eq:src-rcv} gives the smooth $n$-manifold $\M$ and the open source and receiver sets $\src$ and $\rcv$:
\begin{equation}\label{eq:nonuniqueness-sets}
  \M = C\times F,
  \qquad
  \src=\pi_C\bigl(\ointerval{-\tfrac{1}{4}}{\tfrac{1}{4}}\bigr)\times F,
  \qquad
  \rcv=\pi_C\bigl(\ointerval{9-\tfrac{1}{4}}{9+\tfrac{1}{4}}\bigr)\times F,
\end{equation}
where $\pi_C\colon\R\to C$ is the quotient map. By \cref{eq:warped-product-metrics}, the corresponding warped-product metrics on $\M$ are
\begin{equation}\label{eq:nonuniqueness-metrics}
  g = g_{q_w} = dr^2 + a(r)^2h
  \quad\text{and}\quad
  \tg = g_{q_{\wt{w}}} = dr^2 + \wt{a}(r)^2h,
\end{equation}
where $a:=q_w^{2/m}$ and $\wt{a}:=q_{\wt{w}}^{2/m}$. The pair $(\M,g)$, $(\M,\tg)$ is thus the pair $(\M_n,g_q)$, $(\M_n,g_{\tq})$ of \cref{eq:standing-manifolds} with the parameters above, so the results of the preceding subsections apply to it. As noted in \cref{sec:ingredients}, the sets $\src$ and $\rcv$ are disjoint.

\begin{remark}\label{rem:uniform-fiber-choice}
  For every $n\ge 2$, the choice of $(F,h)$ as the sphere $\bbS^{n-1}$ with the metric induced by the Euclidean metric on $\R^n$ satisfies \cref{par:fiber}.
\end{remark}

Recall the definition of equivalent cross-spectral data in \cref{def:cross-spectral-data}, and that, by \cref{item:summary-equivalence} of \cref{prop:data-summary}, the same information can be expressed in terms of the local source-to-solution operator for the wave equation.
\begin{theorem}\label{prop:nonuniqueness}
  Let $n\ge 2$ be an integer.
  \begin{enumerate}
    \item\label{item:nonuniqueness} There exist a smooth, connected, closed $n$-manifold $\M$, two nonempty open sets $\src,\rcv\subset\M$, and two Riemannian metrics $g$ and $\tg$ on $\M$, such that the Riemannian manifolds $(\M,g)$ and $(\M,\tg)$ have equivalent cross-spectral data relative to the identity maps $\id_\src$ and $\id_\rcv$, but are not related by any Riemannian isometry.
    \item\label{item:nonuniqueness-concrete} Concretely, let the parameters be as in \crefrange{par:dimension}{par:profiles} above. Then the manifold $\M$, the sets $\src$ and $\rcv$, and the warped-product metrics $g$ and $\tg$ defined by \cref{eq:nonuniqueness-sets,eq:nonuniqueness-metrics} have the properties stated in \cref{item:nonuniqueness}.
  \end{enumerate}
\end{theorem}

\begin{remark}
  \begin{enumerate}
    \item Note the asymmetry between statements of uniqueness and nonuniqueness: uniqueness provides a Riemannian isometry that restricts to the identifications $\Phi_\src$ and $\Phi_\rcv$ on $\src$ and $\rcv$, respectively; nonuniqueness says that no Riemannian isometry at all exists between the manifolds.
    \item By \cref{item:classical-result} of \cref{prop:main-corollaries}, a necessary condition for nonuniqueness is $\src\cap\rcv=\emptyset$; see also~\cite[Corollary~2]{MR3826551}.
  \end{enumerate}
\end{remark}

The restrictions on the fiber manifold $(F,h)$ in \cref{par:fiber}, beyond smoothness, closedness, and connectedness, are used only to rule out isometries via \cref{prop:isometry-splitting}. We first establish, without these additional restrictions, that the two manifolds have equivalent cross-spectral data and that no Riemannian isometry $\tau$ of $C$ satisfies $q=\tq\circ\tau$; the proof of \cref{prop:nonuniqueness} then follows.

\begin{proposition}\label{prop:word-manifolds}
  Let the parameters be as in \crefrange{par:dimension}{par:profiles} above, except for $(F,h)$, which may be any smooth, connected, closed Riemannian manifold of dimension $n-1$. Let the manifolds $(\M,g)$ and $(\M,\tg)$ and the sets $\src$ and $\rcv$ be defined by \cref{eq:nonuniqueness-sets,eq:nonuniqueness-metrics}.

  The Riemannian $n$-manifolds $(\M,g)$ and $(\M,\tg)$ have equivalent cross-spectral data relative to the identity maps $\id_\src$ and $\id_\rcv$. Moreover, no Riemannian isometry $\tau\colon C\to C$ satisfies $q=\tq\circ\tau$.
\end{proposition}

\begin{proof}
  Fix $k=0,1,\dots$ and $z<0$, and set $\mu:=\mu_k$. By \cref{par:profiles}, $q=q_w$ and $\tq=q_{\wt{w}}$, so the transfer matrices in \cref{eq:transfer-matrix-entries} are $\Tmat{0}{L}(\mu,z)=T_{L\shortleftarrow0}(\mu,z;w)$ and $\tTmat{0}{L}(\mu,z)=T_{L\shortleftarrow0}(\mu,z;\wt{w})$. By \cref{item:word-transfer} of \cref{prop:block-transfer}, these are the word matrices $T_w$ and $T_{\wt{w}}$, and by \cref{item:block-sharp}, the equality $T_\bone=(T_\bzero)^\sharp$ holds. Applying \cref{item:off-diagonal} of \cref{prop:nine-letter-words} with $M_\bzero=T_\bzero$ and $M_\bone=T_\bone$ then shows that $T_w$ and $T_{\wt{w}}$ have equal off-diagonal entries. Since $k$ and $z$ were arbitrary, \cref{prop:transfer-reduction} implies that $(\M,g)$ and $(\M,\tg)$ have equivalent cross-spectral data relative to the identity maps $\id_\src$ and $\id_\rcv$.

  Now $w\not\sim_\sharp\wt{w}$ by \cref{item:no-isometry} of \cref{prop:nine-letter-words}, so \cref{prop:profile-recovery} implies that no isometry $\tau\colon C\to C$ satisfies $q=\tq\circ\tau$.
\end{proof}

\begin{proof}[Proof of \cref{prop:nonuniqueness}]
  Since \cref{item:nonuniqueness-concrete} implies \cref{item:nonuniqueness}, it suffices to prove \cref{item:nonuniqueness-concrete}.

  By \cref{prop:word-manifolds}, the manifolds $(\M,g)$ and $(\M,\tg)$ have equivalent cross-spectral data relative to the identity maps $\id_\src$ and $\id_\rcv$. By \cref{par:fiber}, one of \cref{cond:fiber-circle,cond:fiber-simply-connected} of \cref{prop:isometry-splitting} holds: for $n\ge 3$ the fiber is simply connected, and for $n=2$ it is the unit circle, for which $L_F=\pi$ and $\min_C a=1<9/\pi=L/L_F$. Consequently, any isometry between the manifolds would induce an isometry $\tau\colon C\to C$ such that $q=\tq\circ\tau$, but by \cref{prop:word-manifolds} there exists no such isometry $\tau$. Therefore, there exists no Riemannian isometry between the manifolds $(\M,g)$ and $(\M,\tg)$.
\end{proof}

%% file: nonuniqueness-geometry.tikz
\begingroup
\begin{tikzpicture}[
  line cap=round,
  line join=round,
  every node/.style={font=\small}
]
  \tikzset{
    source arc/.style={draw=blue!55, line width=6.5pt},
    receiver arc/.style={draw=orange!65, line width=6.5pt},
    source bar/.style={draw=blue!55, line width=3.5pt},
    receiver bar/.style={draw=orange!65, line width=3.5pt},
    base curve/.style={draw=black!65, line width=0.8pt},
    common profile/.style={draw=black, line width=1.05pt},
    first profile/.style={draw=black, line width=1.05pt},
    second profile/.style={draw=black!70, line width=1.05pt, densely dashed},
    guide/.style={draw=black!45, line width=0.6pt},
    figure arrow/.style={draw=black, line width=0.7pt},
    tick/.style={draw=black!70, line width=0.6pt},
    tick label/.style={anchor=north, text height=1.6ex, text depth=0.25ex}
  }

  \node[anchor=west, font=\small\bfseries] at (-6.10,3.02) {(a)};
  \begin{scope}[
    shift={(-3.90,1.55)},
    every node/.append style={font=\footnotesize}
  ]
    \def\CircleRadius{1.10}

    \draw[guide, dashed] (-1.32,0) -- (1.32,0);

    \draw[source arc] (-25:\CircleRadius) arc[start angle=-25,end angle=25,radius=\CircleRadius];
    \draw[receiver arc] (155:\CircleRadius) arc[start angle=155,end angle=205,radius=\CircleRadius];
    \draw[base curve] (0,0) circle[radius=\CircleRadius];
    \draw[figure arrow, ->]
      (35:\CircleRadius) arc[start angle=35,end angle=80,radius=\CircleRadius];

    \fill (-\CircleRadius,0) circle[radius=1.35pt];
    \fill ( \CircleRadius,0) circle[radius=1.35pt];

    \node at (0,0.64) {$j_+$};
    \node at (0,-0.64) {$j_-$};
    \node[orange!90!black, anchor=south east] at (-0.82,0.58) {$C_\rcv$};
    \node[blue!70!black, anchor=south west] at (0.82,0.58) {$C_\src$};
    \node[anchor=south] at (0,1.22) {$C$};

    \node[anchor=north, align=center] (fixedL) at (-1.50,-0.38)
      {$\pi_C(L)$\\[-2pt]$=\pi_C(-L)$};
    \node[anchor=north] (fixedzero) at (1.50,-0.38) {$\pi_C(0)$};
    \draw[figure arrow, ->] (fixedL.north) -- (-\CircleRadius,0);
    \draw[figure arrow, ->] (fixedzero.north) -- (\CircleRadius,0);
  \end{scope}

  \node[anchor=north, align=center, font=\footnotesize] at (-3.90,-0.02) {$\displaystyle
    \begin{aligned}
      \M_n&=C\times F,\\
      \src&=C_\src\times F,\\
      \rcv&=C_\rcv\times F.
    \end{aligned}$};

  \begin{scope}[shift={(1.90,0.25)}]
  \node[anchor=west, font=\small\bfseries] at (-3.82,2.65) {(b)};

  \def\xL{3.55}
  \def\xEta{1.00}
  \def\xLEta{2.55}

  \fill[orange!13] (-\xL,-0.06) rectangle (-\xLEta,2.42);
  \fill[blue!11]   (-\xEta,-0.06) rectangle ( \xEta,2.42);
  \fill[orange!13] ( \xLEta,-0.06) rectangle ( \xL,2.42);
  \node[orange!90!black, anchor=north] at (-3.05,2.30) {$C_\rcv$};
  \node[blue!70!black, anchor=north] at (0,2.30) {$C_\src$};
  \node[orange!90!black, anchor=north] at (3.05,2.30) {$C_\rcv$};

  \draw[first profile] (-1.45,2.64) -- (-0.98,2.64);
  \node[anchor=west] at (-0.90,2.64) {$q$};
  \draw[second profile] (0.12,2.64) -- (0.59,2.64);
  \node[anchor=west] at (0.67,2.64) {$\widetilde q$};

  \def\ProfileLevel{1.45}
  \def\ProfileAmplitude{0.30}
  \def\PerturbationAmplitude{0.20}
  \def\ProfileStart{1.18}
  \def\ProfileEnd{2.37}
  \def\ProfileWidth{1.19}

  \draw[common profile]
    plot[domain=-3.55:-2.37, samples=81, variable=\x]
      (\x,{\ProfileLevel-\ProfileAmplitude*cos(360*\x/3.55)});
  \draw[common profile]
    plot[domain=-1.18:1.18, samples=121, variable=\x]
      (\x,{\ProfileLevel-\ProfileAmplitude*cos(360*\x/3.55)});
  \draw[common profile]
    plot[domain=2.37:3.55, samples=81, variable=\x]
      (\x,{\ProfileLevel-\ProfileAmplitude*cos(360*\x/3.55)});

  \draw[second profile]
    ({-\ProfileEnd},
     {\ProfileLevel-\ProfileAmplitude*cos(360*(-\ProfileEnd)/3.55)})
    plot[domain=0.002:0.998, samples=201, variable=\t]
      ({-\ProfileEnd+\ProfileWidth*\t},
       {\ProfileLevel
        -\ProfileAmplitude*cos(360*(-\ProfileEnd+\ProfileWidth*\t)/3.55)
        +\PerturbationAmplitude
         *exp(1.6-0.40/(\t*(1-\t)))*sin(360*\t)})
    -- ({-\ProfileStart},
        {\ProfileLevel-\ProfileAmplitude*cos(360*(-\ProfileStart)/3.55)});
  \draw[second profile]
    (\ProfileStart,
     {\ProfileLevel-\ProfileAmplitude*cos(360*\ProfileStart/3.55)})
    plot[domain=0.002:0.998, samples=201, variable=\t]
      ({\ProfileStart+\ProfileWidth*\t},
       {\ProfileLevel
        -\ProfileAmplitude*cos(360*(\ProfileStart+\ProfileWidth*\t)/3.55)
        -\PerturbationAmplitude
         *exp(1.6-0.40/(\t*(1-\t)))*sin(360*\t)})
    -- (\ProfileEnd,
        {\ProfileLevel-\ProfileAmplitude*cos(360*\ProfileEnd/3.55)});

  \draw[first profile]
    ({-\ProfileEnd},
     {\ProfileLevel-\ProfileAmplitude*cos(360*(-\ProfileEnd)/3.55)})
    plot[domain=0.002:0.998, samples=201, variable=\t]
      ({-\ProfileEnd+\ProfileWidth*\t},
       {\ProfileLevel
        -\ProfileAmplitude*cos(360*(-\ProfileEnd+\ProfileWidth*\t)/3.55)
        -\PerturbationAmplitude
         *exp(1.6-0.40/(\t*(1-\t)))*sin(360*\t)})
    -- ({-\ProfileStart},
        {\ProfileLevel-\ProfileAmplitude*cos(360*(-\ProfileStart)/3.55)});
  \draw[first profile]
    (\ProfileStart,
     {\ProfileLevel-\ProfileAmplitude*cos(360*\ProfileStart/3.55)})
    plot[domain=0.002:0.998, samples=201, variable=\t]
      ({\ProfileStart+\ProfileWidth*\t},
       {\ProfileLevel
        -\ProfileAmplitude*cos(360*(\ProfileStart+\ProfileWidth*\t)/3.55)
        +\PerturbationAmplitude
         *exp(1.6-0.40/(\t*(1-\t)))*sin(360*\t)})
    -- (\ProfileEnd,
        {\ProfileLevel-\ProfileAmplitude*cos(360*\ProfileEnd/3.55)});

  \draw[base curve] (-3.70,0) -- (3.70,0);
  \foreach \x in {-3.55,-2.55,-1.00,0,1.00,2.55,3.55}
    \draw[tick] (\x,-0.07) -- (\x,0.07);
  \node[tick label] at (-3.55,-0.10) {$-L$};
  \node[tick label] at (-2.55,-0.10) {$-L+\eta$};
  \node[tick label] at (-1.00,-0.10) {$-\eta$};
  \node[tick label] at (0,-0.10) {$0$};
  \node[tick label] at (1.00,-0.10) {$\eta$};
  \node[tick label] at (2.55,-0.10) {$L-\eta$};
  \node[tick label] at (3.55,-0.10) {$L$};

  \draw[source bar] (0,-0.72) -- (1.00,-0.72);
  \draw[receiver bar] (2.55,-0.72) -- (3.55,-0.72);
  \node[blue!70!black, anchor=north] at (0.50,-0.82) {$I_\src$};
  \node[orange!90!black, anchor=north] at (3.05,-0.82) {$I_\rcv$};
  \draw[guide] (0,-1.22) -- (0,-1.38) -- (3.55,-1.38) -- (3.55,-1.22);
  \node[anchor=north] at (1.78,-1.43)
    {semicircle coordinate $x\in[0,L]$};
  \end{scope}
\end{tikzpicture}
\endgroup

%% file: nonuniqueness-word-profile.tikz
\begingroup
\definecolor{wordzero}{RGB}{163,218,201}
\definecolor{wordone}{RGB}{211,177,231}
\begin{tikzpicture}[
  line cap=round,
  line join=round,
  every node/.style={font=\small}
]
  \tikzset{
    profile/.style={draw=black, line width=1.05pt},
    axis/.style={draw=black!65, line width=0.65pt},
    guide/.style={draw=black!35, line width=0.5pt},
    figure arrow/.style={draw=black, line width=0.7pt},
    zero cell/.style={draw=black!55, fill=wordzero, line width=0.55pt},
    one cell/.style={draw=black!55, fill=wordone, line width=0.55pt}
  }

  \def\DrawZeroBlock#1#2#3#4{%
    \begin{scope}[shift={(#1,#2)},xscale=#3,yscale=#4]
      \draw[profile]
        (0,0) -- (0.27,0)
        .. controls (0.30,0) and (0.30,0.46) .. (0.333,0.46)
        .. controls (0.366,0.46) and (0.366,0) .. (0.40,0)
        -- (0.43,0)
        .. controls (0.455,0) and (0.445,1) .. (0.48,1)
        -- (0.52,1)
        .. controls (0.555,1) and (0.545,0) .. (0.57,0)
        -- (1,0);
    \end{scope}%
  }
  \def\DrawOneBlock#1#2#3#4{%
    \begin{scope}[shift={(#1,#2)},xscale=#3,yscale=#4]
      \draw[profile]
        (0,0) -- (0.43,0)
        .. controls (0.455,0) and (0.445,1) .. (0.48,1)
        -- (0.52,1)
        .. controls (0.555,1) and (0.545,0) .. (0.57,0)
        -- (0.60,0)
        .. controls (0.634,0) and (0.634,0.46) .. (0.667,0.46)
        .. controls (0.70,0.46) and (0.70,0) .. (0.73,0)
        -- (1,0);
    \end{scope}%
  }

  \node[anchor=west, font=\small\bfseries] at (-5.95,2.82) {(a)};
  \DrawZeroBlock{-5.45}{1.42}{2.20}{0.78}
  \DrawOneBlock{-5.45}{0.05}{2.20}{0.78}
  \draw[axis] (-5.50,1.25) -- (-3.20,1.25);
  \draw[axis] (-5.50,-0.12) -- (-3.20,-0.12);
  \foreach \y in {1.25,-0.12} {
    \draw[axis] (-5.45,\y-0.05) -- (-5.45,\y+0.05);
    \draw[axis] (-3.25,\y-0.05) -- (-3.25,\y+0.05);
  }
  \node[anchor=east] at (-5.62,1.77) {$p_{\mathtt{0}}$};
  \node[anchor=east] at (-5.62,0.40) {$p_{\mathtt{1}}$};
  \node[font=\footnotesize, anchor=north] at (-5.45,1.19) {$0$};
  \node[font=\footnotesize, anchor=north] at (-3.25,1.19) {$1$};
  \node[font=\footnotesize, anchor=north] at (-5.45,-0.18) {$0$};
  \node[font=\footnotesize, anchor=north] at (-3.25,-0.18) {$1$};
  \node[font=\footnotesize, anchor=south east] at (-4.95,2.04) {marker};
  \node[font=\footnotesize, anchor=south west, xshift=-2em]
    at (-3.98,2.36) {plateau};
  \draw[figure arrow, ->] (-4.93,2.02) -- (-4.72,1.80);
  \draw[figure arrow, ->] (-4.23,2.43) -- (-4.34,2.20);

  \node[anchor=west, font=\small\bfseries] at (-2.18,2.58) {(b)};
  \def\BlockWidth{0.92}
  \def\ProfileLeft{-2.00}
  \def\ProfileBase{1.30}
  \DrawZeroBlock{-2.00}{1.30}{0.92}{0.72}
  \DrawZeroBlock{-1.08}{1.30}{0.92}{0.72}
  \DrawOneBlock{-0.16}{1.30}{0.92}{0.72}
  \DrawZeroBlock{0.76}{1.30}{0.92}{0.72}
  \foreach \x in {-2.00,-1.08,-0.16,0.76,1.68}
    \draw[guide, densely dotted] (\x,0.78) -- (\x,2.08);
  \foreach \bit [count=\i from 0] in {0,0,1,0} {
    \pgfmathsetmacro{\cellleft}{-2.00+0.92*\i}
    \ifnum\bit=0
      \path[zero cell] (\cellleft,0.49) rectangle ++(0.92,0.31);
    \else
      \path[one cell] (\cellleft,0.49) rectangle ++(0.92,0.31);
    \fi
    \pgfmathsetmacro{\cellmid}{\cellleft+0.46}
    \node[font=\footnotesize] at (\cellmid,0.645) {$\mathtt{\bit}$};
  }
  \node[anchor=east] at (-2.14,1.66) {$q_s$};
  \node[anchor=north] at (-2.00,0.43) {$0$};
  \node[anchor=north] at (1.68,0.43) {$4$};
  \node[anchor=north] at (-0.16,0.02) {$s=\mathtt{0010}$};

  \node[anchor=west, font=\small\bfseries] at (2.56,2.58) {(c)};
  \begin{scope}[shift={(4.00,1.03)}]
    \def\OuterRadius{1.24}
    \def\InnerRadius{0.76}
    \foreach \bit [count=\i from 0] in {0,0,1,0,1,0,1,1} {
      \pgfmathsetmacro{\startangle}{45*\i}
      \pgfmathsetmacro{\endangle}{45*(\i+1)}
      \pgfmathsetmacro{\midangle}{45*(\i+0.5)}
      \ifnum\bit=0
        \path[zero cell]
          (\startangle:\InnerRadius) -- (\startangle:\OuterRadius)
          arc[start angle=\startangle,end angle=\endangle,radius=\OuterRadius]
          -- (\endangle:\InnerRadius)
          arc[start angle=\endangle,end angle=\startangle,radius=\InnerRadius]
          -- cycle;
      \else
        \path[one cell]
          (\startangle:\InnerRadius) -- (\startangle:\OuterRadius)
          arc[start angle=\startangle,end angle=\endangle,radius=\OuterRadius]
          -- (\endangle:\InnerRadius)
          arc[start angle=\endangle,end angle=\startangle,radius=\InnerRadius]
          -- cycle;
      \fi
      \node[font=\footnotesize] at (\midangle:1.00) {$\mathtt{\bit}$};
    }
    \draw[guide, dashed] (-1.40,0) -- (-0.70,0);
    \draw[guide, dashed] (0.70,0) -- (1.40,0);
    \fill (0:\OuterRadius) circle[radius=1.25pt];
    \fill (180:\OuterRadius) circle[radius=1.25pt];
    \draw[figure arrow, ->]
      (12:1.40) arc[start angle=12,end angle=38,radius=1.40];
    \node at (0,0.17) {$ss^\sharp$};
    \node[font=\footnotesize] at (0,-0.18) {$\mathtt{0010\,1011}$};
    \node[anchor=south] at (0,1.35) {$s$};
    \node[anchor=north] at (0,-1.35) {$s^\sharp$};
    \node[font=\footnotesize, anchor=west] at (1.31,-0.20) {$\pi_C(0)$};
    \node[font=\footnotesize, anchor=east] at (-1.24,-0.20) {$\pi_C(4)$};
  \end{scope}
\end{tikzpicture}
\endgroup

%% file: nonuniqueness-nine-letter-words.tikz
\begingroup
\definecolor{wordzero}{RGB}{163,218,201}
\definecolor{wordone}{RGB}{211,177,231}
\begin{tikzpicture}[
  line cap=round,
  line join=round,
  every node/.style={font=\small}
]
  \tikzset{
    profile/.style={draw=black, line width=1.05pt},
    axis/.style={draw=black!65, line width=0.65pt},
    guide/.style={draw=black!35, line width=0.5pt},
    figure arrow/.style={draw=black, line width=0.7pt},
    zero cell/.style={draw=black!55, fill=wordzero, line width=0.5pt},
    one cell/.style={draw=black!55, fill=wordone, line width=0.5pt},
    source bar/.style={draw=blue!55, line width=3.2pt},
    receiver bar/.style={draw=orange!65, line width=3.2pt},
    run mark/.style={draw=black, line width=1.1pt}
  }

  \def\DrawZeroBlock#1#2#3#4#5{%
    \begin{scope}[shift={(#1,#2)},xscale=#3,yscale=#4]
      \draw[#5]
        (0,0) -- (0.27,0)
        .. controls (0.30,0) and (0.30,0.46) .. (0.333,0.46)
        .. controls (0.366,0.46) and (0.366,0) .. (0.40,0)
        -- (0.43,0)
        .. controls (0.455,0) and (0.445,1) .. (0.48,1)
        -- (0.52,1)
        .. controls (0.555,1) and (0.545,0) .. (0.57,0)
        -- (1,0);
    \end{scope}%
  }
  \def\DrawOneBlock#1#2#3#4#5{%
    \begin{scope}[shift={(#1,#2)},xscale=#3,yscale=#4]
      \draw[#5]
        (0,0) -- (0.43,0)
        .. controls (0.455,0) and (0.445,1) .. (0.48,1)
        -- (0.52,1)
        .. controls (0.555,1) and (0.545,0) .. (0.57,0)
        -- (0.60,0)
        .. controls (0.634,0) and (0.634,0.46) .. (0.667,0.46)
        .. controls (0.70,0.46) and (0.70,0) .. (0.73,0)
        -- (1,0);
    \end{scope}%
  }

  \node[anchor=west, font=\small\bfseries] at (-5.55,2.95) {(a)};
  \def\ProfileLeft{-4.50}
  \def\BlockWidth{1.00}
  \def\FirstBase{2.35}
  \def\SecondBase{0.80}

  \foreach \bit [count=\i from 0] in {0,0,1,0,1,1,1,0,0} {
    \pgfmathsetmacro{\cellleft}{\ProfileLeft+\BlockWidth*\i}
    \ifnum\bit=0
      \path[zero cell] (\cellleft,1.75) rectangle ++(\BlockWidth,0.32);
    \else
      \path[one cell] (\cellleft,1.75) rectangle ++(\BlockWidth,0.32);
    \fi
    \pgfmathsetmacro{\cellmid}{\cellleft+0.5*\BlockWidth}
    \node[font=\footnotesize] at (\cellmid,1.91) {$\mathtt{\bit}$};
  }
  \foreach \bit [count=\i from 0] in {0,1,1,0,0,1,1,1,0} {
    \pgfmathsetmacro{\cellleft}{\ProfileLeft+\BlockWidth*\i}
    \ifnum\bit=0
      \path[zero cell] (\cellleft,0.20) rectangle ++(\BlockWidth,0.32);
    \else
      \path[one cell] (\cellleft,0.20) rectangle ++(\BlockWidth,0.32);
    \fi
    \pgfmathsetmacro{\cellmid}{\cellleft+0.5*\BlockWidth}
    \node[font=\footnotesize] at (\cellmid,0.36) {$\mathtt{\bit}$};
  }

  \foreach \i in {0,...,9} {
    \pgfmathsetmacro{\seam}{\ProfileLeft+\BlockWidth*\i}
    \draw[guide, densely dotted] (\seam,2.07) -- (\seam,3.08);
    \draw[guide, densely dotted] (\seam,0.52) -- (\seam,1.53);
  }

  \foreach \bit [count=\i from 0] in {0,0,1,0,1,1,1,0,0} {
    \pgfmathsetmacro{\blockleft}{\ProfileLeft+\BlockWidth*\i}
    \ifnum\bit=0
      \DrawZeroBlock{\blockleft}{\FirstBase}{\BlockWidth}{0.68}{profile}
    \else
      \DrawOneBlock{\blockleft}{\FirstBase}{\BlockWidth}{0.68}{profile}
    \fi
  }
  \foreach \bit [count=\i from 0] in {0,1,1,0,0,1,1,1,0} {
    \pgfmathsetmacro{\blockleft}{\ProfileLeft+\BlockWidth*\i}
    \ifnum\bit=0
      \DrawZeroBlock{\blockleft}{\SecondBase}{\BlockWidth}{0.68}{profile}
    \else
      \DrawOneBlock{\blockleft}{\SecondBase}{\BlockWidth}{0.68}{profile}
    \fi
  }

  \node[anchor=east] at (-4.58,2.60) {$q_w$};
  \node[anchor=east] at (-4.58,1.05) {$q_{\widetilde w}$};
  \node[font=\footnotesize, anchor=north] at (0,-0.04)
    {positive semicircle $[0,9]$};
  \node[anchor=north] at (-4.50,-0.04) {$0$};
  \node[anchor=north] at (4.50,-0.04) {$9$};

  \draw[source bar] (-4.50,1.61) -- (-4.25,1.61);
  \draw[receiver bar] (4.25,1.61) -- (4.50,1.61);
  \draw[source bar] (-4.50,0.06) -- (-4.25,0.06);
  \draw[receiver bar] (4.25,0.06) -- (4.50,0.06);

  \node[anchor=west, font=\small\bfseries] at (-5.55,-0.62) {(b)};

  \def\DrawWordCircle#1#2#3#4#5#6{%
    \begin{scope}[shift={(#1,#2)}]
      \def\OuterRadius{1.20}
      \def\InnerRadius{0.74}
      \foreach \bit [count=\i from 0] in {#3} {
        \pgfmathsetmacro{\startangle}{20*\i}
        \pgfmathsetmacro{\endangle}{20*(\i+1)}
        \pgfmathsetmacro{\midangle}{20*(\i+0.5)}
        \ifnum\bit=0
          \path[zero cell]
            (\startangle:\InnerRadius) -- (\startangle:\OuterRadius)
            arc[start angle=\startangle,end angle=\endangle,radius=\OuterRadius]
            -- (\endangle:\InnerRadius)
            arc[start angle=\endangle,end angle=\startangle,radius=\InnerRadius]
            -- cycle;
        \else
          \path[one cell]
            (\startangle:\InnerRadius) -- (\startangle:\OuterRadius)
            arc[start angle=\startangle,end angle=\endangle,radius=\OuterRadius]
            -- (\endangle:\InnerRadius)
            arc[start angle=\endangle,end angle=\startangle,radius=\InnerRadius]
            -- cycle;
        \fi
        \node[font=\scriptsize] at (\midangle:0.97) {$\mathtt{\bit}$};
      }
      \draw[guide, dashed] (-1.34,0) -- (-0.69,0);
      \draw[guide, dashed] (0.69,0) -- (1.34,0);
      \fill (0:\OuterRadius) circle[radius=1.15pt];
      \fill (180:\OuterRadius) circle[radius=1.15pt];
      \draw[figure arrow, ->]
        (8:1.34) arc[start angle=8,end angle=29,radius=1.34];
      \node at (18.5:1.63) {#5};
      \draw[figure arrow, ->]
        (188:1.34) arc[start angle=188,end angle=209,radius=1.34];
      \node at (198.5:1.63) {#6};
      \node[align=center] at (0,0) {#4};
    \end{scope}%
  }

  \DrawWordCircle{-2.35}{-2.15}{0,0,1,0,1,1,1,0,0,1,1,0,0,0,1,0,1,1}{$ww^\sharp$}{$w$}{$w^\sharp$}
  \DrawWordCircle{2.35}{-2.15}{0,1,1,0,0,1,1,1,0,1,0,0,0,1,1,0,0,1}{$\widetilde w\widetilde w^\sharp$}{$\widetilde w$}{$\widetilde w^\sharp$}

  \node[font=\footnotesize, anchor=west] at (-1.05,-2.26) {$\pi_C(0)$};
  \node[font=\footnotesize, anchor=west] at (3.65,-2.26) {$\pi_C(0)$};

  \begin{scope}[shift={(-2.35,-2.15)}]
    \foreach \angle in {90,110,130} {
      \draw[run mark] (\angle:0.97) circle[radius=0.16];
    }
  \end{scope}
  \begin{scope}[shift={(2.35,-2.15)}]
    \foreach \angle in {110,130,150} {
      \draw[run mark] (\angle:0.97) circle[radius=0.16];
    }
  \end{scope}
\end{tikzpicture}
\endgroup

%% file: appendix.tex
\section{Ball inclusions determine the manifold}\label{sec:HLOS-balls}

The following \namecref{prop:HLOS} isolates the geometric core of the argument of Helin, Lassas, Oksanen, and Saksala~\cite{MR3826551}. Their reconstruction is formulated in terms of a local source-to-solution operator for the wave equation, with sources and observations in the same set rather than in the possibly disjoint sets considered here, and its geometric ingredients are distributed across several statements. Here we assemble these ingredients into a result stated purely in terms of ball inclusions.

Related geometric reconstruction methods for manifolds with boundary were developed earlier in~\cite{MR1889089} and, for disjoint data, in~\cite{MR3192525}. In~\cite{MR4833220}, a different construction for closed manifolds with disjoint data uses domain-of-influence relations and normal geodesics from the boundary of the observation set. See~\cite{MR4969293} for approximate and stable reconstruction from noisy distance data.

\begin{theorem}[after Helin, Lassas, Oksanen, and Saksala~\cite{MR3826551}]\label{prop:HLOS}
  Let $(N,g_N)$ and $(\tN,g_{\tN})$ be smooth, connected, complete Riemannian manifolds without boundary, of dimension at least two. Write $d_N$ for the distance function of $N$ and $B_N(x,\rho)$ for the open geodesic ball of radius $\rho>0$ centered at $x\in N$, and write $d_{\tN}$ and $B_{\tN}$ for the corresponding objects on $\tN$. Let $U\subset N$ and $\wt{U}\subset\tN$ be nonempty open sets with a diffeomorphism $\Phi_U\colon U\to\wt{U}$.

  The following are equivalent:
  \begin{cond}

  \item\label{item:HLOS-1} There exists a Riemannian isometry $\Phi\colon(N,g_N)\to(\tN,g_{\tN})$ such that $\Phi|_U=\Phi_U$.

  \item\label{item:HLOS-2} For every $x_0,x_1,x_2\in U$ and every $\rho_0,\rho_1,\rho_2\in\ointerval{0}{\infty}$, with $\tx_j:=\Phi_U(x_j)$ for $j=0,1,2$, it holds that
    \begin{equation}\label{eq:HLOS-iff-1}
      B_N(x_0,\rho_0)
      \subset
      \overline{B_N(x_1,\rho_1) \cup B_N(x_2,\rho_2)}
    \end{equation}
    if and only if
    \begin{equation}\label{eq:HLOS-iff-2}
      B_{\tN}(\tx_0,\rho_0)
      \subset
      \overline{B_{\tN}(\tx_1,\rho_1) \cup B_{\tN}(\tx_2,\rho_2)}.
    \end{equation}
  \end{cond}
\end{theorem}

\begin{proof}
  (\labelcref{item:HLOS-1} $\Rightarrow$ \labelcref{item:HLOS-2}) Fix $x_0,x_1,x_2\in U$ and $\rho_0,\rho_1,\rho_2\in\ointerval{0}{\infty}$, and let $\Phi$ be a Riemannian isometry that extends $\Phi_U$. Then it is an isometry from $(N,d_N)$ onto $(\tN,d_{\tN})$, and hence $\Phi\bigl(B_N(x_j,\rho_j)\bigr) = B_{\tN}(\tx_j,\rho_j)$ for $j=0,1,2$. Since $\Phi$ is a homeomorphism,
  \begin{align*}
    \Phi \left(\overline{B_N(x_1,\rho_1) \cup B_N(x_2,\rho_2)}\right)
    &= \overline{\Phi\bigl(B_N(x_1,\rho_1)\bigr) \cup \Phi\bigl(B_N(x_2,\rho_2)\bigr)}\\
    &= \overline{B_{\tN}(\tx_1,\rho_1) \cup B_{\tN}(\tx_2,\rho_2)}.
  \end{align*}
  Therefore, the inclusions in \cref{eq:HLOS-iff-1,eq:HLOS-iff-2} are equivalent.

  (\labelcref{item:HLOS-2} $\Rightarrow$ \labelcref{item:HLOS-1}) Assume that \cref{item:HLOS-2} holds. We show that \cref{item:HLOS-1} holds by a sequence of steps, following~\cite[Section~5]{MR3826551}. In our argument, \cref{item:HLOS-2} takes the place of~\cite[Lemma~10]{MR3826551}, which converts inclusions of balls into statements about the local source-to-solution operator.

  \Cref{step:HLOS-dist,step:HLOS-Riemannian-isometry} identify the distance function and Riemannian metric on $U$, providing the two geometric inputs to~\cite[Theorem~7]{MR3826551}. Following the proof of that theorem, \crefrange{step:HLOS-cut-distance}{step:HLOS-distance-data} identify the cut distances, then the point-to-$U$ distances along pre-cut geodesics, and finally the full families of restricted distance functions. \Cref{step:HLOS-isometry} then applies~\cite[Theorem~8]{MR3826551} to both manifolds and compares the resulting reconstructions, as in the proof of~\cite[Theorem~2]{MR3826551}. Unlike the conclusion of that theorem, \cref{item:HLOS-1} also asserts that the isometry extends $\Phi_U$.

  \begin{step}
  \item\label{step:HLOS-dist} For every $x,y\in U$, we have $d_N(x,y)=d_{\tN}(\Phi_U(x), \Phi_U(y))$.

    \begin{stepproof}
      Fix $x,y\in U$ and $\varepsilon>0$, and set $\tx:=\Phi_U(x)$ and $\ty:=\Phi_U(y)$. Define
      \begin{align*}
        r &:= \inf\left\{\rho>0 : B_N(x,\varepsilon) \subset \overline{B_N(y,\rho)}\right\},\\
        \tilde{r} &:= \inf\left\{\rho>0 : B_{\tN}(\tx,\varepsilon) \subset \overline{B_{\tN}(\ty,\rho)}\right\}.
      \end{align*}
      A straightforward argument shows that $r\in\interval{d_N(x,y)}{d_N(x,y)+\varepsilon}$ and analogously $\tilde{r}\in\interval{d_{\tN}(\tx,\ty)}{d_{\tN}(\tx,\ty)+\varepsilon}$. Applying \cref{item:HLOS-2} for every $\rho>0$ with $x_0=x$, $x_1=x_2=y$, $\rho_0=\varepsilon$, and $\rho_1=\rho_2=\rho$ shows that $r=\tilde{r}$. Therefore, $\abs{d_N(x,y)-d_{\tN}(\tx,\ty)}\le\varepsilon$, and the claim follows.
    \end{stepproof}

  \item\label{step:HLOS-Riemannian-isometry} The map $\Phi_U\colon U\to\wt{U}$ is a Riemannian isometry. In particular, $\Phi_U$ maps geodesic segments of $N$ whose images are included in $U$ to geodesic segments of $\tN$.

    \begin{stepproof}
      Fix $x\in U$ and $\xi\in T_xU$. Choose $\varepsilon>0$ and a smooth curve $\gamma\colon\ointerval{-\varepsilon}{\varepsilon}\to U$ such that $\gamma(0)=x$ and $\dot{\gamma}(0)=\xi$. Then $\wt{\gamma} := \Phi_U\circ\gamma\colon \ointerval{-\varepsilon}{\varepsilon}\to\wt{U}$ is a smooth curve with $\wt{\gamma}(0)=\Phi_U(x)=:\tx$ and $\dot{\wt{\gamma}}(0)=d\Phi_U(x)\xi=:\wt{\xi}\in T_{\tx}\wt{U}$. Using the standard infinitesimal relation between the Riemannian distance and the metric tensor together with \cref{step:HLOS-dist}, we obtain
      \begin{equation*}
        \pointnorm{\xi}_{g_N}
        = \lim_{t\to 0} \frac{d_N(\gamma(t),x)}{\abs{t}}
        = \lim_{t\to 0} \frac{d_{\tN}(\wt{\gamma}(t),\tx)}{\abs{t}}
        = \pointnorm[\big]{\wt{\xi}}_{g_{\tN}}.
      \end{equation*}
      Thus $d\Phi_U$ preserves norms and, by polarization, $\Phi_U^\ast g_{\tN} = g_{N}$ on $U$. That is, $\Phi_U$ is a Riemannian isometry.

      Because $U$ is open, a curve with image included in $U$ is a geodesic in $N$ if and only if it is a geodesic in $U$; the analogous statement holds for $\wt{U}$ and $\tN$. Riemannian isometries map geodesics to geodesics, so the final assertion follows.
    \end{stepproof}
  \end{step}

  For any point $(x,\xi)\in TN$, let $\gamma_{x,\xi}$ denote the maximal geodesic in $(N,g_N)$ with initial data $\gamma_{x,\xi}(0)=x$ and $\dot{\gamma}_{x,\xi}(0)=\xi$. For $(x,\xi)\in TN$ with $\pointnorm{\xi}_{g_N}=1$, define the cut distance $\tau(x,\xi) :=\sup\{t>0 : d_N(x,\gamma_{x,\xi}(t))=t\}\in\linterval{0}{\infty}$. Define $\wt{\gamma}_{\tx,\wt{\xi}}$ and $\wt{\tau}(\tx,\wt{\xi})$ analogously on $(\tN,g_{\tN})$.

  Even when $x\in U$, the geodesic $\gamma_{x,\xi}$ need not remain in $U$, so $\Phi_U(\gamma_{x,\xi}(t))$ is defined only for those $t$ with $\gamma_{x,\xi}(t)\in U$. Let $I\subset\R$ be the connected component containing $0$ of that set. By \cref{step:HLOS-Riemannian-isometry} and uniqueness of geodesics, $\Phi_U\circ\gamma_{x,\xi}=\wt{\gamma}_{\tx,\wt{\xi}}$ on $I$, where $\tx:=\Phi_U(x)$ and $\wt{\xi}:=d\Phi_U(x)\xi$. Outside $I$ the two geodesics are related only through their matched initial conditions.

  For use in the next two steps, note that if $y\in U$ and $\xi\in T_yN$ is a unit vector, then $\wt{\xi}:=d\Phi_U(y)\xi$ is also a unit vector by \cref{step:HLOS-Riemannian-isometry}.

  \begin{step}[resume]
  \item\label{step:HLOS-cut-distance} For every $y\in U$ and $\xi\in T_yN$ with $\pointnorm{\xi}_{g_N}=1$, we have $\tau(y,\xi)=\wt{\tau}(\ty,\wt{\xi})$, where $\ty:=\Phi_U(y)$ and $\wt{\xi}:=d\Phi_U(y)\xi$.

    \begin{stepproof}
      Fix $y\in U$ and $\xi\in T_yN$ with $\pointnorm{\xi}_{g_N}=1$. Choose $\sigma\in\ointerval{0}{\tau(y,\xi)}$ such that $\gamma_{y,\xi}(\interval{0}{\sigma})\subset U$, and set $x:=\gamma_{y,\xi}(\sigma)$. Then $\sigma=d_N(y,x)$ and $\gamma_{y,\xi}$ restricts to a minimizing geodesic from $y$ to $x$.

      Define $\ty:=\Phi_U(y)$, $\wt{\xi}:=d\Phi_U(y)\xi$, and $\tx:=\Phi_U(x)$. By \cref{step:HLOS-dist,step:HLOS-Riemannian-isometry}, $\sigma=d_{\tN}(\ty,\tx)$ and $\tx=\wt{\gamma}_{\ty,\wt{\xi}}(\sigma)$, hence $\wt{\gamma}_{\ty,\wt{\xi}}$ restricts to a minimizing geodesic from $\ty$ to $\tx$.

      For $\rho>0$, let $P(\rho)$ be the predicate
      \begin{equation*}
        \text{there exists }
        \varepsilon > 0
        \text{ such that }
        B_N(x,\rho+\varepsilon)\subset\overline{B_N(y,\sigma+\rho)}
      \end{equation*}
      (see~\cite[Equation~(31)]{MR3826551}), and let $\wt{P}(\rho)$ be the predicate obtained by replacing $N$, $x$, and $y$ in $P$ by $\tN$, $\tx$, and $\ty$, respectively. By \cref{item:HLOS-2}, $P(\rho)$ holds if and only if $\wt{P}(\rho)$ holds. Consequently,
      \begin{equation*}
        \tau(y,\xi)
        = \sigma + \inf\{\rho>0 : P(\rho)\}
        = \sigma + \inf\{\rho>0 : \wt{P}(\rho)\}
        = \wt{\tau}(\ty,\wt{\xi}).
      \end{equation*}
      The first and last equalities follow from the two implications in \cite[Lemma~11]{MR3826551}; see also~\cite[Equation~(32)]{MR3826551}. Here we use the convention $\inf\emptyset=+\infty$ (the sets in question are empty exactly when $\tau(y,\xi)=+\infty$).
    \end{stepproof}

  \item\label{step:HLOS-distance-functions} Let $y,z\in U$ and $\xi\in T_yN$ with $\pointnorm{\xi}_{g_N}=1$. Fix $t\in\rinterval{0}{\tau(y,\xi)}$ and define $p:=\gamma_{y,\xi}(t)$. Define also $\wt{z}:=\Phi_U(z)$ and $\wt{p}:=\wt{\gamma}_{\ty,\wt{\xi}}(t)$, where $\ty:=\Phi_U(y)$ and $\wt{\xi}:=d\Phi_U(y)\xi$. Then $d_N(p,z)=d_{\tN}(\wt{p},\wt{z})$.

    \begin{stepproof}
      For $t=0$ the claim reduces to \cref{step:HLOS-dist}, so we may assume $t>0$. Choose $\sigma\in\ointerval{0}{t}$ such that $\gamma_{y,\xi}(\interval{0}{\sigma})\subset U$, and define $x:=\gamma_{y,\xi}(\sigma)$ and $\rho:=t-\sigma$.

      For $R>0$, let $Q(R)$ be the predicate
      \begin{equation*}
        \text{there exists } \varepsilon > 0 \text{ such that }
        B_N(x,\rho+\varepsilon)
        \subset
        \overline{B_N(y,\rho+\sigma)\cup B_N(z,R)}
      \end{equation*}
      (see~\cite[Equation~(34)]{MR3826551}). The proof of~\cite[Proposition~2]{MR3826551} establishes that $d_N(p,z) = \inf\{R>0 : Q(R)\}$.

      By \cref{step:HLOS-cut-distance} we have $t<\wt{\tau}(\ty,\wt{\xi})$. Define $\tx:=\wt{\gamma}_{\ty,\wt{\xi}}(\sigma)$, and note that, by \cref{step:HLOS-Riemannian-isometry}, $\wt{\gamma}_{\ty,\wt{\xi}}(\interval{0}{\sigma})=\Phi_U\bigl(\gamma_{y,\xi}(\interval{0}{\sigma})\bigr)\subset\wt{U}$. In particular, $\tx=\Phi_U(x)$. Let $\wt{Q}(R)$ be the predicate obtained by replacing $N$, $x$, $y$, and $z$ in $Q(R)$ by $\tN$, $\tx$, $\ty$, and $\wt{z}$, respectively. Applying the same argument to $\tN$ gives $d_{\tN}(\wt{p},\wt{z}) = \inf\{R>0 : \wt{Q}(R)\}$.

      For every $R>0$, \cref{item:HLOS-2} implies that $Q(R)$ holds if and only if $\wt{Q}(R)$ holds. Hence the two infima coincide, proving the claim.
    \end{stepproof}

  \item\label{step:HLOS-distance-data} We have
    \begin{equation}\label{eq:HLOS-distance-data}
      \{d_N(p,\cdot)|_U : p\in N\}
      = \{d_{\tN}(\wt{p},\Phi_U(\cdot)) : \wt{p}\in\tN\}.
    \end{equation}

    \begin{stepproof}
      Let $p\in N$. By~\cite[Lemma~12]{MR3826551}, there exist $y\in U$, $\xi\in T_yN$ with $\pointnorm{\xi}_{g_N}=1$, and $t\in\rinterval{0}{\tau(y,\xi)}$ such that $p=\gamma_{y,\xi}(t)$. Then, by \cref{step:HLOS-distance-functions}, we have $d_N(p,\cdot)|_U = d_{\tN}(\wt{p},\Phi_U(\cdot))$, where $\ty:=\Phi_U(y)$, $\wt{\xi}:=d\Phi_U(y)\xi$, and $\wt{p}:=\wt{\gamma}_{\ty,\wt{\xi}}(t)\in\tN$. This proves that the left-hand side of \cref{eq:HLOS-distance-data} is included in the right-hand side.

      For the reverse inclusion, let $\wt{p}\in\tN$. By~\cite[Lemma~12]{MR3826551}, there exist $\ty\in\wt{U}$, $\wt{\eta}\in T_{\ty}\tN$ with $\pointnorm{\wt{\eta}}_{g_{\tN}}=1$, and $t\in\rinterval{0}{\wt{\tau}(\ty,\wt{\eta})}$ such that $\wt{p}=\wt{\gamma}_{\ty,\wt{\eta}}(t)$. Define $y:=\Phi_U^{-1}(\ty)\in U$ and $\eta:=(d\Phi_U(y))^{-1}\wt{\eta}\in T_yN$. Then $\pointnorm{\eta}_{g_N}=1$ by \cref{step:HLOS-Riemannian-isometry}. Hence $\wt{\tau}(\ty,\wt{\eta})=\tau(y,\eta)$ by \cref{step:HLOS-cut-distance}, and so $t\in\rinterval{0}{\tau(y,\eta)}$. Applying \cref{step:HLOS-distance-functions} to $y$, $\eta$, and $t$ gives $d_N(p,\cdot)|_U=d_{\tN}(\wt{p},\Phi_U(\cdot))$, where $p:=\gamma_{y,\eta}(t)$. This proves the reverse inclusion.
    \end{stepproof}
  \end{step}

  To prepare for the application of~\cite[Theorem~8]{MR3826551}, choose a nonempty compact set $K\subset U$ whose topological boundary $\partial K$ is a smooth embedded hypersurface and such that $\overline{\interior K}=K$ (for example, a sufficiently small closed geodesic ball with center in $U$). Note that $\partial(\interior K)=\partial K$. Set $\wt{K}:=\Phi_U(K)$. Then $K$ and $\wt{K}$ are compact smooth manifolds with boundary, of the same dimension as $N$ and $\tN$. The restriction $\Phi_K:=\Phi_U|_K\colon K\to\wt{K}$ is a homeomorphism. By \cref{step:HLOS-Riemannian-isometry}, its restriction $\Phi_K|_{\interior K}\colon\interior K\to\interior\wt{K}$ is a Riemannian isometry. Thus $\interior K$ and $\interior\wt{K}$ are bounded open sets with smooth boundary and closures $K$ and $\wt{K}$, respectively, and hence satisfy the hypotheses of that theorem.

  Define the distance representation maps $D\colon N\to C(K)$ by $D(p):=d_N(p,\cdot)|_K$, and $\wt{D}\colon \tN\to C(\wt{K})$ analogously. Endow $C(K)$ and $C(\wt{K})$ with the supremum norms. Define $(\Phi_K)_\ast\colon C(K)\to C(\wt{K})$ by $(\Phi_K)_\ast(f):=f\circ\Phi_K^{-1}$. Then $(\Phi_K)_\ast$ is a linear isometric bijection between these Banach spaces.

  \begin{step}[resume]
  \item\label{step:HLOS-isometry} There is a Riemannian isometry $\Phi\colon N\to\tN$ such that $\Phi|_U=\Phi_U$.

    \begin{stepproof}
      Since $K\subset U$, \cref{step:HLOS-distance-data} shows that $(\Phi_K)_\ast|_{D(N)}\colon D(N) \to \wt{D}(\tN)$ is a bijective isometry of metric spaces.

      Consider the maps $D\colon (N,g_N) \to D(N)$ and $\wt{D}\colon (\tN,g_{\tN}) \to \wt{D}(\tN)$, which are homeomorphisms by~\cite[Proposition~3]{MR3826551}. It is shown in~\cite[Proposition~4]{MR3826551} that smooth structures can be constructed on the topological manifolds $D(N)$ and $\wt{D}(\tN)$ from the distance data in~\cite[Equation~(35)]{MR3826551}, and that with respect to these smooth structures, the maps $D$ and $\wt{D}$ are diffeomorphisms. Tracing through the construction and using \cref{step:HLOS-Riemannian-isometry}, one verifies that the map $(\Phi_K)_\ast|_{D(N)}$ is a diffeomorphism.

      Similarly,~\cite[Proposition~5]{MR3826551} shows that Riemannian metrics can be constructed on $D(N)$ and $\wt{D}(\tN)$ for which $D$ and $\wt{D}$ are Riemannian isometries. The same reasoning, using \cref{step:HLOS-Riemannian-isometry}, shows that $(\Phi_K)_\ast|_{D(N)}$ is a Riemannian isometry.

      Consequently, the map $\Phi:=\wt{D}^{-1}\circ(\Phi_K)_\ast\circ D\colon (N,g_N)\to (\tN,g_{\tN})$ is a Riemannian isometry, so it is enough to show that it restricts to $\Phi_U$ on $U$.

      Let $x\in U$ and $\wt{z}\in\wt{K}$, and set $z:=\Phi_K^{-1}(\wt{z})\in K\subset U$. Then, by \cref{step:HLOS-dist},
      \begin{equation*}
        (\Phi_K)_\ast(D(x))(\wt{z})
        = d_N(x,z)
        = d_{\tN}(\Phi_U(x),\wt{z})
        = \wt{D}(\Phi_U(x))(\wt{z}).
      \end{equation*}
      Therefore $\wt{D}(\Phi(x))=\wt{D}(\Phi_U(x))$, and injectivity of $\wt{D}$ gives $\Phi(x)=\Phi_U(x)$. As $x\in U$ was arbitrary, $\Phi|_U=\Phi_U$.
    \end{stepproof}
  \end{step}

  Thus \cref{step:HLOS-isometry} establishes \cref{item:HLOS-1}, which completes the proof of the reverse implication.
\end{proof}